\documentclass[11pt]{article}
\usepackage[margin=0.85in]{geometry}
\usepackage[utf8]{inputenc}
\usepackage{amsmath,amsthm,amssymb,enumerate, graphicx,subcaption}
\usepackage{setspace}
\usepackage{thmtools}
\usepackage{thm-restate}
\usepackage{cleveref,color}
\usepackage{natbib}
\usepackage[mathscr]{euscript}
\usepackage[shortlabels]{enumitem}
\usepackage{prodint}

\newif\ifproof
\prooffalse

\declaretheorem[name=Theorem]{thm}

\declaretheorem[name=Lemma]{lemma}

\DeclareMathOperator*{\inprob}{\stackrel{P}{\longrightarrow}}
\DeclareMathOperator*{\inproblow}{\rightarrow_{P}}

\DeclareMathOperator*{\indist}{\stackrel{d}{\longrightarrow}}
\newcommand{\bounded}{O_{\mathrm{P}}}
\newcommand{\boundeddet}{O}
\newcommand{\fasterthan}{o_{\mathrm{P}}}
\newcommand{\fasterthandet}{o}

\DeclareMathOperator*{\argmin}{argmin}
\DeclareMathOperator*{\argmax}{argmax}

\newcommand{\s}[1]{\mathscr{#1}}
\renewcommand{\d}[1]{\mathbb{#1}}

\newcommand{\n}[1]{\mathrm{#1}}

\usepackage{natbib}

\allowdisplaybreaks

\title{A unified study of nonparametric inference\\ for monotone functions}
\author{Ted\ Westling\\ Center for Causal Inference \\ University of Pennsylvania\\ tgwest@pennmedicine.upenn.edu \and Marco Carone\\ Department of Biostatistics \\ University of Washington \\ mcarone@uw.edu\\}
\date{}
\begin{document}

\maketitle

\begin{abstract}
The problem of nonparametric inference on a monotone function has been extensively studied in many particular cases. Estimators considered have often been of so-called Grenander type, being representable as the left derivative of the greatest convex minorant or least concave majorant of an estimator of a primitive function. In this paper, we provide general conditions for consistency and pointwise convergence in distribution of a class of generalized Grenander-type estimators of a monotone function. This broad class allows the minorization or majoratization operation to be performed on a data-dependent transformation of the domain, possibly yielding benefits in practice. Additionally, we provide simpler conditions and more concrete distributional theory in the important case that the primitive estimator and data-dependent transformation function are asymptotically linear.  We use our general results in the context of various well-studied problems, and show that we readily recover classical results established separately in each case. More importantly, we show that our results allow us to tackle more challenging problems involving parameters for which the use of flexible learning strategies appears necessary. In particular, we study inference on monotone density and hazard functions using informatively right-censored data, extending the classical work on independent censoring, and on a covariate-marginalized conditional mean function, extending the classical work on monotone regression functions. In addition to a theoretical study, we present numerical evidence supporting our large-sample results.
\end{abstract}

\section{Introduction}

\doublespacing

\subsection{Background}

In many scientific settings, investigators are interested in learning about a function known to be monotone, either due to probabilistic constraints or in view of existing scientific knowledge. The statistical treatment of nonparametric monotone function estimation has a long and rich history. Early on, \cite{grenander1956theory} derived the nonparametric maximum likelihood estimator (NPMLE) of a monotone density function, now commonly referred to as the Grenander estimator. Since then, monotone estimators of many other parameters, including hazard and regression functions, have been proposed and studied.

In the literature, most monotone function estimators have been constructed via empirical risk minimization. Specifically, these are obtained by minimizing the empirical risk over the space of non-decreasing, or non-increasing, candidate functions  based on an appropriate loss function. The theoretical study of these estimators has often hinged strongly on their characterization as empirical risk minimizers. This is the case, for example, for the asymptotic theory developed by \cite{rao1969density} and \cite{rao1970hazard} for the NPMLE  of monotone density and hazard functions, respectively, and by  \cite{brunk1970regression} for the least-squares estimator of a monotone regression function. \cite{kim1990cube} unified the study of these various estimators by studying the \emph{argmin} process typically driving the pointwise distributional theory of  monotone empirical risk minimizers.

Many of the parameters treated in the literature on monotone function estimation can be viewed as an index of the statistical model, in the sense that the model space is in bijection with the product space corresponding to the parameter of interest and an additional variation-independent parameter. In such cases, identifying an appropriate loss function is often easy, and a risk minimization representation is therefore usually available. However, when the parameter of interest is a complex functional of the data-generating mechanism, an appropriate loss function may not be readily available. This occurs often, for example, when identification of the parameter of interest based on the observed data distribution requires adjustment for sampling complications (e.g., informative treatment attribution, missing data or loss to follow-up). It is thus imperative to develop and study estimation methods that do not rely upon risk minimization.

It is a simple fact that the primitive of a non-decreasing function is convex. This observation serves as motivation to consider as an estimator of the function of interest the derivative of the greatest convex minorant (GCM) of an estimator of its primitive function. In the literature on monotone function estimation, many estimators obtained as empirical risk minimizers can alternatively be represented as the left derivative of the GCM of some primitive estimator. This is because the definition of the GCM is intimately tied to the necessary and sufficient conditions for optimization of certain risk functionals over the convex cone of monotone functions (see, e.g., Chapter 2 of \citealp{groene2014shape}). In particular, Grenander's NPMLE of a monotone density equals the left derivative of the GCM of the empirical distribution function. In the recent literature, estimators obtained in this fashion have thus been referred to as being of \emph{Grenander-type}. \cite{leurgans1982} is an early example of a general study of Grenander-type estimators for a class of regression problems.

In a seminal paper, \cite{groeneboom1985density} introduced an approach to studying GCMs based on an inversion operation. This approach has facilitated the theoretical study of certain Grenander-type estimators without the need to utilize their representation as empirical risk minimizers. For example, under the assumption of independent right-censoring, \cite{huang1995right} used this approach to derive large-sample properties of a monotone hazard function estimator obtained by differentiating the GCM of the Nelson-Aalen estimator of the cumulative hazard function. This general strategy was also used by \cite{van2006estimating}, who derived and studied an estimator of a covariate-marginalized survival curve based on current-status data, including possibly high-dimensional and time-varying covariates. More recently, there has been interest in deriving general results for Grenander-type estimators applicable to a variety of cases. For instance, \cite{anevski2006general} derived pointwise distributional limit results for Grenander-type estimators in a very general setting including, in particular, dependent data. \cite{durot2007}, \cite{durot2012} and \cite{lopuhaa2016hellinger} derived limit results for the estimation error of Grenander-type estimators under $L_p$, supremum and Hellinger norms, respectively. \cite{durot2013testing} studied the problem of testing the equality of generic monotone functions with $K$ independent samples. \cite{durot2014kiefer}, \cite{beare2017weak} and \cite{lopuhaa2018naive} studied properties of the least concave majorant of an arbitrary estimator of the primitive function of a monotone parameter. The monograph of \cite{groene2014shape} also summarizes certain large-sample properties for these estimators.

\subsection{Contribution and organization of the article}

In this paper, we wish to address the following three key objectives:
\begin{enumerate}
\item to provide a unified framework for studying a large class of nonparametric monotone function estimators that implies classical results but also applies in more complicated, modern applications;\vspace{-.075in}
\item to derive tractable sufficient conditions under which estimators in this class are known to be consistent and have a non-degenerate limit distribution upon proper centering and scaling;\vspace{-.075in}
\item to illustrate the use of this general framework to construct targeted estimators of monotone parameters that are possibly complex summaries of the observed data distribution, and whose estimation may require the use of data-adaptive estimators of nuisance functions.
\end{enumerate}

Our first goal is to introduce a class of monotone estimators that allow the greatest convex minorization process to be performed on a possibly data-dependent transformation of the domain. For many monotone estimators in the literature, the greatest convex minorization is performed on a transformation of the domain. A strategic domain transformation can lead to significant benefits in practice, including in some cases the elimination of the need to estimate challenging nuisance parameters. Unfortunately, to our knowledge, existing results for general Grenander-type estimators do not apply in a straightforward manner in cases in which a data-dependent transformation of the domain has been used. We will define a class that permits such transformations, and demonstrate both how this class encompasses many existing estimators in the literature and how a transformation can be strategically selected in novel problems.

Our second goal is to derive sufficient conditions on the estimator of the primitive function and domain transformation that imply consistency and pointwise convergence in distribution of the monotone function estimator. As noted above, general results on pointwise convergence in distribution for the class of Grenander-type estimators, applicable in a wide variety of settings, were provided in \cite{anevski2006general}. Our work differs from that of \cite{anevski2006general} in a few important ways. First, the role and implications of domain transformations -- which, as we show, are often important in practice -- were not explicitly considered in \cite{anevski2006general}. To our knowledge, the class of generalized Grenander-type estimators we consider in this paper, which allow for domain transformations, has not previously been studied in a unified manner, and hence, general results for this class do not currently exist.  Second, in addition to pointwise distributional results, we study weak consistency. Third, in Sections~\ref{sec:improved},~\ref{sec:examples} and~\ref{sec:sims}, we pay special attention to parameters for which asymptotically linear estimators of the primitive and transformation functions can be constructed -- in such cases, relatively straightforward sufficient conditions can be developed, and the limit distribution has a simpler form. While these results are weaker than those in Section~\ref{sec:general} and in \cite{anevski2006general} because they apply only to a special case, they are useful in many settings. We demonstrate the utility of these results for three groups of examples -- estimation of monotone density, hazard and regression functions -- and show that our results coincide with established results in these settings.

Our third goal is to discuss and illustrate Grenander-type estimation in cases in which nonparametric estimation of the  primitive function requires estimation of challenging nuisance parameters. In this sense, our work follows the lead of \cite{van2006estimating}, whose setting is of this type. More generally, such primitive functions arise frequently, for example, when the observed data unit represents a coarsened version of an ideal data structure, and the coarsening occurs randomly conditional on observed covariates \citep{heitjan1991car}. In our general results, we provide sufficient conditions that can be readily applied to such primitive estimators. To demonstrate the application of our theory in coarsened data structures, we consider extensions of the three classical monotone problems  above to more complex settings in which covariates must be accounted for, because either the censoring process or the treatment allocation mechanism are informative, as is typical in observational studies. Specifically, we derive novel estimators of monotone density and hazard functions for use when the survival data are subject to right-censoring that may depend on covariates, and a novel estimator of a monotone dose-response curve for use when the relationship between the exposure and outcome is confounded by recorded covariates. Unlike for their classical analogues, in these more difficult problems, nonparametric estimation of the primitive function involves nuisance functions for which flexible estimation strategies (e.g., machine learning) must be employed. As \cite{van2006estimating} was able to achieve in a particular problem, our general framework explicitly allows the integration of such strategies while still yielding estimators with a tractable limit theory. 
 


Our paper is organized as follows. In Section \ref{sec:grenander}, we define the class of estimators we consider and briefly introduce our three working examples. In Section~\ref{sec:general}, we present our most general results for the consistency and convergence in distribution of our class of estimators. We provide refined results, including simpler sufficient conditions and distributional results, for the special case in which the primitive and transformation estimators are asymptotically linear in Section \ref{sec:improved}. In Section \ref{sec:examples}, we apply our general theory in three examples, both for classical parameters and for the novel extensions we consider. In Section \ref{sec:sims}, we provide results from simulation studies that evaluate the validity of the theory in two examples. We  provide concluding remarks in Section \ref{sec:discussion}. The proofs of all theorems are provided in Supplementary Material. Additional technical details are found in Supplementary Material.

\section{Generalized Grenander-type estimators}\label{sec:grenander}

\subsection{Statistical setup and definitions}

Throughout, we make use of the following definitions. For intervals $I,J \subseteq R$, define $\ell^{\infty}(I)$ as the space of bounded, real-valued functions on $I$, $\mathscr{D}_I \subset \ell^{\infty}(I)$ as the subset of non-decreasing and c\`{a}dl\`{a}g (right-continuous with left-hand limits) functions on $I$, and $\mathscr{D}_{I,J} \subset \mathscr{D}_I$ as the further subset of functions whose range is contained in $J$. The GCM operator $\n{GCM}_{I} : \ell^{\infty}(I) \to \ell^{\infty}(I)$ is defined for any $G \in \ell^{\infty}(I)$ as the pointwise supremum over all convex functions $H \leq G$ on $I$. We note that $\n{GCM}_{I}(G)$ is necessarily convex. For $G \in \mathscr{D}_I$, we denote by $G^{-}$ the generalized inverse mapping $x\mapsto\inf\{u \in I: G(u) \geq x\}$, and for a left-differentiable $G$, we denote by $\partial_-G$ the left derivative of $G$.

We are interested in making inference about an unknown function $\theta_0\in \mathscr{D}_I$ determined by the true data-generating mechanism $P_0$ for an interval $I \subseteq \d{R}$. We denote the endpoints of $I$ by $a_I := \inf I$ and $b_I := \sup I$. We define the primitive function $\Theta_0$ of $\theta_0$ pointwise for each $x\in I$ as $\Theta_0(x):=\int_{a_I}^x \theta_0(u)du$, where if $a_I = -\infty$ we assume the integral exists. The general results we present in Section~\ref{sec:general} apply in contexts with either independent or dependent data. Starting in Section~\ref{sec:improved}, we focus on problems in which the data consist of independent observations $O_1,O_2, \dotsc, O_n$ from an unknown distribution $P_0$ contained in a nonparametric model $\mathscr{M}$. In such cases, we denote by $O$ a prototypical data unit, by $\mathscr{O}(P)$ the support of $O$ under $P\in\mathscr{M}$, and we set $\mathscr{O}:=\cup_{P\in\mathscr{M}}\mathscr{O}(P)$.

In its simplest formulation, a Grenander-type estimator of $\theta_0$ is given by $\partial_-\n{GCM}_I(\Theta_n)$ for some estimator $\Theta_n$ of $\Theta_0$. However, as a critical step in unifying classical estimators and constructing procedures with possibly improved properties, we wish to allow the GCM procedure to be performed on a possibly data-dependent transformation of the domain $I$. To do so, we first define for any interval $J \subseteq \d{R}$ the operator $\n{Iso}_{J}:\ell^{\infty}(J) \times \mathscr{D}_{I,J}\rightarrow \ell^\infty(I)$ as $\n{Iso}_{J}(\Psi, \Phi) := ( \partial_{-}\n{GCM}_{J}(\Psi)) \circ \Phi$ for each $\Psi\in\ell^\infty(I)$ and $\Phi\in\mathscr{D}_{I,J}$. We set $J_0 := [0,u_0]$, with $u_0\in(0,\infty)$ possibly depending on $P_0$, and suppose that a domain transform $\Phi_0\in \mathscr{D}_{I,J_0}$ is chosen. We may then consider the domain-transformed parameter $\psi_0:= \theta_0 \circ \Phi_0^{-}$, which has primitive $\Psi_0$ defined pointwise as $\Psi_0(t):= \int_0^t \psi_0(u)du $ for $t\in(0,u_0]$. As with $\theta_0$ and $\Theta_0$, $\psi_0$ is non-decreasing and $\Psi_0$ is convex. Thus, it must be true that $\n{Iso}_{J_0}(\Psi_0, \Phi_0)(x) = \theta_0(x)$ for each $x\in I$ at which $\theta_0$ is left-continuous and such that $\Phi_0(u) < \Phi_0(x)$ for all $u < x$. This observation motivates us to consider estimators of $\theta_0$ of the form $\n{Iso}_{J_n}(\Psi_n, \Phi_n)$, where $\Psi_n$, $\Phi_n$ and $u_n$ are estimators of $\Psi_0$, $\Phi_0$ and $u_0$, respectively, and we define $J_n := [0,u_n]$. We refer to any such estimator as being of the \emph{generalized Grenander-type}. This class, of course, contains the standard Grenander-type estimators: setting $\Psi_n = \Theta_n$ and $\Phi_n = \n{Id}$ for $\n{Id}$ the identity mapping yields $\theta_n = \partial_-\n{GCM}_I(\Theta_n)$. We note that, in this formulation,   we require the domain $J_0$ over which the GCM is performed to be bounded, but not so for the domain $I$ of $\theta_0$. Additionally, we assume that the left endpoint of $J_0$ is fixed at 0, while the upper endpoint $u_0$ may depend on $P_0$. However, this entails no loss in generality, since if the desired domain is instead $[\ell_0, u_0]$, where now $\ell_0$ also depends on $P_0$, we can define $\bar{u}_0 := u_0 - \ell_0$ and similarly shift $\Phi_0$ by $\ell_0$ to obtain the new domain $[0, \bar{u}_0]$.

Defining $\Gamma_0:=\Psi_0\circ  \Phi_0$, we suppose that we have at our disposal estimators $\Phi_n$ and $\Gamma_n$ of $\Phi_0$ and $\Gamma_0$, respectively, as well as a weakly consistent estimator $u_n$ of $u_0$. In this work, we study the properties of a generic generalized Grenander-type estimator $\theta_n$ of $\theta_0$ of the form 
\begin{equation} \theta_n :=  \n{Iso}_{J_n}(\Gamma_n \circ \Phi_n^{-}, \Phi_n)\ .\label{estimator}\end{equation}
Specifically, our goal is to provide sufficient conditions on the triple $(\Gamma_n, \Phi_n, u_n)$ under which $\theta_n$ is consistent, and under which a suitable standardization of $\theta_n$ converges in distribution to a nondegenerate limit. As stated above, our only requirement for $u_n$ is that it tend in probability to $u_0$. Therefore, our focus will be on the pair $(\Gamma_n, \Phi_n)$.

We note that estimators taking form $\eqref{estimator}$ constitute a more restrictive class than the set of all estimators of the form $\n{Iso}_{J_n}(\Psi_n, \Phi_n)$ for arbitrary $\Psi_n$. Our focus on this slightly less general form is motivated by two reasons.  First, as we will see in various examples, $\Gamma_0$ often has a simpler form than $\Psi_0$, and in such cases, it may be significantly easier to verify required regularity conditions for  $\Gamma_n$ and to derive limit distribution properties based on $\Gamma_n$ rather than $\Psi_n$. Second, many celebrated monotone estimators in the literature follow this particular form. This can be seen by noting that, if $\Phi_n$ is a right-continuous step function with jumps at points $x_1,x_2, \dotsc, x_m$, then for each $x \in I$ the estimator $\theta_n(x)$ given in \eqref{estimator} equals  the slope at $\Phi_n(x)$ of the greatest convex minorant of the diagram of points $\{(\Phi_n(x_j), \Gamma_n(x_j)) : j = 0, 1, \dotsc, m\}$, where $x_0 = a_I$. We highlight well-known examples of estimators of this type below. In brief, we sacrifice a little generality for a substantial gain in the ease of application of our results, both for well-known and novel monotone estimators. Nevertheless, conditions on the pair $(\Psi_n, \Phi_n)$ under which consistency and distributional results hold for $\theta_n$ can be derived similarly.

\subsection{Examples}

Before proceeding to our main results, we briefly discuss the several examples we will use to illustrate how our framework allows us to not only obtain results on classical estimators in the monotone estimation literature directly, but also tackle more complex problems for which no estimators are currently available. These examples will be studied extensively in Section \ref{sec:examples}.

\subsubsection*{Example 1: monotone density function}

Suppose that $T$ is a univariate positive random variable with non-decreasing density function $f_0$, and that $T$ is right-censored by an independent random censoring time $C$. The observed data unit is $O:= (Y, \Delta)$, where $Y:= \min(T, C)$ and $\Delta:= I(T\leq C)$, with distribution $P_0$ implied by the true marginal distributions of $T$ and $C$. The parameter of interest is $\theta_0:=f_0$, the density function of $T$ with support $I$. Taking $\Phi_0$ to be the identity function, we get that $\psi_0 = \theta_0$. Here, both $\Psi_0$ and $\Theta_0$ represent the distribution function $F_0$ of $T$, and $\Phi_0$ plays no role. A natural estimator $\theta_n$ of $\theta_0$ can be obtained by taking $\Psi_n$ to be the Kaplan-Meier estimator of the distribution function $\Psi_0$. With $\Phi_n$ the identity map, $\Gamma_n:=\Psi_n$ and $u_n := \max_i Y_i$, the estimator $\theta_n := \n{Iso}_{J_n}(\Gamma_n, \Phi_n)$ is precisely the estimator studied by \cite{huang1995right}. When $C = +\infty$ with probability one, $\Psi_n$ is the empirical distribution function based on $Y_1, Y_2, \dotsc, Y_n$, and $\theta_n$ is precisely the Grenander estimator, the NPMLE of $\theta_0$.

In Section~\ref{sec:examples}, we extend estimation of a monotone density function to the setting in which the data are subject to possibly informative right-censoring. Specifically, we only require $T$ and $C$ to be independent conditionally upon a vector $W$ of baseline covariates. We will study the estimator defined by differentiating the GCM of a one-step estimator of $\Psi_0$. In this context, estimation of $\Psi_0$ requires estimation of nuisance functions. We will use our  general results to provide conditions on the nuisance estimators that imply consistency and distributional results for $\theta_n$.

\subsubsection*{Example 2: monotone hazard function}

\noindent Suppose now that $T$ is a univariate positive random variable with non-decreasing hazard function $\lambda_0$.  In this example, we are interested in $\theta_0:=\lambda_0$. Setting $S_0:=1-F_0$ to be the survival function of $T$, we note that $\Gamma_0(u)= \int_0^u  f_0(v) / S_0(v)\Phi_0(dv)$, and so, taking $\Phi_0$ to satisfy $\Phi_0(dv) = S_0(v)dv$ makes $\Gamma_0=F_0$. The restricted mean lifetime function $\Phi_0(u) := \int_0^u S_0(v)dv$ satisfies this condition. Using this transformation, the estimator of the monotone hazard function $\theta_0$ only requires estimation of $F_0$.

In Section \ref{sec:examples}, we again extend estimation of a monotone hazard function to allow the data to be subject to possibly informative right-censoring using the same one-step estimator $\Gamma_n$ of $\Gamma_0=F_0$ that will be introduced in Example 1 and the data-dependent transformation $\Phi_n(u) := \int_0^u [1 - \Gamma_n(v)]dv$. We will show that, once the simpler details regarding the estimation of a monotone density are established,  the asymptotic properties of this estimator of a monotone hazard are obtained essentially for free.

\subsubsection*{Example 3: monotone regression function}

As our last example, we study estimation of a non-decreasing regression function. In the simplest setup, the data unit is $O:=(Y, A)$ and we are interested in $\theta_0: x\mapsto E_0 \left(Y \mid A = x\right)$. Assume without loss of generality that the data are sorted according to the observed values of $A$.  Taking $I$ to be the support of $A$ and $\Phi_0$ to be the marginal distribution function of $A$,  we have that $\psi_0(u)=E_0\left[Y \mid \Phi_0(A) = u\right]$ for each $u\in [0,1]$, and $\Gamma_0(x)=E_0\left[Y I_{(-\infty,x]}(A)\right]$ for each $x \in I$. Thus, $\Gamma_n(x) :=\frac{1}{n}\sum_{i=1}^n Y_i  I_{(-\infty,x]}(A_i)$ and $\Phi_n(x) := \tfrac{1}{n}\sum_{i=1}^n I_{(-\infty, x]}(A_i)$ are natural nonparametric estimators of $\Gamma_0(x)$ and $\Phi_0(x)$, respectively. Then, $\theta_n := \n{Iso}_{[0,1]}(\Gamma_n, \Phi_n)$ is the classical monotone least-squares estimator of $\theta_0$.

In Section \ref{sec:examples}, we consider an extension to estimation of a covariate-marginalized regression function, for use when the relationship  between exposure and outcome of interest is confounded. Specifically, we will consider the data unit $O:=(Y, A, W)$, with $W$ representing a vector of potential confounders, and focus on $\theta_0: x\mapsto E_0 \left[E_0\left(Y \mid A = x, W\right)\right]$. Under untestable causal identifiability conditions, $\theta_0(x)$ is the mean of the counterfactual outcome $Y(x)$ obtained by setting exposure at level $A = x$. This parameter plays a critical role in causal inference, particularly when the available data are obtained from an observational study and the exposure assignment process may be informative. As before, tackling this more complex parameter will require estimation of certain nuisance functions. 


\section{General results}\label{sec:general}

We begin with our first set of results on the large-sample properties of $\theta_n$. Our goal is to establish conditions under which consistency and pointwise convergence in distribution hold. First, we provide general results on the consistency  of $\theta_n$, both pointwise and uniformly. We note that the results of \cite{anevski2006general}, \cite{durot2007},  \cite{durot2012} and \cite{lopuhaa2016hellinger} imply conditions for consistency of Grenander-type estimators.  However, because the objective of their work is to establish distributional theory for a global discrepancy between the estimated and true monotone function, the conditions they require are stronger than needed for consistency alone. Also, their work is restricted to Grenander-type estimators, without data-dependent transformations of the domain.

Below, we refer to the sets $I_n := \{z\in I : z = \Phi_n^-(u), u \in J_n\}$ and $I_{n,\beta} := \{ x \in I : 0 \leq \Phi_0(x - \beta)\leq \Phi_0(x + \beta) \leq u_n\}$ for $\beta\geq0$.

\begin{thm}[Weak consistency]\label{thm:consistency}


\begin{enumerate}[(1),leftmargin=*]

\item Suppose $\theta_0$ is continuous at $x\in I$ and, for some $\delta>0$ such that $[x-\delta,x+\delta]\subset \Phi_0^{-1}(J_0)$, $\Phi_0$ is strictly increasing and continuous on $[x-\delta,x+\delta]$.  If $\|\Gamma_n - \Gamma_0\|_{\infty, I_n}$, $\| \Phi_n - \Phi_0\|_{\infty, I_n}$ and $\| \Phi_n - \Phi_0\|_{\infty, [x-\delta, x + \delta]}$ tend to zero in probability, then $\theta_n(x) \inprob \theta_0(x)$.

\item Suppose $\theta_0$ and $\Phi_0$ are uniformly continuous on $I$, and $\Phi_0$ is strictly increasing on $I$. If $\|\Gamma_n - \Gamma_0\|_{\infty, I_n}$ and $\| \Phi_n - \Phi_0\|_{\infty, I}$ tend to zero in probability, then $\| \theta_n - \theta_0\|_{\infty, I_{n,\beta}} \inprob 0$ for each fixed $\beta > 0$.
\end{enumerate}
\end{thm}

We note that in part 1 of Theorem \ref{thm:consistency}, we require uniform convergence of $\Gamma_n$ and $\Phi_n$ to obtain a pointwise result for $\theta_n$ -- this will also be the case for Theorem \ref{thm:rates} below. This is because the GCM is a global procedure, and so, the value of $\theta_n(x_1)$ depends on $\Gamma_n(x_2)$ even for $x_2$ not near $x_1$. Without uniform consistency of $\Gamma_n$,  $\theta_n$ may indeed fail to be pointwise consistent. Also, we note that in part 1 of Theorem \ref{thm:consistency}, we require that $\Gamma_n - \Gamma_0$ and $\Phi_n - \Phi_0$ tend to zero uniformly over the set $I_n$. This requirement stems from the fact that $\theta_n$ only depends on $\Gamma_n$ through the composition $\Gamma_n \circ \Phi_n^-$, and so, values of $\Gamma_n$ only matter at points in the range of $\Phi_n^-$. In part 1, we also require that $\Phi_n - \Phi_0$ tend to zero uniformly in a neighborhood of $x$, while in part 2, we require that $\Phi_n - \Phi_0$ tend to zero uniformly over $I$. These requirements allow us to obtain results for $x$ values that are possibly outside $I_n$ for all $n$. In many applications, it may be the case that $\Gamma_n - \Gamma_0$ and $\Phi_n - \Phi_0$ both tend to zero in probability uniformly over $I$, which implies convergence over $I_n$.

The weak conditions required for Theorem~\ref{thm:consistency} are especially important for the extensions of the classical parameters that we consider in Section~\ref{sec:examples}. The estimators we propose often require estimating difficult nuisance parameters, such as conditional hazard, density and mean functions. While under mild conditions it is typically possible to construct uniformly consistent estimators of these nuisance parameters, ensuring a given local or uniform rate of convergence often requires additional knowledge about the true function. Thus, Theorem~\ref{thm:consistency} is useful for guaranteeing consistency under weak conditions.

We now provide lower bounds on the convergence rate of $\theta_n$, both pointwise and uniformly, depending on (a) the uniform rates of convergence of $\Gamma_n$ and $\Phi_n$, and (b) the moduli of continuity of $\theta_0$ and $\Phi_0^-$.


\begin{thm}[Rates of convergence]\label{thm:rates}
Let $x\in I$ be given. Suppose that, for some $\delta>0$, $[x-\delta,x+\delta]\subset \Phi_0^{-1}(J_0)$ and $\Phi_0$ is strictly increasing and continuous on $[x-\delta,x+\delta]$. Let $r_n$ be a fixed sequence such that $r_n\| \Gamma_n - \Gamma_0\|_{\infty, I_n}$, $r_n\|\Phi_n - \Phi_0\|_{\infty, I_n}$ and $r_n\|\Phi_n - \Phi_0\|_{\infty, [x-\delta, x + \delta]}$ are bounded in probability.  
\begin{enumerate}[(1),leftmargin=*]
\item If there exist $K_1(x), K_2(x)\in[0,\infty)$ and $\alpha_1, \alpha_2\in (0,1]$ such that $|\theta_0(u) - \theta_0(x)| \leq K_1(x)|u-x|^{\alpha_1}$ for all $u \in I$ and $|\Phi_0^-(u) - \Phi^-(x)| \leq K_2(x) | u -x|^{\alpha_2}$ for all $u \in J_0$, then \[r_n^{\frac{\alpha_1\alpha_2}{1+\alpha_1\alpha_2}}\left[\theta_n(x) - \theta_0(x)\right] = \bounded(1)\ .\]

\item If $\theta_0$ is constant on $[x - \delta, x + \delta]$, then $r_n\left[\theta_n(x) - \theta_0(x)\right]=\bounded(1)$.
\end{enumerate}
Let $r_n$ be a fixed sequence such that $r_n\| \Gamma_n - \Gamma_0\|_{\infty, I_n}$ and $r_n\|\Phi_n - \Phi_0\|_{\infty, I}$ are bounded in probability, and suppose that $\Phi_0$ is strictly increasing on $I$.
\begin{enumerate}[(1),leftmargin=*]
\setcounter{enumi}{2}
\item If there exist $K_1, K_2 \in[0,\infty)$ and $\alpha_1, \alpha_2\in (0,1]$ such that $|\theta_0(u) - \theta_0(v)| \leq K_1|u-v|^{\alpha_1}$ for all $u,v \in I$ and $|\Phi_0^{-}(u) - \Phi_0^{-}(v)| \leq K_2 | u -v|^{\alpha_2}$ for all $u,v \in J_0$, then
\[r_n^{\frac{\alpha_1\alpha_2}{1+\alpha_1\alpha_2}}\|\theta_n - \theta_0\|_{\infty, I_{n,\beta_n}} = \bounded(1)\] for any (possibly random) positive real sequence $\beta_n$ such that $\beta_n r_n^{1/(1+\alpha_1\alpha_2)} \inprob \infty$.
\end{enumerate}
\end{thm}

We note here that the uniform results only cover subintervals of the interval over which the GCM procedure is performed. This should not be surprising given the poor behavior of Grenander-type estimators at the boundary of the GCM interval, as discussed, for example, in \cite{woodroofe1993penalized}, \cite{kulikov2006} and \cite{balabdaoui2011grenander}. Various boundary corrections have been proposed -- applying these in our general framework is an interesting avenue for future work.

We also note that, in Theorem~\ref{thm:rates}, when $\theta_0$ and $\Phi_0$ are locally or globally Lipschitz, then $\alpha_1 = \alpha_2 =1$ and the resulting rate is $\bounded(r_n^{-1/2})$, which yields $\bounded(n^{-1/4})$ when $r_n =n^{1/2}$. This rate is slower than the rate $n^{-1/3}$ that is often achievable for pointwise convergence when $\theta_0$ and $\Phi_0$ are differentiable at $x$ and the primitive estimator converges at rate $n^{-1/2}$, as we discuss below. However, the assumptions in Theorems \ref{thm:rates} are significantly weaker than typically required for the $n^{-1/3}$ rate of convergence: they constrain the supremum norm of the estimation error rather than its modulus of continuity, and hold when the true function is Lipschitz but not differentiable. Our results also cover situations in which $\theta_0$ or $\Phi_0$ are in H\"{o}lder classes. The rates provided by Theorem \ref{thm:rates} should thus be seen as lower bounds on the true rate, for use when less is known about the properties of the estimation error or of the true functions. The distributional results we provide below recover the usual rates under stronger conditions.


For a fixed sequence $r_n$ of positive real numbers, we now study the pointwise convergence in distribution of $r_n\left[\theta_n(x) - \theta_0(x)\right]$ at an interior point $x\in I$ at which $\Phi_0$ has a strictly positive derivative. The rate $r_n$ depends on two interdependent factors. First, we suppose that there exists some $\alpha > 0$ such that $|\theta_0(x + u) - \theta_0(x)| = \pi_0(x)|u|^\alpha + \fasterthandet(1)$ as $u \to 0$ for some constant $\pi_0(x)>0$. Second, writing $\Gamma_{n,0}:=\Gamma_n-\Gamma_0$ and $\Phi_{n,0}:=\Phi_n-\Phi_0$, we suppose that there exists a sequence of positive real numbers $c_n \to \infty$ such that the appropriately localized process 
\[W_{n,x}:u\mapsto c_n^{\alpha + 1}\left\{ \Gamma_{n,0}( x + u c_n^{-1}) -  \Gamma_{n,0}(x) -  \theta_0(x)\left[ \Phi_{n,0}( x + uc_n^{-1}) -  \Phi_{n,0}(x)\right]\right\}\] 
converges weakly. We note that $W_{n,x}$ depends on $\alpha$. As we formalize below, if $r_n=c_n^\alpha$, then $r_n\left[\theta_n(x) - \theta_0(x)\right]$ has a nondegenerate limit distribution under some conditions. We now introduce some of the conditions that we build upon:
\begin{description}[style=multiline,leftmargin=1cm]
\item[(A1)] for each $M>0$, $\left\{W_{n,x}(u):|u|\leq M\right\}$ converges weakly in $\ell^{\infty}[-M,M]$ to a tight limit process $\left\{W_{x}(u):|u|\leq M\right\}$ with almost surely lower semi-continuous sample paths;
\item[(A2)] for every $c \in \d{R}$, $\displaystyle\sup\argmax_{u \in \d{R}}\left\{W_{x}(u) +  \pi_0(x) \Phi_0'(x)(\alpha + 1)^{-1} |u|^{\alpha + 1} + c \Phi_0'(x) u\right\} = \bounded(1)$;
\item[(A3)] there exist $\beta \in (1,1 + \alpha)$, $\delta^* > 0$ and a sequence $f_n: \d{R}^+\rightarrow \d{R}^+$ such that  $u\mapsto u^{-\beta} f_n(c_n u)$ is decreasing, $f_n(1) = O(1)$, and $E_{0}\left[\sup_{|u| \leq c_n\delta} |W_{n, x}(u)|\right] \leq f_n( c_n\delta)$ for all large $n$ and $\delta \leq \delta^*$.
\end{description}
In addition, we introduce  conditions on the uniform convergence of  estimators $\Phi_n$ and $\Gamma_n$:
\begin{description}[style=multiline,leftmargin=1cm]
\item[(A4)] $c_n E_{0}\left[\sup_{|v| < \delta} |\Phi_n(x + v) - \Phi_0(x+v)|\right] \longrightarrow 0$ for some $\delta > 0$;
\item[(A5)] $\| \Gamma_{n,0} - \theta_0(x)\cdot \Phi_{n,0}\|_{\infty,I_n} \inprob 0$.
\end{description}
\begin{thm}[Convergence in distribution] \label{thm:conv_in_dist} If $x$ is an interior point of $I$ at which $\Phi_0$ is continuously differentiable with positive derivative and $\theta_0$ satisfies $\lim_{u \to 0} |\theta_0(x+u) - \theta_0(x)|/|u|^\alpha = \pi_0(x)$, conditions (A1)--(A5) imply that 
\[r_n \left[ \theta_n(x) - \theta_0(x) \right] \indist \Phi_0'(x)^{-1} \partial_- \n{GCM}_{\d{R}}\left\{ v\mapsto W_{x}(v) + \left[\frac{\pi_0(x) \Phi_0'(x)}{\alpha+1}\right] |v|^{\alpha + 1} \right\}(0)\] 
with $r_n := c_n^{\alpha}$. If in addition $\alpha = 1$, $\pi_0(x) = \theta_0'(x)$ and $W_x$ possesses stationary increments, then 
\[ r_n \left[ \theta_n(x) - \theta_0(x) \right] \indist  -\theta_0'(x) \argmin_{u \in \d{R}} \left\{W_{x}(u) + \tfrac{1}{2} \theta_0'(x)\Phi'_0(x) u^2 \right\}.\]
Furthermore, if $W_{x} = [\kappa_0(x)]^{1/2} W_0$ with $W_0$ a standard two-sided Brownian motion process satisfying $W_0(0)=0$, then $r_n \left[ \theta_n(x) - \theta_0(x) \right] \indist \tau_0(x) Z $ with  $\tau_0(x):=\left[4\theta_0'(x)\kappa_0(x)/\Phi_0'(x)^2\right]^{1/3}$ and $Z:=\argmin_{u \in \d{R}} \left\{W_0(u) + u^2 \right\}$.
\end{thm}
The latter limit distribution is referred to as a scaled Chernoff distribution, since $Z$ is said to follow the standard Chernoff distribution. This distribution appears prominently in classical results in nonparametric monotone function estimation and has been extensively studied (e.g., \citealp{groeneboom2001jcgs}). It can also be defined as the distribution of the slope at zero of $\n{GCM}_{\d{R}}\{u\mapsto W_0(u) + u^2\}$.

Theorem~\ref{thm:conv_in_dist} applies in the common setting in which $\theta_0$ is differentiable at $x$ with positive derivative -- in other words, when $\alpha = 1$. However, as in \cite{wright1981} and \cite{anevski2006general},  Theorem~\ref{thm:conv_in_dist} also applies in additional situations, including when $\theta_0$ has $\alpha\in\{2,3,\ldots\}$ derivatives at $x$, with null derivatives of order $j<\alpha$ and positive derivative of order $\alpha$. Nevertheless, Theorem~\ref{thm:conv_in_dist} does not cover situations in which $\theta_0$ is flat in a neighborhood of $x$. The limit distribution of the Grenander estimator at flat points was studied in \cite{carolan2008}, but it appears that similar results have not been derived for Grenander-type or generalized Grenander-type estimators.

We note the similarity of our Theorem~\ref{thm:conv_in_dist} to Theorem~2 of \cite{anevski2006general}. For the special case in which $\Phi_0$ is the identity transform, the consequents of the two results coincide. Our result explicitly permits alternative transforms.  Both results require weak convergence of a stochastic part of the primitive process, and also require the same local rate of growth of $\theta_0$. Additionally, condition (A2) is implied if for every $\epsilon$ and $\delta$ positive, there exists a finite $m \in (0, +\infty)$ such that $P_0(\sup_{|v| \geq m} |W_x(v)| |v|^{-\alpha -1} > \epsilon) < \delta$, as in Assumption A5 of \cite{anevski2006general}. However, the remaining conditions and methods of proof differ. To prove our result, we first generalize the switch relation of \cite{groeneboom1985density} and use it to convert $P_0(r_n \left[ \theta_n(x) - \theta_0(x) \right]> \eta)$ into the probability that the minimizer of a process involving $W_{n,x}$ falls below some value. After establishing weak convergence of this process, we then use conditions (A2) through (A5) to justify application of the argmin continuous mapping theorem. In contrast, \cite{anevski2006general} establish their result using a direct appeal to convergence in distribution of $\partial_-\n{GCM}_C(Y_n)(0)$ to $\partial_-\n{GCM}_C(Y_0)(0)$, where $Y_n$ is a local limit process and $Y_0$ its weak limit. They also provide lower-level sufficient conditions for this convergence. It may be possible to establish the consequent of Theorem~\ref{thm:conv_in_dist}, permitting in particular the use of a non-trivial transformation $\Phi_0$, using Theorem~2 of \cite{anevski2006general} or a suitable generalization thereof. We have specified our sufficient conditions with applications to the setting $\alpha = 1$ and $c_n = n^{1/3}$ in mind, as we discuss at length in the next section.

Suppose that $W^0_x$ is the limit process that arises when no domain transformation is used in the construction of a generalized Grenander-type estimator, that is, when both $\Phi_0$ and $\Phi_n$ are taken to be the identity map. In this case, under (A1)--(A5), Theorem~\ref{thm:conv_in_dist} indicates that 
\[ r_n \left[ \theta_n(x) - \theta_0(x) \right] \indist \partial_- \n{GCM}_{\d{R}}\left\{ v\mapsto W_{x}^0(v) + \left[\frac{\pi_0(x)}{\alpha + 1}\right] |v|^{\alpha + 1} \right\}(0)\ .\]
It is natural to ask how this limit distribution compares to the one obtained using a non-trivial transformation $\Phi_0$. In particular, does using $\Phi_0$ change the pointwise distributional results for $\theta_n$? The answer is  of course negative whenever $W_x$ and $\Phi_0'(x) W^0_x$ are equal in distribution, since $\n{GCM}_{\d{R}}$ is a homogeneous operator. A more detailed discussion of this question and lower-level conditions are provided in the next section.

\section{Refined results for asymptotically linear primitive and transformation estimators}\label{sec:improved}

\subsection{Distributional results}


In applications of their main result, \cite{anevski2006general} focus primarily on providing lower-level conditions to characterize the relationship between various dependence structures and asymptotic results for monotone regression and density function estimation. \cite{anevski2011}, \cite{dedecker2011} and \cite{bagchi2017} provide additional applications of \cite{anevski2006general} to monotone function estimation with dependent data. Our Theorem~\ref{thm:conv_in_dist} could be used, for instance, to relax the common assumption of a uniform design in the analysis of monotone regression estimators. Here, we pursue an alternative direction, focusing instead on providing lower-level conditions for consistency of $\theta_n$ and convergence in distribution of $r_n[\theta_n(x)-\theta_0(x)]$ for use in the important setting in which $\alpha = 1$, $r_n=c_n= n^{1/3}$, the data are independent and identically distributed, and $\Gamma_n$ and $\Phi_n$ are asymptotically linear estimators. Such settings arise frequently, for instance, when the primitive and transformation parameters are smooth mappings of the data-generating mechanism.

Below, we write $Pf$ to denote $\int f (o)dP(o)$ for any probability measure $P$ and $P$-integrable function $f:\mathscr{O}\rightarrow\mathbb{R}$. We also use $\d{P}_n$ to denote the empirical distribution of independent observations $O_1,O_2,\ldots,O_n$ from $P_0$ so that $\d{P}_nf=\frac{1}{n}\sum_{i=1}^{n}f(O_i)$ for any $f:\mathscr{O}\rightarrow\mathbb{R}$.

Suppose that there exist functions $D_{x,0}^*:\mathscr{O}\rightarrow \mathbb{R}$ and $L_{x,0}^*:\mathscr{O}\rightarrow \mathbb{R}$ depending on $P_0$ such that, for each $x\in I$, $P_0D^*_{x,0}=P_0L^*_{x,0}=0$ and both $P_0D^{*2}_{x,0}$ and $P_0L^{*2}_{x,0}$ are finite, and  \begin{align}
\Gamma_{n}(x)-\Gamma_{0}(x)  &=\d{P}_n D_{x, 0}^* + H_{x,n}\quad\mbox{and}\quad \Phi_{n}(x) - \Phi_{0}(x) = \d{P}_n L_{x, 0}^* + R_{x,n}\ ,\label{eq:asy_linear}
\end{align}
where $H_{x,n}$ and $R_{x,n}$ are stochastic remainder terms. If $n^{1/2}\sup_{x \in I} |H_{x,n}|$ and $n^{1/2}\sup_{x \in I} |R_{x,n}|$ tend to zero in probability, we say that $\Gamma_n$ and $\Phi_n$ are \emph{uniformly asymptotically linear} over $I$ as estimators of $\Gamma_0$ and $\Phi_0$, respectively. The objects $D^*_{x,0}$ and $L^*_{x,0}$ are referred to as the influence functions of $\Gamma_n(x)$ and $\Phi_n(x)$, respectively, under sampling from $P_0$.

Assessing consistency and uniform consistency of $\theta_n$ is straightforward when display~\eqref{eq:asy_linear} holds. For example, if the classes $\{ D_{x,0}^* : x \in I\}$ and $\{L_{x,0}^* : x \in I\}$ are $P_0$-Donsker, and  $n^{1/2}\sup_{x \in I} |H_{x,n}|$ and $n^{1/2}\sup_{x \in I} |R_{x,n}|$ are bounded in probability, then $n^{1/2}\|\Gamma_n - \Gamma_0\|_{\infty, I}$ and $n^{1/2}\|\Phi_n - \Phi_0\|_{\infty, I}$ are both bounded in probability. Thus, Theorems~\ref{thm:consistency} and~\ref{thm:rates}  can be directly applied with $r_n = n^{1/2}$ provided the required conditions on $\theta_0$ and $\Phi_0$ hold. As such, we focus here on deriving a refined version of Theorem \ref{thm:conv_in_dist} for use whenever display~\eqref{eq:asy_linear} holds.

It is reasonable to expect the linear terms $\d{P}_nD^*_{x,0}$ and $\d{P}_nL^*_{x,0}$ to drive the behavior of the standardized difference $r_n[\theta_n(x)-\theta_0(x)]$ in Theorem~\ref{thm:conv_in_dist}. The natural rate here is $c_n = r_n=n^{1/3}$, for which \cite{kim1990cube} provide intuition. Our first goal in this section is to provide sufficient conditions for weak convergence of the process $\{ n^{1/6}\d{G}_n g_{x,n^{-1/3} u} : |u| \leq M\}$, where $\d{G}_n$ is the empirical process $n^{1/2}(\d{P}_n-P_0)$ and we define the localized difference function $g_{x,v} := D_{x + v,0}^* - D_{x,0}^*- \theta_0(x)(L_{x + v,0}^* - L_{x,0}^*)$.  \cite{kim1990cube} also provide detailed conditions for weak convergence of processes of this type. Building upon their results, we are able to provide simplified sufficient conditions for convergence in distribution of $n^{1/3}[\theta_n(x)-\theta_0(x)]$ when $\Gamma_n$ and $\Phi_n$ are uniformly asymptotically linear estimators.

We begin by introducing conditions we will refer to. First, we define $\s{G}_{x,R} := \{g_{x,u} : |u| \leq R\}$ and suppose that $\s{G}_R$ has envelope function $G_{x,R}$. The first two conditions concern the size of $\s{G}_{x,R}$ for small $R$ in terms of bracketing or uniform entropy numbers, which for completeness we define here -- see  \cite{van1996weak} for a comprehensive treatment. Denote by $\| G\|_{P,2} = [ P( G^2)]^{1/2}$ the $L_2(P)$ norm of a given $P$-square-integrable function $G:\mathscr{O}(P)\rightarrow\mathbb{R}$. The bracketing number $N_{[]}(\varepsilon, \s{G}, L_2(P))$ of a class $\s{G}$ with respect to the $L_2(P)$ norm is the smallest number of $\varepsilon$-brackets needed to cover $\s{G}$, where an $\varepsilon$-bracket is any set of functions $\{ f: \ell \leq f \leq u\}$ with $\ell$ and $u$ such that $\|\ell-u\|_{P,2} < \varepsilon$. The covering number $N(\varepsilon, \s{G}, L_2(Q))$ of $\s{G}$ with respect to the $L_2(Q)$ norm is the smallest number of $\varepsilon$-balls in $L_2(Q)$ required to cover $\s{G}$. The uniform covering number is the supremum of $N(\varepsilon\|G\|_{2,Q}, \s{G}, L_2(Q))$ over all discrete probability measures $Q$ such that $\|G\|_{2,Q} > 0$, where $G$ is an envelope function for $\s{G}$. We consider conditions on the size of $\s{G}_{x,R}$:
\begin{description}[style=multiline,leftmargin=1cm]
\item[(B1)] for some constants $C>0$ and $V>-1$, either (B1a) $\log N_{[]}(\varepsilon \|G_{x,R}\|_{P_0,2}, \s{G}_{x,R}, L_2(P_0)) \leq C\varepsilon^{2V}$ or (B1b) $\log \sup_Q N(\varepsilon \|G_{x,R}\|_{Q,2}, \s{G}_{x,R}, L_2(Q)) \leq C\varepsilon^{2V}$ for all $\varepsilon \in (0,1]$ and $R$ small enough;
\item[(B2)] $P_0 G_{x,R}^2 = O(R)$, and for all $\eta > 0$, $P_0 G_{x,R}^2\{RG_{x,R}>\eta\} = o(R)$, as $R\rightarrow 0$.
\end{description}
Condition (B1) replaces the notion of \emph{uniform manageability} of the class  $\s{G}_{x,R}$ for small $R$ as defined in \cite{kim1990cube}, whereas condition (B2) directly corresponds to their condition (vi). Since bounds on the bracketing and uniform entropy numbers have been derived for many common classes of functions, condition (B1) can be readily checked in practice. Together, conditions (B1) and (B2) ensure that $\mathscr{G}_{x,R}$ is a relatively small class, and this helps to establish the weak convergence of the localized process $\{W_{n,x}(u):|u|\leq M\}$.

As in \cite{kim1990cube}, to guarantee that the covariance function of this localized process stabilizes, it suffices that $\delta^{-1}\sup_{|u - v| < \delta} P_0(g_{x,u}- g_{x,v})^2$ be bounded for small enough $\delta>0$ and that, up to a scaling factor possibly depending on $x$, $\sigma_{x,\alpha}(u,v):=\alpha^{-1}P_0 [(g_{x, \alpha u}-P_0g_{x,\alpha u})(g_{x, \alpha v}-P_0g_{x,\alpha v})]$ tend to the covariance function $\sigma^2(u,v)$ of a two-sided Brownian motion as $\alpha\rightarrow 0$. Below, we provide simple conditions that imply these two statements for a broad class of settings that includes our examples.

The covariance function of the Gaussian process to which $\{\d{G}_n[D_{t,0}^* -\theta_0(x) L_{t,0}^*]:t\}$ converges weakly is defined pointwise as $\Sigma_0(s,t):= P_0 [D_{s,0}^* -\theta_0(x) L_{s,0}^*][D_{t,0}^* -\theta_0(x) L_{t,0}^*]$. The behavior of $\Sigma_0$ near $(x,x)$ dictates the covariance of the local limit process $W_x$ and hence the scale parameter $\kappa_0(x)$. If $\Sigma_0$ is differentiable in $(s,t)$ at $(x,x)$, it follows that $\kappa_0(x)=0$ and $\theta_n$ converges at a faster rate, although possibly with an asymptotic bias. When instead scaled Chernoff asymptotics apply, the covariance function can typically be written as 
\begin{equation} \Sigma_0(s,t) = \Sigma^*_0(s,t) + \iint_{-\infty}^{s \wedge t} A_0(s, t, v, w)H_0(dv, w)Q_0(dw)\label{eq:cov_decomp}\end{equation}
for some functions $\Sigma^*_0:I\times I\rightarrow \mathbb{R}$, $A_0:I\times I\times I\times \mathscr{W}\rightarrow \mathbb{R}$ and $H_0:I\times \mathscr{W}\rightarrow \mathbb{R}$ depending on $P_0$, where $Q_0$ is a probability measure induced by $P_0$ on some measurable space $\mathscr{W}$. In this representation, $\Sigma^*_0$ is taken to be the differentiable portion of the covariance function, which does not contribute to the scale parameter. The second summand is not differentiable at $(x,x)$ and makes $\sigma_{x,\alpha}(u,v)$ tend to a non-zero limit. We consider cases in which $\Sigma^*_0$, $A_0$ and $H_0$ satisfy the following conditions:
\begin{description}[style=multiline,leftmargin=1cm]
\item[(B3)] Representation \eqref{eq:cov_decomp} holds, and for some $\delta>0$, setting $B_\delta(x):=(x-\delta,x+\delta)$, it is also true that:
\begin{description}[style=multiline,leftmargin=1.25cm]
\item[(B3a)] $\Sigma^*_0$ is symmetric in its arguments and continuously differentiable on $B_{\delta}(x)$;
\item[(B3b)] $A_0$ is symmetric in its first two arguments, and $s \mapsto A_0(s, t, v, w)$ is differentiable for $Q_0$-almost every $w$ and each $s,t,v\in B_{\delta}(x)$, with derivative $A_0'(s,t,v,w)$ continuous in $s,t,v$ each in $B_{\delta}(x)$ for $Q_0$-almost every $w$ and satisfying the boundedness condition
\[\iint_{-\infty}^{x + \delta} \sup_{s, t \in B_{\delta}(x)} |A_0'(s, t, v, w)| H_0(dv, w) Q_0(dw) < \infty\ ;\]
\item[(B3c)] $v \mapsto A_0(x,x,v, w)$ is continuous at $v=x$ uniformly in $w$ over the support of $Q_0$;
\item[(B3d)] $v \mapsto H_0(v, w)$ is nondecreasing for all $w$ and differentiable at each $v \in B_{\delta}(x)$, with derivative $H'_0(v,w)$ continuous at $v=x$ uniformly in $w$ over the support of $Q_0$.
\end{description}
\end{description}
Representation \eqref{eq:cov_decomp} is deliberately broad to encompass a wide variety of parameters. Nevertheless, in many settings, the covariance function can be considerably simplified, leading then to simpler conditions in (B3). For instance, when  $W$ is a vector of covariates over which marginalization is performed to compute the parameter, $Q_0$ typically plays the role of the marginal distribution of $W$ under $P_0$. In classical problems in which there is no adjustment for covariates, this feature of representation \eqref{eq:cov_decomp} is not needed and indeed vanishes. In other settings, $A_0(s,t,v,w)$ depends on $v$ and $w$ but not on $s$ and $t$.

Finally, we must ensure that the stochastic remainder terms $H_{x,n}$ and $R_{x,n}$ arising the asymptotic linear representations of $\Gamma_n$ and $\Phi_n$ do not contribute to the limit distribution. Defining $\tilde{H}_{u,n} := H_{x+u,n} - H_{x,n}$, $\tilde{R}_{u,n} := R_{x+u,n} - R_{x,n}$ and
$K_n(\delta) := n^{2/3}\sup_{|u| \leq \delta n^{-1/3}}|\tilde{H}_{u,n} - \theta_0(x) \tilde{R}_{u,n}|$, we consider the following conditions for the asymptotic negligibility of these remainder terms:
\begin{description}[style=multiline,leftmargin=1cm]
\item[(B4)] $K_n(\delta) \inprob 0$ for each fixed $\delta>0$;
\item[(B5)] for some $\alpha\in(1,2)$, $\delta \mapsto \delta^{-\alpha}E_{0} \left[K_n(\delta)\right]$ is decreasing for all $\delta$ small enough and $n$ large enough.
\end{description}
Condition (B4) guarantees that the remainder terms do not contribute to the weak convergence of $\{W_{n,x}(u) : |u| \leq M\}$, and condition (B5) guarantees that the remainder terms satisfy condition (A3).

Combining the conditions above, we can state the following master theorem for pointwise convergence in distribution when the monotone estimator is based upon asymptotically linear primitive and transformation estimators:
\begin{thm}\label{dist_asy_lin}
Suppose that, at an interior point $x\in I$, $\theta_0$ is differentiable and $\Phi_0$ is continuously differentiable with positive derivative. Suppose also that $\Gamma_n$ and $\Phi_n$ satisfy display~\eqref{eq:asy_linear}, and that conditions (B1)--(B5) and (A4)--(A5) hold (with $c_n = n^{1/3}$). Then, it holds that \[n^{1/3} \left[ \theta_n(x) - \theta_0(x) \right] \indist  \tau_0(x) Z\ ,\] where $Z$ follows the standard Chernoff distribution, and $\tau_0(x):= \left[4\theta_0'(x)\kappa_0(x)/\Phi_0'(x)^{2}\right]^{1/3}$ is a scale factor involving $\kappa_0(x) := \int A_0(x,x,x, w) H_0'(x,w)Q_0(dw)$.
\end{thm}

\subsection{Effect of domain transform on limit distribution}


As was done briefly after Theorem \ref{thm:conv_in_dist}, it is natural to compare the limit distribution obtained by Theorem \ref{dist_asy_lin} when a transformation of the domain is used and when it is not. We will consider $\theta_n := \n{Iso}_{I}( \Theta_n, \n{Id})$, the estimator obtained by directly isotonizing an estimator $\Theta_n$ of the primitive function $\Theta_0$ without use of a domain transformation. Denoting by $\Phi_0$ a candidate non-decreasing transformation function, and letting $\Gamma_0 := \Psi_0 \circ \Phi_0$ be as described in Section \ref{sec:grenander}, we will also consider $\theta_n^* := \n{Iso}_{J_n}(\Gamma_n \circ \Phi_n^-, \Phi_n)$, where $\Gamma_n$ and $\Phi_n$ are estimators of $\Gamma_0$ and $\Phi_0$, respectively. Suppose $\Theta_n(x), \Gamma_n(x)$ and $\Phi_n(x)$ are each asymptotically linear estimators of their respective targets with influence functions $M_{x,0}^*$, $D_{x,0}^*$ and $L_{x,0}^*$, respectively, under sampling from $P_0$.

We wish to compare the scale parameters $\kappa_0(x)$ and $\kappa_0^*(x)$ arising from the use of the distinct estimators $\theta_n(x)$ and $\theta_n^*(x)$. To do so, we can use expression (B3) to examine the covariance obtained in both cases. However, it appears difficult to say much without having more specific forms for the involved influence functions. Unfortunately, it also appears difficult to characterize these influence functions generally since they depend inherently on the parameter of interest $\theta_0$, and we wish to remain agnostic to the form of $\theta_0$. Nevertheless, in our next result, we describe a class of problems, characterized by the generated influence functions and regularity conditions on these, in which domain transformation has no effect on the limit distribution of the generalized Grenander-type estimator.
\begin{thm}\label{thm:transform}
Suppose conditions (B1)--(B5) hold for $(\Theta_n, \n{Id})$ and $(\Gamma_n, \Phi_n)$, and the observed data unit can be partitioned as $O=(U, Z)$ with $U\in\d{R}^+$. Suppose that the influence functions can be expressed as
\begin{align*}
M_{x,0}^*: &\ (u,z)\mapsto I_{[0,x]}(u) M_{x,0}^{(1)}(u,z) + M_{x, 0}^{(2)}(u,z)\ ,\\ 
L_{x,0}^*: &\ (u,z)\mapsto I_{[0,x]}(u) L_{x,0}^{(1)}(u,z) + L_{x,0}^{(2)}(u,z)\ , \\
D_{x,0}^*: &\ (u,z)\mapsto I_{[0,x]}(u) \Phi_0'(u)M_{x,0}^{(1)}(u,z) + D_{x, 0}^{(2)}(u,z) +\int_0^x \theta_0(v) L_{dv,0}^*(u,z)\ ,
\end{align*}
and satisfy the smoothness conditions stated in the Appendix. Suppose that the density function $h_0$ of the conditional distribution of $U$ given $Z$ exists and is continuous in a neighborhood of $x$ uniformly over the support of the marginal distribution $Q_{Z,0}$ of $Z$. Then, it follows that \[\kappa_0(x)=\int \left[M^{(1)}_{x,0}(x,z)\right]^2 h_0(x \mid z)Q_{Z,0}(dz)\ \ \ \mbox{and}\ \ \ \kappa_0^*(x)=\left[\Phi_0'(x)\right]^2\int \left[M_{x,0}^{(1)}(x,z)\right]^2 h_0(x \mid z) Q_{Z,0}(dz)\ .\] Consequently, $n^{1/3}\left[\theta_n(x) - \theta_0(x)\right]$ and $n^{1/3}\left[\theta_n^*(x) - \theta_0(x)\right]$ have the same limit distribution.
\end{thm}
The forms of $M^*_{x,0}$ and $L^*_{x,0}$ arise naturally in a wide variety of settings because the parameters considered involve a primitive function. The supposed form of $D^*_{x,0}$ may seem restrictive at first glance but is in fact expected given the forms of $M^*_{x,0}$ and $L^*_{x,0}$. A heuristic justification based on the product rule for differentiation is provided in the Supplementary Material. In all the examples we study in Section \ref{sec:examples}, the conditions of Theorem~\ref{thm:transform} apply. This provides justification for why, in each of these examples, the use of a domain transform has no impact on the limit distribution.

We remind the reader that, even if the domain transformation has no impact on the pointwise limit distribution, use of a domain transformation is still of great practical value in many circumstances. In complex problems, an estimator $\Theta_n$ may not be readily available for the primitive parameter $\Theta_0$ obtained without the use of a domain transformation. In some cases, $\Theta_0$ may not even be well-defined, so that transformation of the domain is unavoidable. Even when $\Theta_0$ is well-defined and  an estimator $\Theta_n$ is available, with the use of a carefully chosen transformation, it may be possible to avoid the need to estimate certain nuisance parameters or to substantially simplify the verification of conditions (B1)--(B5). Examples of these phenomena are presented in Section~\ref{sec:examples}.

\subsection{Negligibility of remainder terms}\label{remainder}

In some applications, the estimators $\Gamma_n$ and $\Phi_n$ may be linear rather than simply asymptotically linear. In such situations, the remainder terms $H_{x,n}$ and $R_{x,n}$ are identically zero, and conditions (B4) and (B5) are trivially satisfied. Otherwise, these conditions must be  verified. While in general the exact form of these remainder terms depends upon the specific parameter under consideration and estimators used, it is frequently the case that part of the remainder is an empirical process term arising from the estimation of nuisance functions appearing in the influence functions $D_{x,0}^*$ and $L_{x,0}^*$, as we illustrate below with one particular construction. To facilitate the verification of conditions (B4) and (B5) for these empirical process terms, we outline sufficient conditions in terms of uniform entropy and bracketing numbers.

In this subsection, we assume that $\Gamma_0(x)$ and $\Phi_0(x)$ arise as the evaluation at $P_0$ of maps from $\mathscr{M}$ to $\d{R}$, and denote by $\Gamma_P(x)$ and $\Phi_P(x)$ the evaluation of these maps at an arbitrary $P\in\mathscr{M}$. Let $\pi=\pi(P)$ be a summary of $P$, and suppose that $\Gamma_P(x)$, $\Phi_P(x)$ and the nonparametric efficient influence functions of $P\mapsto \Gamma_P(x)$ and $P\mapsto \Phi_P(x)$ at $P$ each only depend on $P$ through $\pi$. Denote these efficient influence functions by $D_{x}^*(\pi)$ and $L_{x}^*(\pi)$, respectively. Since $\mathscr{M}$ is nonparametric, it must be that $D^*_{x,0}=D^*_{x}(\pi_0)$ and $L^*_{x,0}=L^*_x(\pi_0)$ for $\pi_0:=\pi(P_0)$. To emphasize the fact that $\Gamma_P(x)$ and $\Phi_P(x)$ depend on $P$ only through $\pi$, we will use the symbols $\Gamma_\pi(x)$ and $\Phi_\pi(x)$ to refer to $\Gamma_P(x)$ and $\Phi_P(x)$, respectively.

Under regularity conditions, the so-called \emph{one-step estimators}
 \begin{equation}\label{onestep} \Gamma_n(x):=\Gamma_{\pi_n}(x)+\d{P}_nD^*_x(\pi_n)\mbox{\ \ \ \ and\ \ \ \ }\Phi_n(x):=\Phi_{\pi_n}(x)+\d{P}_nL^*_{x}(\pi_n) \end{equation} are asymptotically linear and efficient estimators of $\Gamma_0(x)$ and $\Phi_0(x)$, even when $\pi_n$ is a data-adaptive (e.g., machine learning) estimator of $\pi_0$ (e.g., \citealp{pfanzagl1982}). \cite{van2006estimating} pioneered the use of such one-step estimators in the context of nonparametric monotone function estimation. When this one-step construction is used, it can be shown that the remainder terms have the form $H_{x,n}=H^{(1)}_{x,n}+H^{(2)}_{x,n}$ and  $R_{x,n}=R^{(1)}_{x,n}+R^{(2)}_{x,n}$, where $H^{(1)}_{x,n}:=(\d{P}_n-P_0)\left[D^*_{x}(\pi_n)-D^*_{x}(\pi_0)\right]$ and $R^{(1)}_{x,n}:=(\d{P}_n-P_0)\left[L^*_{x}(\pi_n)-L^*_{x}(\pi_0)\right]$ are empirical process terms, and $H^{(2)}_{x,n}$ and $R^{(2)}_{x,n}$ are so-called \emph{second-order} remainder terms arising from linearization of the corresponding parameter. Similar representations exist when other constructive approaches, such as gradient-based estimating equations methodology (e.g., \citealp{vanderlaan2003,tsiatis2007}) and targeted maximum likelihood estimation (e.g., \citealp{vanderlaan2011tmle}), are used. As we will see in the examples of Section \ref{sec:examples}, these second-order terms can usually be shown to be asymptotically negligible provided $\pi_n$ tends to $\pi_0$ fast enough in some appropriate norm. Here, we provide conditions on $\pi_n$ that ensure that the contribution of $H_{x,n}^{(2)}-\theta_0(x) R_{x,n}^{(2)}$ to $K_n(\delta)$ satisfies conditions (B4) and (B5).

A primary benefit of decomposing the remainder terms as above is that the empirical process terms can be controlled using empirical process theory, a strategy also used in \cite{van2006estimating}. In particular, we can provide conditions under which $H^{(1)}_{x,n}$ and $R^{(1)}_{x,n}$ satisfy conditions (B4) and (B5). Defining $g_{x,u}(\pi) := [D_{x + u}^*(\pi) - D_{x}^*(\pi)] - \theta_0(x)[L_{x + u}^*(\pi) -L_{x}^*(\pi)]$, the relevant contribution of these empirical process terms to $K_n(\delta)$ is   
\begin{equation*}K^{(1)}_n(\delta) := n^{1/6}\sup_{|u| \leq \delta } \left|\d{G}_n \left[g_{x,un^{-1/3}}(\pi_n) - g_{x,un^{-1/3}}(\pi_0)\right] \right|\ .\label{eq:local_empirical}\end{equation*} Suppose that $\pi_n$ falls in a semimetric space $(\mathscr{P}, \rho)$ , with probability tending to one, and that $G_{x,\mathscr{P}, R}$ is an envelope function for $\s{G}_{ x,\s{P}, R} := \{ g_{x,u}(\pi): |u| \leq R , \pi \in \mathscr{P} \}$. We consider the following the conditions:
\begin{description}[style=multiline,leftmargin=1cm]
\item[(C1)] for some constants $C>0$ and $V>-1$, either (C1a) $\log N_{[]}(\varepsilon \|G_{x,\s{P},R}\|_{P_0,2}, \s{G}_{x,\s{P},R}, L_2(P_0)) \leq C\varepsilon^{2V}$ or (C1b) $\log \sup_Q N(\varepsilon \|G_{x,\s{P},R}\|_{Q,2}, \s{G}_{x,\s{P},R}, L_2(Q)) \leq  C\varepsilon^{2V}$ for all $\varepsilon \in (0,1]$ and $R$ small enough;
\item[(C2)] $P_0 G_{x,\s{P},R}^2 =O(R)$, and for all $\eta>0$, $P_0G^2_{x,\s{P},R}\{RG_{x,\s{P},R}>\eta\}=o(R)$, as $R\rightarrow 0$;
\item[(C3)] $P_0\left[ g_{x,u} (\pi) - g_{x,v}(\pi)\right]^2 =O(|u-v|)$ uniformly for $\pi\in\s{P}$, and $P_0\left[g_{x,u}(\pi_1) - g_{x,u}(\pi_2)\right]^2/\rho(\pi_1, \pi_2)^2 =O(|u|)$ uniformly for $\pi_1,\pi_2\in\s{P}$ and $u\in I$;
\item[(C4)]  there exists some $\bar{\pi}\in \s{P}$ such that $\rho(\pi_n, \bar{\pi}) \inprob 0$.
\end{description}
Our next result states that, under these conditions, the remainder term $K^{(1)}_n(\delta)$ stated above is asymptotically negligible in the sense of conditions (B4) and (B5).
\begin{thm}\label{local_empirical}
Suppose that, with probability tending to one, $\pi_n \in \s{P}$ and conditions (C1)--(C4) hold. Then, $K^{(1)}_n(\delta)$ satisfies conditions (B4)--(B5).
\end{thm}


We note that conditions (C1) and (C2) together imply conditions (B1) and (B2). As such, if conditions (C1) and (C2) have been verified, there is no need to also verify conditions (B1) and (B2).


\section{Applications of the general theory}\label{sec:examples}

In this section, we demonstrate the use of our general results for the three examples introduced in Section~\ref{sec:grenander}: estimation of monotone density, hazard and regression functions. For each of these functions, we consider various levels of complexity of the relationship between the ideal and observed data units. This allows us to illustrate that our general results (i) coincide with classical results in the simpler cases that have already been studied, and (ii) suggest novel estimation procedures with well-understood inferential properties, even in the context of complex problems that do not appear to have been previously studied. Below, we focus on distributional results for the various estimators considered. In each case, we state the main results in the text, and present additional technical details in Supplementary Material.

\subsection{Example 1: monotone density function}\label{density}

Let $\theta_0:=f_0$ be the density function of an event time $T$  with support $I:=[0,u_0]$, and suppose that $f_0$ is known to be non-decreasing on $I$. We will not use any transformation in this example, so we take $\Phi_0$ and $\Phi_n$ to be the identity map. Thus, $\psi_0 = \theta_0$ also corresponds to the density function of $T$, and $ \Psi_0 = \Theta_0 =  \Gamma_0$ to its distribution function. Below, we consider various data settings that increase in complexity. In the first setting, available observations are subject to independent right-censoring. In the second, the right-censoring mechanism is allowed to be informative -- only conditional independence of the event and censoring times given a vector of observed covariates is assumed. The first case has been studied in the literature -- for this, we wish to verify that our general results coincide with results already established. The second case is more difficult and does not seem to have been studied before. Our work in this setting not only highlights the generality of the theory in Sections \ref{sec:general} and \ref{sec:improved}, but also yields novel practical methodology.


\subsubsection{Independent censoring}

Suppose that $C$ is a positive random variable independent of $T$, and that the observed data unit is $O=(Y,\Delta)$, where $Y = \min(T, C)$ and $\Delta = I(T \leq C)$. The NPMLE of a monotone density function based on independently right-censored data was obtained in \cite{laslett1982density} and \cite{mcnichols1982density}, and distributional results were derived in \cite{huang1994density}. \cite{huang1995right} considered an estimator $\theta_n$ obtained by differentiating the GCM of the Kaplan-Meier estimator of the distribution function. While this is not the NPMLE, \cite{huang1995right} showed that it is asymptotically equivalent to the NPMLE, and it is an attractive estimator because it is simple to construct and reduces to the Grenander estimator if $T$ is fully observed, that is, if $C\geq T$ almost surely.

Since $\Psi_0$ is the distribution function $F_0=1-S_0$ with $S_0$ denoting the survival function of $T$, it is natural to consider $\Psi_n:=1-S_n$, where $S_n$ is the Kaplan-Meier estimator of $S_0$. It is well known that $n^{1/2}(S_n-S_0)$ converges weakly in $\ell^{\infty}([0,\tau])$ to a tight zero-mean Gaussian process as long as $G_0(\tau) > 0$ and $S_0(\tau)  < 1$, where $G_0$ denotes the survival function of $C$. Denoting by $\Lambda_0$ the cumulative hazard function corresponding to $S_0$, the influence function of the Kaplan-Meier estimator $S_n(x)$ is known to be the nonparametric efficient influence function
\[ D_{0,x}^*:(y, \delta) \mapsto S_0(x) \left[ -\frac{\delta I_{[0,x]}(y)}{S_0(y)G_0(y)} + \int_0^{y \wedge x} \frac{\Lambda_0(du)}{G_0(u)S_0(u)} \right]\] and so, the local difference $g_{x,u}$ can be written as 
\begin{align*}(y,\delta)\mapsto 
 \frac{-[S_0(x + u) - S_0(x)] \delta I_{[0,x+u]}(y)}{S_0(y) G_0(y)}-\frac{S_0(x)\delta I_{(x,x+u]}(y)}{S_0(y) G_0(y)}+ \int_{v<y} \frac{I_{(x,x+u]}(v)}{S_0(v)G_0(v)} \Lambda_0(dv)\ .
 \end{align*}
In Supplementary Material, we verify that condition (B2) is satisfied if $S_0$ and $G_0$ are positive in a neighborhood of $x$, and  that condition (B3) is  satisfied if $\theta_0$ is positive and continuous in a neighborhood of $x$. The covariance function is given by $\Sigma_0:(s,t)\mapsto\int_0^{s\wedge t} \frac{S_0(s)S_0(t)}{S_0(u)G_0(u) }\Lambda_0(du)$. We then get $\kappa_0(x) = [S_0(x) / G_0(x)] \lambda_0(x) = f_0(x) / G_0(x)$, so that the scale parameter is $\tau_0(x) = [4 f_0'(x) f_0(x) / G_0(x)]^{1/3}$. This agrees with the results of \cite{huang1995right}. In Supplementary Material, we  demonstrate that conditions (B4) and (B5) are also satisfied. In the case of no censoring, $\Sigma_0(s, t)$ simplifies to $\Gamma_0(s \wedge t)-\Gamma_0(s)\Gamma_0(t)$, so that $\Sigma^*_0(s,t)=\Theta_0(s)\Theta_0(t)$, $A_0(s,t,y,w)=\theta_0(y)$, $H_0(y,w)=y$ and $\kappa_0(x) = \theta_0(x)$. This agrees with the classical result of \cite{rao1969density} concerning pointwise convergence in distribution of the Grenander estimator.

\subsubsection{Conditionally independent censoring}


In many cases, the censoring mechanism may be informative but still independent of the event time process conditionally on a vector of recorded covariates. For simplicity, we only consider the case in which these covariates are defined at baseline, though the case of time-varying covariates can be tackled similarly. The observed data unit is now $O=(Y,\Delta,W)$, and we assume that $T$ and $C$ are independent given $W$. As long as $P_0(\Delta = 1 \mid W)$ is bounded away from zero almost surely, the survival function $S_0$ of $T$ can be identified pointwise in terms of the distribution $P_0$ of $O$ via the product-limit transform
\[S_0(x) = \int\Prodi_{t\leq x}\left[1-\frac{F_{1,0}(dt, w)}{S_{Y,0}(t\mid w)}\right]Q_0(dw)\ ,\] 
where $F_{1,0}(t\mid w):=P_0(Y\leq t,\Delta=1\mid W=w)$ is the conditional subdistribution function of $Y$ given $W=w$ corresponding to $\Delta=1$, $S_{Y,0}(t\mid w):=P_0(Y\geq t\mid W=w)$ is the conditional proportion-at-risk at time $t$ given $W=w$, and $Q_0$ is the marginal distribution of $W$ under $P_0$. This constitutes an example of coarsening at random, as described in \cite{heitjan1991car} and \cite{gill1997coarsening}. Estimation of the marginal survival function $S_0$ in the context of conditionally independent censoring has been studied before by \citet{hubbard2000survival}, \cite{scharfstein2002informative} and \cite{zeng2004dependent}, among others.

In this context, the nonparametric efficient influence function $D^*_{0,x}$ of $S_0(x)$ has the form $D_{0,x}-S_0(x)$, where $D_{0,x}$ is given by  
\[(y,\delta,w)\mapsto -S_0(x \mid w)\left[\frac{\delta I_{(-\infty, x]}(y)}{S_0(y \mid   w)G_0(y\mid w)}-\int_0^{y \wedge x} \frac{\Lambda( du\mid w)}{S_0(u \mid  w) G_0(u \mid w)}  \right] + S_0(x \mid w)\]
with $S_0(x\mid w)$ and $G_0(x \mid w)$ the conditional survival functions of $T$ and $C$, respectively, at $x$ given $W=w$, and $\Lambda_0(x\mid w)$ is the conditional cumulative hazard function of $T$ at $x$ given $W=w$. A simple one-step estimator of $\Gamma_0(x)$ is given by $\Gamma_n(x):= 1-\d{P}_n D_{n,x}$, where $D_{n,x}$ is obtained by substituting $S_n$ and $G_n$ for $S_0$ and $G_0$, respectively, in $D_{0,x}$. Conditions (B1) and (B2) are satisfied under uniform Lipschitz conditions on $S_0$ and $G_0$. As we show in Supplementary Material, condition (B3) holds, and we get $\kappa_0(x) = \int [f_0(x \mid w)/G_0(x \mid w)]Q_0(dw)$, where $f_0(x \mid w)$ is the conditional density of $T$ at $x$ given $W = w$. It follows directly then that the Chernoff scale factor is \[\tau_0(x)=\left[4f_0'(x)\int \frac{f_0(x \mid w)}{G_0(x \mid w)}Q_0(dw)\right]^{1/3},\] which reduces to the scale factor of \cite{huang1995right}  when $T$ and $C$ are independent. In Supplementary Material, we demonstrate that satisfaction of condition (B4) is highly dependent on the behavior of  $S_n$ and $G_n$. For instance, if $S_n - S_0$ and $G_n - G_0$ uniformly tend to zero in probability at rates faster than $n^{-1/3}$, then conditions (B4) and (B5) are satisfied. This is not a restrictive requirement if $W$ only has few components -- in such cases, many nonparametric smoothing-based estimators satisfy such rates. Otherwise, semiparametric estimators building upon additional structure (e.g., additivity on an appropriate scale) could be used. Alternatively, for higher-dimensional $W$, estimators of the form $S_n(x \mid w) =  \exp\left[-\int_0^x  \lambda_n(v \mid w) dv\right]$ with $\lambda_n$ an estimator of the conditional hazard $\lambda_0$ may be worth considering. For such $S_n$, we require the product of the convergence rates of $\lambda_n-\lambda_0$ and $G_n-G_0$ to be faster than $n^{-1/3}$. In practice, with a moderate or high-dimensional covariate vector $W$, it seems desirable to leverage multiple candidate estimators using ensemble learning (e.g., \citealp{vanderlaan2007super, vanderlaan2011tmle}).
 

\subsection{Example 2: monotone hazard function}

We now consider estimation of $\theta_0:=\lambda_0$, the hazard function of $T$. The most obvious approach to tackle this problem would be to consider an identity transformation as in the previous example. The primitive function of interest is then the cumulative hazard function $\Lambda_0$, which can be expressed as the negative logarithm of the survival function $S_0$ and estimated naturally using any asymptotically linear estimator of $S_0$, for example. The conditions of Theorem \ref{thm:conv_in_dist} and \ref{dist_asy_lin} can then be directly verified. An alternative, more expeditious approach consists of taking the domain transform $\Phi_0$ to be the restricted mean mapping $u\mapsto \int_0^u S_0(v)dv$. In such cases, $\Gamma_0$ is simply the cumulative distribution function $F_0$, and $u_0 = \int_0^\infty S_0(v)dv$ the mean of $T$. This particular choice of domain transformation for estimating a monotone hazard function therefore yields the same parameter $\Gamma_0$ as for estimating a monotone density with the identity transform. Denoting by $S_n$ the estimator of the survival function $S_0$ based on the available data, the resulting generalized Grenander-type estimator $\theta_n$ is defined by taking $\Gamma_n:=1-S_n$ and setting $\Phi_n$ to be $u\mapsto \int_0^u S_n(v)dv$ over $J_n = [0,u_n]$, where $u_n = \int_0^\infty S_n(v)dv$. As the result below suggests, when this special domain transform is used, we can leverage some of the work performed above in analyzing the Grenander-type estimator of a monotone density function under the various right-censoring schemes considered. We recall that $\n{Id}$ denotes the identity function.

\begin{thm}\label{thm:hazard}
Suppose that $E_0\left[\sup_{u \in I_n}|S_n(u) - S_0(u)|\right]=o(r_n^{-1})$ and set $\Gamma_n:=1-S_n$. If the pair $(\Gamma_n,\n{Id})$ satisfies conditions (A1)--(A3), then the pair $(\Gamma_n, \Phi_n)$ with $\Phi_n:u\mapsto \int_{0}^{u}S_n(v)dv$ necessarily satisfies conditions (A1)--(A5). In particular, for $\theta_n := \n{Iso}_{J_n}\left(\Gamma_n \circ \Phi_n^-,\Phi_n\right)$, this implies that
\[ r_n\left[\theta_n(x) - \theta_0(x)\right] \indist -\theta_0'(x) \argmin_{u \in \d{R}} \left\{ W_x(u) + \tfrac{1}{2}\theta_0'(x) S_0(x)u^2 \right\}\ .\]
If $W_x = [\kappa_0(x)]^{1/2} W_0$ for $W_0$ a two-sided Brownian motion, then $r_n \left[\theta_n(x) - \theta_0(x)\right] \indist \tau_0(x) Z$, where $Z$ is the standard Chernoff distribution and $\tau_0(x):=\left[4 \theta_0'(x)\kappa_0(x)/S_0(x)^{2}\right]^{1/3}$.
\end{thm}

Denote by $T_{(j)}$ the $j^{\textrm{th}}$ order statistic of $\{T_1,T_2,\ldots,T_n\}$ and define $T_{(0)}:= 0$. When there is no censoring, the choice $(\Gamma_n,\Phi_n)$ prescribed above indicates that $\Gamma_n$ is the empirical distribution function based on $Y_1,Y_2,\ldots,Y_n$, and $\Phi_n$ is defined pointwise as $\Phi_n(x):=\tfrac{1}{n}\sum_{i=1}^n \min(T_{(i)}, x)$, which is strictly increasing on $[0, T_{(n)}]$. Therefore,  $\theta_n(x)$ is the left derivative at $\Phi_n(x)$ of the GCM  of the graph of $ \{(\Phi_n(T_{(k)}), \Gamma_n(T_{(k)})) : k = 0,1, \dotsc, n\} = \{((\tfrac{n-k}{n})T_{(k)} + \tfrac{1}{n}\sum_{i=1}^k T_{(i)}, \tfrac{k}{n})   : k = 0,1, \dotsc, n\}$. This is the NPMLE of a non-decreasing hazard function with uncensored data -- see, for example, Chapter 2.6 of \cite{groene2014shape}.

In Supplementary Material, we verify conditions (A1)--(A3) for each of three right-censoring schemes when $\Theta_n=1-S_n$, and $\Phi_0$ and $\Phi_n$ are both equal to the identity. Thus, to use Theorem \ref{thm:hazard}, it would suffice to verify that $E_0\left[\sup_{u \in I_n}|S_n(u) - S_0(u)|\right]$ tends to zero faster than $n^{-1/3}$. This is straightforward given the weak convergence of $n^{1/2}\left(S_n-S_0\right)$. Thus, the above theorem provides distributional results for monotone hazard function estimators in each right-censoring scheme considered, as summarized below: \begin{enumerate}[(i)]
\item when there is no censoring, we find $\tau_0(x)=[4\lambda_0'(x)\lambda_0(x)/S_0(x)]^{1/3}$, which agrees with \cite{rao1970hazard};
\item when there is independent right-censoring, we find that $\tau_0(x)=\{4\lambda_0'(x)\lambda_0(x)/[G_0(x)S_0(x)]\}^{1/3}$, which agrees with \cite{huang1995right};
\item when there is conditionally independent right-censoring, an important setting that does not seem to have been previously studied in the literature, we find that \[\tau_0(x)=\left[\frac{4\lambda'_0(x)}{S_0(x)^2}\int \frac{f_0(x\mid w)}{G_0(x\mid w)}Q_0(dw)\right]^{1/3} = \left\{\frac{4\lambda'_0(x) \lambda_0(x)}{G_0(x)S_0(x)} \left[\frac{G_0(x)}{f_0(x)}\int \frac{f_0(x\mid w)}{G_0(x\mid w)}Q_0(dw)\right]\right\}^{1/3}.\]
\end{enumerate} 
If either $T$ or $C$ are independent of $W$, the unadjusted Kaplan-Meier estimator is consistent for the true marginal survival function of $T$, and so, unadjusted estimators of the density and hazard functions are consistent. In these cases, we may then ask how the asymptotic distributions of the adjusted and unadjusted estimators compare. Since all limit distributions are of the scaled Chernoff type, it suffices to compare the scale factors arising from the different estimators. The second expression in (iii) is helpful to assess the impact of unnecessary covariate adjustment. If $C$ and $W$ are independent, then $G_0(x\mid w)=G_0(x)$ for each $w$, and so, the scale factors in (ii) and (iii) are identical. If $T$ and $W$ are dependent, so that $f_0(x \mid w) = f_0(x)$ for each $w$, but $C$ and $W$ are not, then the scale factor in (iii) is generally larger than the scale factor in (ii). In summary, when using an adjusted rather than unadjusted estimator of the hazard function, there may only be a penalty in asymptotic efficiency when adjusting for covariates that $C$ depends on but $T$ does not. The relative loss of efficiency is given by $\left\{\int [G_0(x)/G_0(x \mid w)]Q_0(dw)\right\}^{1/3}$.

\subsection{Example 3: monotone regression function}

We finally consider estimation of a monotone regression function. We first focus on the simple case in which the association between the outcome and exposure of interest is not confounded. In such cases, the parameter of interest is the conditional mean of the outcome given exposure level, and the standard least-squares isotonic regression estimators can be used. We show that our general theory covers this classical case. We then consider the case in which the relationship between outcome and exposure is confounded but the confounders of this relationship have been recorded. In this more challenging case, we consider the marginalization (or standardization) of the conditional mean outcome given exposure level and confounders over the marginal confounder distribution. We study this problem using results from Section \ref{sec:improved}, which allow us to provide theory for a novel estimator proposed for this important case.

\subsubsection{No confounding}

In the standard least-squares isotonic regression problem, we observe independent replicates of $O:=(A, Y)$, where $Y \in \d{R}$ is an outcome and $A \in \d{R}$ is the exposure of interest. We are interested in the conditional mean function $\theta_0:=\mu_0$, where $\mu_0(x):=E_0\left(Y \mid A = x\right)$ is the mean outcome at exposure level $x$. The primitive function of $\theta_0$ can be written as $\Theta_0(t)=E_0\left[ YI_{(-\infty,t]}(A)/ f_0(A)\right]$ for each $t$, where $f_0$ is the marginal density of $A$. The corresponding primitive parameter at $x$ is pathwise differentiable with nonparametric efficient influence function $(a,y)\mapsto yI_{(-\infty,x]}(a)/ f_0(a) - \Theta_0(x)$. An obvious approach to estimation of $\theta_0$ consists of constructing an asymptotically linear estimator of $\Theta_0$ -- this involves nonparametric estimation of the nuisance density $f_0$ -- and differentiating the GCM of the resulting curve -- this involves selecting the interval over which the GCM is calculated.

By using a domain transformation, it is possible to avoid both the need for nonparametric density estimation and the choice of isotonization interval. Let $\Phi_0$ be the marginal distribution function of $A$. With this transformation, we note that $\Psi_0(t)=E_0\left[YI_{(-\infty,t]}(\Phi_0(A))\right]$ and $\Gamma_0(t)= E_0\left[ Y I_{(-\infty,t]}(A)\right]$ for each $t$. This suggests taking $\Phi_n$ to be the empirical distribution function based on $A_1,A_2,\ldots,A_n$ and  $\Gamma_n(x):=\frac{1}{n} \sum_{i=1}^n Y_i I_{(-\infty,x]}(A_i)$. The resulting estimator $\theta_n(x)$ is precisely the well-known least-squares isotonic regression estimator of $\theta_0(x)$. Since $\Phi_n$ is a step function with jumps at the observed values of $A$, $\theta_n(x)$ is equal to the left-hand slope of the GCM at $\Phi_n(x)$ of the so-called \emph{cusum diagram} $\{(\Phi_n(A_{k}), \Gamma_n(A_{k})) : k=0,1,\ldots,n\} =  \{ ( \tfrac{k}{n}, \tfrac{S_k}{n}) : k = 0,1, \dotsc, n\}$, where we let $A_0 = -\infty$, $S_0 = 0$ and $S_k = \sum_{i=1}^k Y_i$ for $k \geq 1$.



Because both $\Gamma_n$ and $\Phi_n$ are linear estimators, these estimators do not generate second-order remainder terms to analyze. The influence functions of $\Gamma_n$ and $\Phi_n$ are, respectively, $D^*_{0,x}: (a,y)\mapsto y I_{(-\infty,x]}(a)-\Gamma_0(x)$ and $L^*_{0,x}: (a,y)\mapsto I_{(-\infty,x]}(a)-\Phi_0(x)$. In Supplementary Material, we demonstrate that if in a neighborhood of $x$, the conditional variance function, defined pointwise as $\sigma_0^2(t) := \n{Var}_0(Y \mid A = t)$, is bounded and continuous, and $\Phi_0$ possesses a positive, continuous density, then Theorem~\ref{dist_asy_lin} holds with 
\[ \tau_0(x)=\left[\frac{4 \mu_0'(x) \sigma_0^2(x)}{f_0(x)} \right]^{1/3},\]
 coinciding with the classical results of \cite{brunk1970regression}.
 
 \subsubsection{Confounding by recorded covariates}
 
We now consider a scenario in which the relationship between outcome $Y$ and exposure $A$ is confounded by a vector $W$ of recorded covariates. The observed data unit is thus $O:=(W,A,Y)$. A more relevant estimand in this scenario might be the marginalized regression function $\theta_0:=\nu_0$ with $\nu_0(x)$ defined as $E_{0}\left[ E_0\left( Y \mid A = x, W\right)\right]$. We note that $\nu_0(x)$ can be interpreted as a causal dose-response curve if (i) $W$ includes all confounders of the relationship between $A$ and $Y$, and (ii) the probability of observing an individual subject to exposure level $x$ is positive in $P_0$-almost every stratum defined by  $W$. In many scientific settings, it may be known that the causal dose-response curve is monotone in exposure level.
 
We again consider transformation by the marginal distribution function of $A$. In other words, we set $\Phi_0(x):=P_0(A\leq x)$ and take $\Phi_n(x):=\frac{1}{n}\sum_{i=1}^{n}I_{(-\infty,x]}(A_i)$ for each $x$. We then have that 
\[\Gamma_0(x) = E_{0}\left[ \frac{YI_{(-\infty,x]}(A)}{g_0(A,W)}\right] = \iint  I_{(-\infty,x]}(a)\mu_0(a, w) \Phi_0(da) Q_0(dw)\ ,\]
 where $g_0$ is the density ratio $(a, w)\mapsto f_0(a \mid w) / f_0(a)$, with $f_0(a \mid w)$ denoting the conditional density function of $A$ at $a$ given $W=w$ and $f_0(a)$ the marginal density function of $A$ at $a$ as before, and $\mu_0$ is the regression function $(a, w) \mapsto E_0 (Y \mid A = a, W =w)$. While in this case the domain transform does not eliminate the need to estimate nuisance functions, it nevertheless results in a procedure for which there is no need to choose the interval over which the GCM is calculated. 

Setting $\eta_0(x, w) := \int I_{(-\infty,x]}(a)\mu_0(a, w) \Phi_0(da)$ for each $x$ and $w$, the nonparametric efficient influence function of $\Gamma_0(x)$ is
\[ (w,a,y)\mapsto I_{(-\infty,x]}(a)\left[ \frac{y - \mu_0(a, w)}{g_0(a, w)}+\theta_0(a)\right] + \eta_0(x, w) - 2\Gamma_0(x)\ . \] Suppose that $\mu_n$ and $g_n$ denote estimators of $\mu_0$ and $g_0$, respectively. If the empirical distributions $\Phi_n$ and $Q_n$ based on $A_1,A_2,\ldots,A_n$ and $W_1,W_2,\ldots,W_n$, respectively, are used as estimators of $\Phi_0$ and $Q_0$, it is not difficult to show that  
\[ \Gamma_n(x) := \frac{1}{n}\sum_{i=1}^{n} I_{(-\infty,x]}(A_i)\left[\frac{Y_i-\mu_n(A_i,W_i)}{g_n(A_i,W_i)}+\frac{1}{n}\sum_{j=1}^{n}\mu_n(A_i,W_j)\right] \] is a one-step estimator of $\Gamma_0(x)$, and that it is asymptotically efficient under regularity conditions on the nuisance estimators $\mu_n$ and $g_n$.

Conditions (B1)--(B5) can be verified with routine but tedious work. Here, we focus on condition (B3), which allows us to obtain the scale parameter of the limit distribution, and on condition (B4), which requires that the nuisance estimators converge sufficiently fast. We find  that condition (B4) is satisfied if, for some $\epsilon > 0$,
\[ \sup_{|x - u| \leq \epsilon}E_{0}\left[ \mu_n(u, W) - \mu_0(u, W)\right]^2  \sup_{|x - u| \leq \epsilon} E_{0}\left[ \frac{g_0(u, W)}{g_n(u,W)} - 1\right]^2 = \fasterthan\left(n^{-1/3}\right),\]
and additional empirical process conditions hold. Turning to condition (B3), under certain smoothness conditions, we have that $\kappa_0(x) = f_0(x)^2\int \left[ \sigma_0^2(x,w)/f_0(x \mid w)\right]Q_0(dw)$, where $\sigma_0^2:(a,w)\mapsto \n{Var}_0(Y \mid A = a, W =w)$ denotes the conditional variance function of $Y$ given $A$ and $W$. We then find that the scale parameter of the limit Chernoff distribution is 
\[ \tau_0(x) = \left\{ 4 \nu_0'(x)\int\left[ \frac{\sigma_0^2(x, W)}{f_0(x \mid W)}\right]Q_0(dw) \right\}^{1/3}.\]

The marginalized and marginal regression functions exactly coincide -- that is, $\nu_0=\mu_0$ -- if, for example, (i) $Y$ and $W$ are conditionally independent given $A$, or (ii) $A$ and $W$ are independent. It is natural then to ask how the limit distribution of estimators of these two parameters compare under scenarios (i) and (ii), when the parameters in fact agree with each other.  In scenario (i), the scale parameter obtained based on the estimator accounting for potential confounding reduces to \[\tau_{0,red}(x)\ =\ \left\{ 4 \mu_0'(x)\sigma_0^2(x)\int\frac{Q_0(dw)}{f_0(x \mid w)} \right\}^{1/3}\ \geq\ \left\{ \frac{4 \mu_0'(x)\sigma_0^2(x)}{\int f_0(x \mid w)Q_0(dw)} \right\}^{1/3}\ =\ \left\{\frac{4\mu_0'(x)\sigma^2_0(x)}{f_0(x)}\right\}^{1/3} \] by Jensen's inequality. Thus, if $Y$ and $W$ are conditionally independent given $A$, in which case there is no need to adjust for potential confounders, the marginal isotonic regression estimator has a more concentrated limit distribution than the marginalized isotonic regression estimator. In scenario (ii), the scale parameter of the estimator accounting for potential confounding reduces to \[\tau_{0,red}(x)\ =\ \left\{ \frac{4 \mu_0'(x)}{f_0(x)}\int \sigma_0^2(x, W)Q_0(dw) \right\}^{1/3}\ \leq\ \left\{ \frac{4 \mu_0'(x)\sigma_0^2(x)}{f_0(x)} \right\}^{1/3}\] given that $\int \sigma^2_0(x,w)Q_0(dw)\leq \sigma^2_0(x)$ by the law of total variance. Thus, if $A$ and $W$ are independent, the marginal isotonic regression estimator has a less concentrated limit distribution than the marginalized isotonic regression estimator. In both scenarios (i) and (ii), the difference in concentration between the limit distributions of the two estimators varies with the amount of dependence between $A$ and $W$. We note that these observations  are analogous to those obtained in linear regression.


\section{Simulation study}\label{sec:sims}

In this section, we report results from a small simulation study conducted to illustrate the large-sample results derived in Sections \ref{sec:general} and \ref{sec:improved}. Here, we consider Examples 1 and 2 from Section \ref{sec:examples}, namely estimation of a monotone density and hazard functions. Since the purpose of studying the cases without censoring or with independent censoring was to verify our general results in previously studied settings, our simulation is focused on the novel and more difficult scenario in which censoring is only conditionally independent. Through our simulation study, we wish to assess how well the finite-sample distribution of $n^{1/3}\left[\theta_n(x) - \theta_0(x)\right]$ approximates the limit distributions derived in the previous section.

Conditionally on a single covariate $W$ distributed uniformly on the interval $(-1,+1)$, we consider the event and censoring times $T$ and $C$ to be independent and to each follow a Weibull distribution. Specifically, we take the conditional distribution of $T$ given $W=w$ to be a Weibull distribution with shape parameter $4$ and scale parameter $\exp\left(\alpha_0+\alpha_1w\right)$, while we take the conditional distribution of $C$ given $W=w$  to be a Weibull distribution with shape parameter $2$ and scale parameter $\exp\left(\beta_0+\beta_1w\right)$. We perform simulations under four distinct settings: (i) both $T$ and $C$ depend on $W$; (ii) only $T$ depends on $W$; (iii) only $C$ depends on $W$; and (iv) neither $T$ nor $C$ depend on $W$. To achieve this,  in settings (i), (ii), (iii) and (iv), we set the vector $(\alpha_0,\alpha_1,\beta_0,\beta_1)$ of parameters to be $(0.25,-0.375,0.25,-0.75)$, $(0.25,-0.375,1,0)$, $(0.25,0,0.25,-0.75)$ and $(0.25,0,1,0)$, respectively. We note that $T$ and $C$ follow proportional hazards models conditionally on $W$, and that the marginal density and hazard functions of $T$ are monotone over the interval $[0,1]$.


We used the generalized Grenander-type estimators proposed in the previous section to estimate the marginal density and hazard functions of $T$  over $[0,1]$ in each of the four simulation settings. First, we employed a naive procedure based on the Kaplan-Meier estimator of $S_0$, and second, we used a one-step procedure based on estimating the underlying conditional event and censoring hazard functions using a Cox model with single covariate $W$ as main term only. We note that our goal differs from recent work on estimating a monotone baseline hazard (e.g., \citealp{lopuhaa2013breslow,lopuhaa2013shape,lopuhaa2017smooth,lopuhaa2018smooth}). Our interest is in the marginal distribution of $T$ rather than the conditional distribution of $T$ given $W = 0$. Additionally, in principle, other consistent estimators of the conditional distributions of $T$ and $C$ given $W$ could be used instead of Cox model-based estimators without changing the asymptotic results, as discussed in the previous section.


The true density and hazard functions are plotted in Figure \ref{fig:dens.haz.example} along with an overlay of ten realizations of the estimator based on the naive and one-step procedures for estimating the marginal survival function $S_0$ based on random samples of size $n=5000$. Realizations of the estimator based on the one-step procedure track the true marginal density and hazard functions of $T$ over all four simulation settings, as expected. Realizations of the estimator based on the naive procedure also track the true marginal density and hazard functions of $T$ for settings (ii) through (iv), since in each of these settings $T$ and $C$ are independent. However, in setting (i), the estimator based on the naive procedure is inconsistent. The limit of the estimators of the marginal  density and hazard functions can be derived to be the density and hazard functions, respectively, corresponding to the survival function 
\[ t \mapsto \exp\left[\int_0^t -\frac{\int f_0(v \mid w) G_0(v \mid w)Q_0(dw)}{\int S_0(v \mid w) G_0(v \mid w)Q_0(dw)} \, dv\right] .\]
These density and hazard functions are shown as black dotted lines in Figure~\ref{fig:dens.haz.example}.

In Figure \ref{fig:dens.haz.vars}, the empirical variance over 1000 simulations of $n^{1/3}\left[\theta_n(x) - \theta_0(x)\right]$ for $n=5000$ is compared to the corresponding theoretical variances based on the limit  theory we have presented in Section \ref{sec:examples}, for values of $x$ between 0 and 1 and under the four considered scenarios. The sampling variance of the estimator appears close to the theoretical large-sample variance, except for $x$ values near the upper boundary of the isotonizing interval. As expected, estimators based on the naive and one-step procedures have nearly identical sampling variances when only $T$ is dependent on $W$ (second column) and when neither $T$ nor $C$ are dependent on $W$ (fourth column), but the sampling variance of the estimator based on the naive procedure is smaller than that based on the one-step procedure when only $C$ is dependent on $W$ (third column).

The empirical sampling distribution over 1000 simulations of $n^{1/3}\left[\theta_n(0.7) - \theta_0(0.7)\right]$ for $n=5000$ is compared in Figure \ref{fig:dens.haz.dists} to the theoretical scaled Chernoff limit distributions under the four different scenarios. In all situations, the sampling distribution approximates the theoretical limit. In the left-most columns, the bias of the estimator based on the naive procedure is evident. In Figure \ref{fig:dens.haz.dists.n}, the empirical sampling distribution of the estimators in settings where both $T$ and $C$ are dependent on $W$ is plotted against the theoretical scaled Chernoff limit distribution for four different values of sample size $n$. At $n = 500$, the estimators are moderately biased downward, but as $n$ increases, this bias vanishes.

\section{Concluding remarks}\label{sec:discussion}

We have studied a broad class of estimators of monotone functions based on differentiating the greatest convex minorant of a preliminary estimator of a primitive parameter. A novel aspect of the class we have considered is its allowance for the primitive parameter to involve a possibly data-dependent transformation of the domain. The class we have defined is useful because it generalizes classical approaches for simple monotone functions, including density, hazard and regression functions, facilitates the integration of flexible, data-adaptive learning techniques, and allows valid asymptotic statistical inference. We have provided general asymptotic results for estimators in this class and have also derived refined results for  the important case wherein the primitive estimator is uniformly asymptotically linear. We have proposed novel estimators of extensions of classical monotone parameters that deal with common sampling complications, and described their large-sample properties using our general results.



Our primary goal in this paper has been to establish general theoretical results that can be applied to study many specific estimators, and as such, there are numerous potential applications of our results.  There are also a multitude of useful properties and modifications of Grenander-type estimators that have been studied in the literature and whose extension to our class would be important. For instance, kernel smoothing of a Grenander-type estimator yields a monotone estimator that possesses many of the properties of usual kernel smoothing estimators, including possibly faster convergence to a normal distribution (e.g., \citealp{mukerjee1988smooth, mammen1991smooth,groeneboom2010smooth}). The asymptotic distribution of the supremum norm error of Grenander-type estimators has also been derived (e.g., \citealp{durot2012}), and extending this result to our class would refine further our pointwise results. Asymptotic results at the boundaries of the domain and corrections for poor behavior there have been developed and would further enhance the utility of these methods (e.g., \citealp{woodroofe1993penalized, balabdaoui2011grenander, kulikov2006}).

There have also been various proposals for constructing asymptotically valid pointwise confidence intervals for Grenander-type estimators without the need to compute the complicated scale parameters appearing in their limit distribution. In regular statistical problems, the bootstrap is one of the most widely used such methods; unfortunately, the nonparametric bootstrap is known to fail for Grenander-type estimators (e.g., \citealp{kosorok2008bootstrap, sen2010bootstrap}). However, these articles have demonstrated that the $m$-out-of-$n$ bootstrap can be valid for Grenander-type estimators, and that bootstrapping smoothed versions of Grenander-type estimators can also be an effective strategy for performing inference. Asymptotically pivotal distributions based on likelihood ratios have also been used to avoid the need to estimate nuisance parameters in the limit distribution and to provide a basis for improved finite-sample inference (e.g., \citealp{banerjee2001ratio, banerjee2005semi, banerjee2005local, banerjee2007response, groeneboom2015nonparametric}). Considering these strategies in our setting would be particularly interesting.

\vspace{.1in}
\singlespacing
{\footnotesize
\section*{Acknowledgements}
The authors thank the referees and associate editor for providing constructive and insightful feedback that helped them improve this manuscript. They also thank Antoine Chambaz and Mark van der Laan for stimulating conversations that sparked their interest in this problem, Jon Wellner for sharing insight and helping them better understand the history of this problem, and Alex Luedtke and Peter Gilbert for providing feedback early on in this work. The authors  also gratefully acknowledge the support of NIAID grant 5UM1AI058635 (TW, MC) and the Career Development Fund of the Department of Biostatistics at the University of Washington (MC).\bibliographystyle{apa}
\bibliography{monotone_theory}
}

\begin{figure}[p]
\centering
\includegraphics[width=6in]{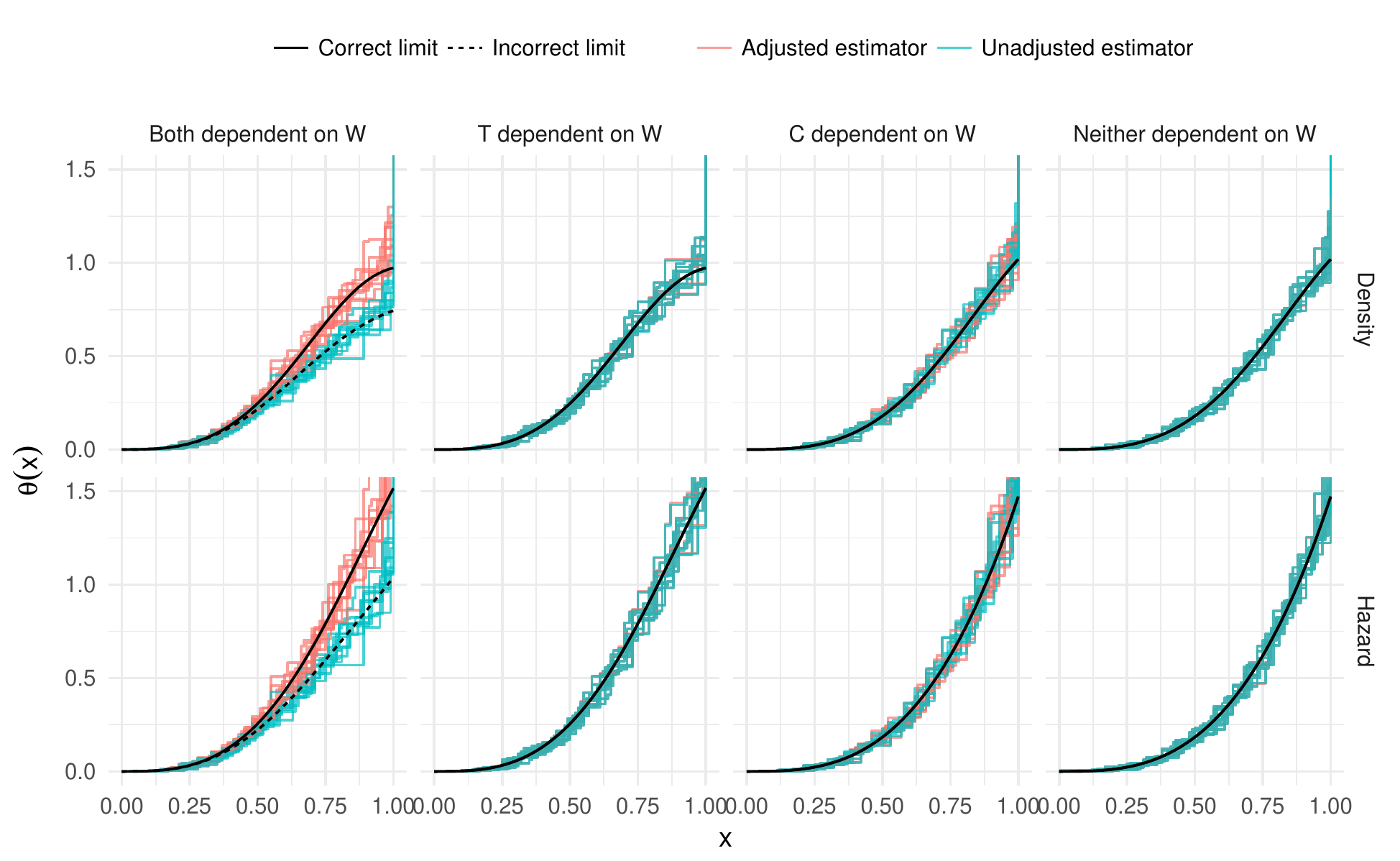}
\caption{Estimated monotone density and hazard functions based on 10 realizations of datasets including 5000 right-censored observations. Solid black lines are the true density and hazard functions. Dotted black lines indicate limit of unadjusted estimators. }
\label{fig:dens.haz.example}
\vspace{.1in}
\includegraphics[width=6in]{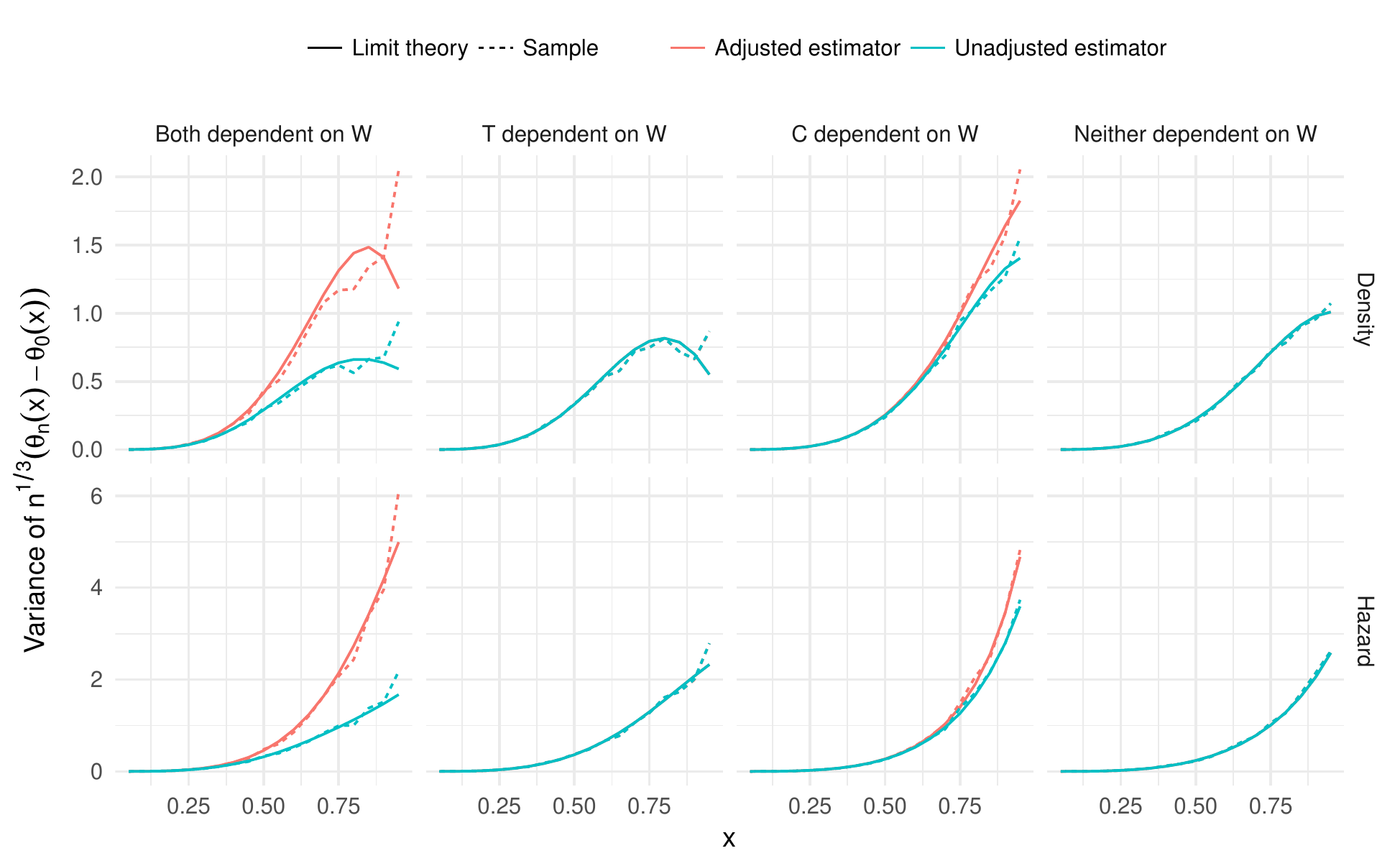}
\caption{Empirical variance over 1000 simulations of the standardized monotone density and hazard estimators and theoretical variance of the corresponding Chernoff limit distribution.}
\label{fig:dens.haz.vars}
\end{figure}

\begin{figure}[p]
\centering
\includegraphics[width=6in]{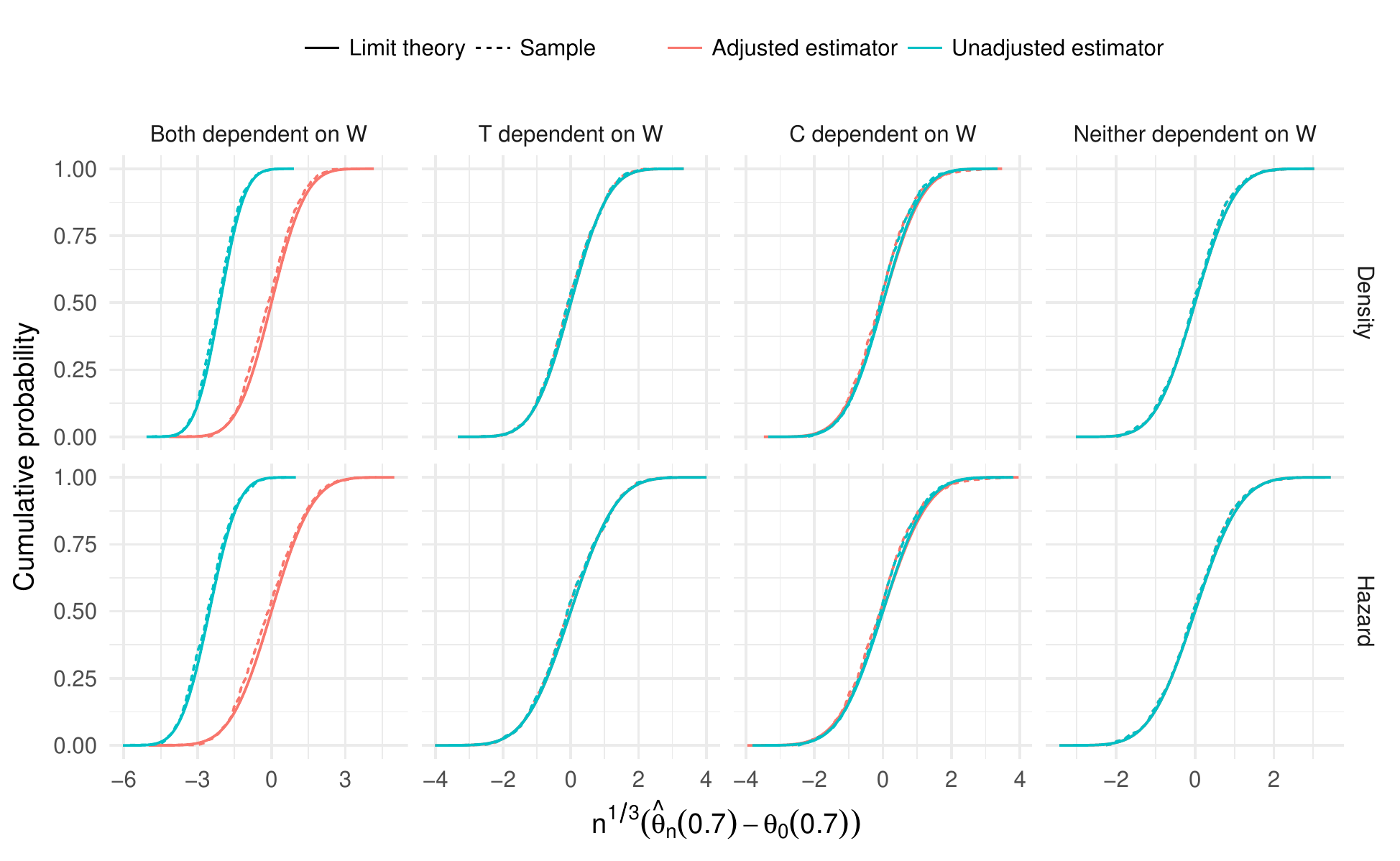}
\caption{Sampling distribution over 1000 simulations of the monotone density and hazard estimators at $x = 0.7$ and the corresponding theoretical scaled Chernoff limit distribution.}
\label{fig:dens.haz.dists}
\vspace{.1in}
\includegraphics[width=6in]{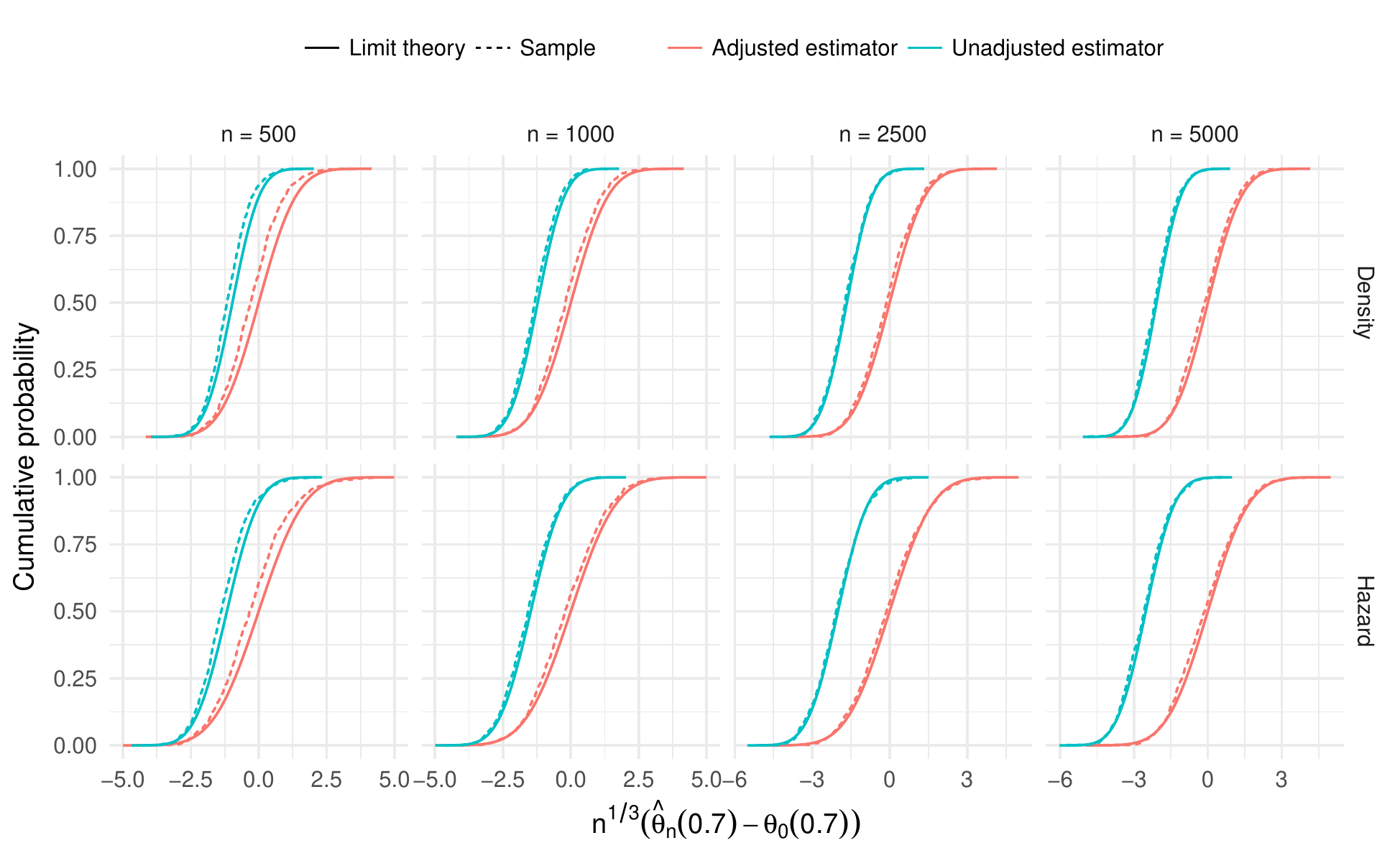}
\caption{Sampling distribution over 1000 simulations of the monotone density and hazard estimators at $x = 0.7$ and the corresponding theoretical scaled Chernoff limit distribution. These figures are based on the scenario wherein $T$ and $C$ depend on $W$.}
\label{fig:dens.haz.dists.n}
\end{figure}

\newpage
\section*{Appendix}
We begin by stating two lemmas we will require -- proofs are provided in Supplementary Material. The first lemma is a generalization of the switch relation first introduced in \cite{groeneboom1985density} and discussed in detail on page 296 of \cite{van1996weak}, on page 64 of \cite{van2006estimating}, in \cite{groene2014shape} and in  \cite{balabdaoui2011grenander}. For brevity, throughout, we will  refer to \cite{van1996weak} as VW.

\begin{lemma}\label{lemma:modified_switch}
Let $\Phi$ and $\Gamma$ be functions from a closed interval $I\subseteq \mathbb{R}$  to $[a, b]\subset \mathbb{R}$, where $\Phi$ is nondecreasing and c\`{a}dl\`{a}g , $\Gamma$ and $\Psi:=\Gamma \circ \Phi^-$ are lower semi-continuous, and $\{a,b\}\subset \Phi(I)$. Let $\psi$ be the left derivative of the GCM $\bar{\Psi}$ of $\Psi$ and $\theta := \psi \circ \Phi$. Then, for any $c \in \d{R}$ and $x \in I$ with $\Phi(x) \in (a, b)$, $\theta(x) > c$ if and only if $\sup \argmax_{v \in I^*} \left\{ c \Phi(v) - \Gamma(v) \right\} < \Phi^- (\Phi(x))$, where $I^* := I \cap \Phi^-([a,b]) = \{x \in I : x= \Phi^-(u), u \in [a,b]\}$.
\end{lemma} The switch relation requires $\Psi$ to be lower semi-continuous. If $\Psi_n$ is not so, it can be replaced by its greatest lower semi-continuous minorant. As argued in \cite{van2006estimating}, this only possibly changes the GCM at the endpoints of the interval and has no effect on asymptotic properties (e.g., weak convergence of $\Psi_n$). In the second lemma, pointwise and uniform finite-sample tail bounds are provided. These tail bounds are not sharp but suffice to derive consistency results in broad generality. Simpler tail bounds can be derived in the absence of a transformation $\Phi_0$.

\begin{lemma}\label{tail_bound} Suppose that $|\Phi_0^{-}(u) - x|\leq \gamma(|u - \Phi_0(x)|)$ for all $u \in J_0$ and a continuous, strictly increasing function $\gamma:\mathbb{R}^+\rightarrow\mathbb{R}^+$ with $\gamma(0)=0$, and that $\Phi_0$ is strictly increasing and continuous on $[x - \delta, x + \delta] \subset \Phi_0^{-1}(J_n)$. Let $\omega:\mathbb{R}^+\rightarrow\mathbb{R}^+$ be a non-decreasing function satisfying $\lim_{z\downarrow 0} \omega(z) = \omega(0)=0$, and suppose $|\theta_0(u) - \theta_0(x)| \leq \omega(|u-x)|)$.  Define $c(\delta, \eta) := \gamma^{-1}\left( \delta \wedge \omega^-(\eta) \right)$ and $r(\delta, \eta) := \int_0^{c(\delta, \eta) / 2} \left[ \eta -\omega(\gamma(u)) \right]du$. Then, for any $\eta>0$ and $x \in I$ such that $\Phi_n(x), \Phi_0(x) \in (0, u_n)$,
\begin{align*} 
P_0\left(|\theta_n(x) - \theta_0(x)| > \eta \right) \leq  P_0\left(A_{n,1}(\eta) > r(\delta/2, \eta/2) \right)+ P_0\left(A_{n,2} \geq  c(\delta / 2, \eta/ 2)  \right)
\end{align*}
with $A_{n,1}(\eta):=2\|\Gamma_n - \Gamma_0\|_{\infty, I_n} + (2|\theta_0(x)| + \eta)\|\Phi_n - \Phi_0\|_{\infty, I_n}$ and $A_{n,2}:=4\|\Phi_n - \Phi_0\|_{\infty, [x-\delta, x + \delta]}$.
If $\Phi_0$ is strictly increasing and continuous on $I$, $|\Phi_0^{-1}(u) - \Phi_0^{-1}(v)| \leq \gamma(|u - v|)$ for all $u,v \in J_0$, and $|\theta_0(u) - \theta_0(v)| \leq \omega(|u-v)|)$ for all $u, v\in I$, then, for any $\eta, \beta > 0$,
\begin{align*}
 P_0\left(\| \theta_n - \theta_0\|_{\infty, I_{n,\beta}} > \eta \right) \leq  P_0\left(B_{n,1}(\eta) >r(\beta/2, \eta/2) \right)+ P_0\left(B_{n,2} \geq  c(\beta/2, \eta/2)\right)
\end{align*} with $B_{n,1}(\eta):=2\|\Gamma_n - \Gamma_0\|_{\infty, I_n} + (2\|\theta_0\|_{\infty, I} + \eta)\|\Phi_n - \Phi_0\|_{\infty, I_n}$ and $B_{n,2}:=4\|\Phi_n - \Phi_0\|_{\infty, I}$.
\end{lemma}


\begin{proof}[\bfseries Proof of Theorems \ref{thm:consistency} and \ref{thm:rates}]
For part 1 of Theorem~\ref{thm:consistency} and parts 1 and 2 of Theorem~\ref{thm:rates}, we use the pointwise tail bound in Lemma~\ref{tail_bound} with different choices of $\omega$ and $\gamma$. Since $u_n \inproblow u_0$,  $\Phi_0(x) \in (0, u_0)$ and $\Phi_n(x) \inproblow \Phi_0(x)$, with probability tending to one, $\Phi_n(x)\in(0, u_n)$ and $[x-\delta', x + \delta'] \subset \Phi_0^{-1}(J_n)$ for some $\delta' > 0$. For part 2 of Theorem~\ref{thm:consistency} and part 3 of Theorem~\ref{thm:rates}, we use instead the uniform tail bound. We note that for any $\delta, \eta > 0$, $c(\delta, \eta) > 0$ and $r(\delta, \eta) > 0$, which we show in the proof of Lemma~\ref{tail_bound}.

For part 1 of Theorem \ref{thm:consistency}, we take $\omega(v) := [\theta_0(x+v) - \theta_0(x)]\vee[\theta_0(x) - \theta_0(x - v)]$, which is a valid choice since $\theta_0$ is non-decreasing and continuous at $x$. Since $\Phi_0$ is continuous and strictly increasing in a neighborhood of $x$, so is $\Phi_0^-$. Since $J_0$ is bounded, such an invertible $\gamma$ exists. By the pointwise tail bound in Lemma~\ref{tail_bound}, both $A_{n,1}(\eta)$ and $A_{n,2}$  are $\fasterthan(1)$ by assumption, and the result follows.

For part 1 of Theorem \ref{thm:rates}, we consider the pointwise tail bound with $\eta=\eta_n:= \eta_0 r_n^{-\alpha_1\alpha_2 / (\alpha_1\alpha_2 + 1)}$. By assumption, $\omega(v) := K_1(x) v^{\alpha_1}$ and $\gamma(v) := K_2(x) v^{\alpha_2}$ are valid choices. Since $\delta > 0$ and $\eta_n \to 0$, $c(\delta/2, \eta_n/2) \sim \eta_n^{1/(\alpha_1\alpha_2)} =\eta_0 r_n^{-1/(1+\alpha_1\alpha_2)}$ for large $n$. Thus, the second term of the upper bound is $P_0(4r_n\|\Phi_n - \Phi_0\|_{\infty, [x-\delta, x + \delta]} \geq \eta^2_0\eta_n^{-1})$ for large $n$. Since $r_n\|\Phi_n - \Phi_0\|_{\infty, [x-\delta, x + \delta]} = \bounded(1)$ and $\eta_n=o(1)$ by assumption, this term tends to zero. Because $r(\delta/2, \eta_n/2) \sim r_n^{-1}$, the first term of the upper bound is bounded for any $\eta_0 > 0$ as $\|\Gamma_n - \Gamma_0\|_{\infty, I_n}$ and $\|\Phi_n - \Phi_0\|_{\infty, I_n}$ are both $\bounded(r_n^{-1})$ by assumption.

For part 2 of Theorem \ref{thm:rates}, we take $\omega(v) := 0$ for $v \leq \delta$ since $\theta_0$ is constant on $[x-\delta, x + \delta]$. As before, since $\Phi_0$ is continuous and strictly increasing in a neighborhood of $x$, such an invertible $\gamma$ exists. Letting $\eta=\eta_n := \eta_0 r_n^{-1}$, we have $c(\delta/2, \eta_n/2) = \gamma^{-1}(\delta/2) > 0$ for all $n$, so the second term of the upper bound tends to zero. Since $r(\delta/2, \eta_n/2) \sim r_n^{-1}$, the first term of the upper bound is bounded.

For part 2 of Theorem~\ref{thm:consistency}, since it is uniformly continuous, $\theta_0$ admits a uniform modulus of continuity, which we choose as $\omega$. Since $\Phi_0$ is strictly increasing and continuous, $\Phi_0^-$ is well-defined and continuous. Since $J_0$ is compact, $\Phi_0^{-1}$ is uniformly continuous and possesses a continuous and invertible uniform modulus of continuity, which we choose as $\gamma$. Thus, $c(\beta/2, \eta/2) > 0$ and $r(\beta/2, \eta/2)>0$ for any $\beta, \eta > 0$, and so, both terms in the uniform upper bound tend to zero by assumption.


For part 3 of Theorem~\ref{thm:rates}, we consider the uniform tail bound with $\eta=\eta_n:=\eta_0 r_n^{-\alpha_1\alpha_2 / (\alpha_1\alpha_2 + 1)}$. By assumption, $\omega(v) := K_1 v^{\alpha_1}$ and $\gamma(v) = K_2 v^{\alpha_2}$ are valid choices. With probability tending to one, $\beta_n > r_n^{-1/(1+\alpha_1\alpha_2)}$, $c(\beta_n/2, \eta_n/2) \sim r_n^{-1/(1+\alpha_1\alpha_2)}$, and so, $r_n c(\beta_n/2, \eta_n/2)$ tends to $+\infty$ in probability. Thus, the second term in the upper bound tends to zero. Since $r(\beta_n/2, \eta_n/2) \sim r_n^{-1}$, the first term in the upper bound is bounded for any $\eta_0>0$.  
\end{proof}

\begin{proof}[\bfseries Proof of Theorem \ref{thm:conv_in_dist}]
We note that $r_n\left[\theta_n(x) - \theta_0(x)\right] > \eta$ if and only if $\theta_n(x) > \theta_0(x) + r_n^{-1} \eta$, which, by Lemma \ref{lemma:modified_switch}, occurs if and only if $\sup\argmax_{v \in I_n}\left\{ [ \theta_0(x) + r_n^{-1} \eta ]\Phi_n(v) -\Gamma_n(v)\right\} < \Phi_n^-( \Phi_n(x))$. The latter event occurs if and only if 
\[\sup\argmax_{v \in c_n(I_n - x)}\left\{ \left[ \theta_0(x) + c_n^{-\alpha} \eta \right]\Phi_n(x + c_n^{-1} v) -\Gamma_n(x + c_n^{-1} v)\right\} < c_n\left[\Phi_n^-( \Phi_n(x))  - x\right].\]
Since adding terms not depending on $v$ and scaling by constants does not affect the value of the maximizer,  the left-hand side of the inequality above equals
$
\sup\argmax_{v\in c_n(I_n-x)}\left\{H_{n,x,\eta}(v) + R_{n}(v)\right\}
$
for $H_{n,x, \eta}(v) := -W_{n,x}(v) + \left[\eta \Phi_0'(x) \right] v - \left[\Phi_0'(x) \pi_0(x)(\alpha + 1)^{-1} \right] |v|^{\alpha + 1}$ and $R_{n}(v) = R_{n,1}(v)+ R_{n,2}(v)- R_{n,3}(v)$, with
\begin{align*}
R_{n, 1}(v)\ &:=\ c_n \eta \left[\Phi_n(x + c_n^{-1} v) - \Phi_0(x +  c_n^{-1}v) \right];\\
R_{n,2}(v)\ &:=\ c_n\eta \left[ \Phi_0(x + c_n^{-1}v) - \Phi_0(x) - \Phi_0'(x) (c_n^{-1}v ) \right];\\
R_{n,3}(v)\ &:=\ c_n^{\alpha + 1} \left[ M_{0,x}(c_n^{-1}v)- \Phi_0'(x)\pi_0(x)(\alpha + 1)^{-1}|c_n^{-1} v|^{\alpha + 1}\right]\ ,
\end{align*}
where we define $M_{0,x}(u):=[\Gamma_0(x + u) - \theta_0(x) \Phi_0(x + u)] - [\Gamma_0(x) - \theta_0(x) \Phi_0(x)]$. By Slutsky's Theorem, $\left\{  H_{n,x, \eta}(v) : |v| \leq M\right\}$ converges weakly to $\{ H_{x,\eta}(v) : |v| \leq M\}$ in $\ell^{\infty}[-M,M]$ for every $M > 0$ for $H_{x,\eta}(v) := -W_{x}(v) + \left[\eta \Phi_0'(x) \right] v - \left[\pi_0(x) \Phi_0'(x)(\alpha + 1)^{-1}\right] |v|^{\alpha + 1}.$ By the uniform consistency of $\Phi_n$ to $\Phi_0$ at rate faster than $c_n^{-1}$ in a neighborhood of $x$,  $\sup_{|v| \leq M} |R_{n,1}(v)| = \fasterthan(1)$ for all $M > 0$. Continuous differentiability of $\Phi_0$ at $x$ gives $\sup_{|v| \leq M} |R_{n,2}(v)| = \fasterthandet(1)$ for all $M > 0$. For $R_{n,3}$, clearly, $M_{0,x}(0) = 0$, and $M_{0,x}'(u) = \Phi_0'(x + u) [\theta_0(x + u) -\theta_0(x)]$ for $u$ in a neighborhood of $0$, so that $M_{0,x}'(0) = 0$. Furthermore, $|M_{0,x}'(u)| / |u|^\alpha \to \Phi_0'(x) \pi_0(x)$ as $u \to 0$ by the assumed order of growth of $\theta_0$ and continuity of $\Phi_0'$ at $x$. Therefore, by L'H\^{o}pital's rule, $\lim_{u \to 0} M_{0,x}(u) / |u|^{\alpha + 1} = (\alpha + 1)^{-1}\lim_{u \to 0} \n{sign}(u)M_{0,x}'(u) / |u|^{\alpha} = (\alpha + 1)^{-1}\lim_{u \to 0}|M_{0,x}'(u)| / |u|^{\alpha} =  (\alpha + 1)^{-1}\Phi_0'(x) \pi_0(x)$. It follows that $\sup_{|v| \leq M} |R_{n,3}(v)| = \fasterthandet(1)$ for all $M > 0$. In view of these findings, we have that $\{H_{n,x, \eta}(v) + R_n(v) : |v| \leq M\}$ converges weakly to $\{ H_{x,\eta}(v) : |v| \leq M\}$ for every $M > 0$. Since there is a neighborhood of $x$ in which $\Phi_0$ is strictly increasing and $\Phi_n^-$ is uniformly consistent, $c_n(I_n - x) \to \d{R}$ in probability.  Therefore, the argmax continuous mapping theorem (Theorem 3.2.2 of VW) implies that
\[\hat{v}_{n}(x,\eta) := \sup \argmax_{v \in c_n(I_n - x)} \{  H_{n, x, \eta}(v) + R_{n}(v)\} \indist \sup\argmax_{v \in \d{R}} \{H_{x, \eta}(v)\} =: \hat{v}(x,\eta)\]
as long as $\hat{v}_{n}(x, \eta) = \bounded(1)$, where we have used the assumptions that $ \sup\argmax_{v \in \d{R}}\{H_{x,\eta}(v)\}$ is bounded in probability and that $W_{x}$ is almost surely lower semi-continuous. Lemma 3 of the Supplementary Material establishes that $\hat{v}_{n}(x, \eta) = \bounded(1)$ under the stated conditions. Since $c_n\sup_{|u - x| \leq \delta} |\Phi_n(u) - \Phi_0(u)| =\fasterthan(1)$ by assumption and $\Phi_0$ is continuously differentiable at $x$ with positive derivative, $c_n[\Phi_n^-(\Phi_n(x)) - x] =\fasterthan(1)$. Thus, we find that
\[P_0\left( r_n\left[\theta_n(x) - \theta_0(x)\right] > \eta \right) = P_0\left( \hat{v}_n(x,\eta) -c_n\left[\Phi_n^-(\Phi_n(x)) - x\right] < 0 \right) \longrightarrow P_0\left(\hat{v}(x,\eta) < 0\right)\ .\]
We note that 
\[ \hat{v}(x, \eta) =  \sup\argmax_{v \in \d{R}} \left\{-W_{x}(v) - \left[\frac{\pi_0(x) \Phi_0'(x)}{\alpha + 1} \right] |v|^{\alpha + 1} - \left[-\eta \Phi_0'(x) \right] v \right\}\ .\]
Thus, by the standard switch relation (e.g.\ Lemma 3.2 of \citealp{groene2014shape}), $\hat{v}(x, \eta) < 0$ if and only if
\[  \Phi_0'(x)^{-1} \partial_- \n{GCM}_{\d{R}}\left\{ W_{x}(v) + \left[\frac{\pi_0(x) \Phi_0'(x)}{\alpha + 1} \right] |v|^{\alpha + 1} \right\}(0) > \eta\ .\]
where we have again used that $W_x$ is almost surely lower semi-continuous. Therefore, 
\[ P_0\left(\hat{v}(x,\eta) < 0\right) = P_0\left( \Phi_0'(x)^{-1} \partial_- \n{GCM}_{\d{R}}\left\{ W_{x}(v) + \left[\frac{ \pi_0(x) \Phi_0'(x)}{\alpha + 1}\right] |v|^{\alpha + 1} \right\}(0) > \eta\right) .\]
The result follows from the Portmanteau Theorem.

If $W_x$ has stationary increments and $\alpha = 1$ so that $\pi_0(x) = \theta_0'(x)$, then
\begin{align*}
 P_0\left(\hat{v}(x,\eta) < 0\right)  &=P_0\left(-\theta_0'(x) \argmin_{u \in \d{R}} \left\{W_{x}(u+\eta/\theta_0'(x)) + \tfrac{1}{2}\theta_0'(x) \Phi_0'(x)  u^2 \right\} > \eta \right)\\
 &=P_0\left(-\theta_0'(x) \argmin_{u \in \d{R}} \left\{W_{x}(u) + \tfrac{1}{2}\theta_0'(x) \Phi_0'(x)  u^2 \right\} > \eta \right)
\end{align*}
for each $\eta$, and the result again follows by the Portmanteau Theorem. Finally, if $W_{x} =[\kappa_0(x)]^{1/2} W_0$ for $W_0$ a standard two-sided Brownian motion, a standard argument (see Problem 3.2.5 of VW) shows that 
\[ \argmin_{u \in \d{R}} \left\{W_{x}(u) + \tfrac{1}{2}\theta_0'(x) \Phi_0'(x)  u^2 \right\} \ \stackrel{d}{=}\ \left\{\frac{2[\kappa_0(x)]^{1/2}}{\theta_0'(x) \Phi_0'(x)}\right\}^{2/3} \argmin_{u \in \d{R}} \left\{W_0(u) + u^2 \right\}.\qedhere\]
%
\end{proof}


\begin{proof}[\bfseries Proof of Theorem \ref{dist_asy_lin}]
We use Theorems 2.11.22 and 2.11.23 of VW to show weak convergence of $W_{n,x}$ to $[\kappa_0(x) ]^{1/2}W_0$. In their notation, $f_{n,u} = n^{1/6} g_{x, un^{-1/3}}$ and $\s{F}_{n,M} = \{ f_{n,u} : |u| \leq M\} = n^{1/6}\s{G}_{x,Mn^{-1/3}}$ with envelope $F_{n,M} = n^{1/6}G_{Mn^{-1/3}}$. Thus, we have that $P_0 F_{n,M}^2 = n^{1/3} P_0 G_{x,Mn^{-1/3}}^2 = n^{1/3} \boundeddet(Mn^{-1/3}) = \boundeddet(1)$ for each $M>0$ by (B2). For any $\epsilon > 0$ and $\eta > 0$, $R^{-1} P_0 G_{x,R}^2 \{ G_{x,R} > \eta (MR)^{-1}\} < M\epsilon$ for all $R$ small enough, so that after some rearrangement, for all $n$ large enough,
\[ P_0 F_{n,M}^2 \{ F_{n,M} > \eta n^{1/2}\}  < \epsilon\ . \]
In the case of Theorem 2.11.23, we will use the first possibility of (B1a) to establish the convergence of the bracketing entropy integral:
\begin{align*}
&\int_{0}^{\delta_n} \left[\log N_{[]}(\varepsilon \|F_{n,M}\|_{P_0,2}, \s{F}_{n,M}, L_2(P_0))\right]^{1/2} d\varepsilon\\
&\hspace{0in}= \int_{0}^{\delta_n} \left[\log N_{[]}(\varepsilon n^{1/6} \|G_{x,Mn^{-1/3}}\|_{P_0,2}, n^{1/6} \s{G}_{x,Mn^{-1/3}}, L_2(P_0))\right]^{1/2} d\varepsilon \\
&\hspace{0in}=  \int_{0}^{\delta_n} \left[\log N_{[]}(\varepsilon \|G_{x,Mn^{-1/3}}\|_{P_0,2}, \s{G}_{x,Mn^{-1/3}}, L_2(P_0))\right]^{1/2} d\varepsilon\ =\ \boundeddet\left(\int_{0}^{\delta_n} \varepsilon^V d\varepsilon\right)\ =\ \boundeddet\left(\frac{\delta_n^{V+1}}{V+1} \right)\ \rightarrow\ 0
\end{align*}
for all $\delta_n \to 0$. The calculation for the uniform entropy integral using the second possibility (B1b) to establish Theorem 2.11.22 is identical.

We now show that (B3) implies that, for all $\delta$ small enough, $\sup_{|u - v| < \delta} P_0(g_{x,u}- g_{x,v})^2 =\boundeddet(\delta)$ and that $\alpha^{-1} [P_0 (g_{x, \alpha u} g_{x, \alpha v}) - P_0 g_{x, \alpha u}P_0 g_{x, \alpha v}] \to \sigma^2(u,v)\kappa_0(x)$
 as $\alpha \to 0$, where $\sigma^2(u,v) :=  (u \wedge v) - u I_{(-\infty,0)}(u) - v I_{(-\infty,0)}(v)$ is the covariance of a two-sided Brownian motion. Then we will have that
\[ \sup_{|s - t| < \delta_n} P_0(f_{n,s} - f_{n,t})^2 = n^{1/3} \sup_{|u - v| < \delta_n n^{-1/3}} P_0(g_{x,u}- g_{x,v})^2 = \boundeddet\left( n^{1/3} \delta_nn^{-1/3}\right) = \boundeddet\left(\delta_n\right) \to 0\]
for all $\delta_n \to 0$ and that $P_0 f_{n,u} f_{n,v} - P_0 f_{n,u} P_0 f_{n,v} = n^{1/3} P_0 g_{x, un^{-1/3}} g_{x, vn^{-1/3}} - n^{1/3}  P_0 g_{x, un^{-1/3}}P_0 g_{x, vn^{-1/3}}$ tends to $\sigma^2(u,v) \kappa_0(x)$; both of these statements are conditions of  Theorems 2.11.22 and 2.11.23 of VW. 

Writing $s := x + u$ and $t := x + v$, we can show that $P_0 (g_{x,u} - g_{x,v})^2=  \Sigma_0( s, s)   -2 \Sigma_0( s, t)  + \Sigma_0( t, t)$.
Hence, for the first claim, it is sufficient to show that $|\Sigma_0( s, s)   - \Sigma_0( s, t)| =\boundeddet( |s - t|)$ for all $s, t$ in a neighborhood of $x$. By assumption, $\Sigma_0^*$ is continuously differentiable at $(x,x)$, which implies that  $|\Sigma_0^*(s, s)   - \Sigma_0^*( s,t)|=\boundeddet( |s -t|)$ for $s, t$ in a neighborhood of $x$. We can decompose $ \iint_{-\infty}^{s \wedge t} A_0(s, t,u, w) H_0(du, w) Q_0(dw)$ as $\bar{\Sigma}_0(s,t)+\tilde{\Sigma}_0(s,t)$, where we set
\[\bar{\Sigma}_0(s,t):=\iint_{-\infty}^{x} A_0(s, t,u, w) H_0(du, w)Q_0(dw),\ \tilde{\Sigma}_0(s,t):=\iint_{x}^{s \wedge t} A_0(s, t, u, w) H_0(du, w) Q_0(dw)\ .\] By (B3b), $\bar{\Sigma}_0$ is continuously differentiable at $(x,x)$, which implies that  $|\bar{\Sigma}_0(s, s)   - \bar{\Sigma}_0( s,t)|=\boundeddet( |s -t|)$ for $s, t$ in a neighborhood of $x$. For $\tilde\Sigma_0$, we have that $|\tilde\Sigma_0(s,t) - \tilde\Sigma_0(s,s)|$ is bounded above by 
\[\iint_{x}^s \left|A_0(s, s, u, w) - A_0(s, t, u, w)\right| H_0(du, w)Q_0(dw)  + \iint_{s}^{s\wedge t}  \left| A_0(s, t, u, w) \right| H_0(du, w)Q_0(dw)\ .\]
Continuous differentiability of $A_0$ around $(x,x)$ implies that the first summand is bounded above by
\[ |s - t| \iint_{x}^s \sup_{s,t \in B_{\delta}(x)}|A_0'(s, t, u, w)| H_0(du, w) Q_0(dw)\] for $s,t$ close enough to $x$,
which is bounded up to a constant by $|s - t|$ by assumption. Boundedness of $A_0$ and continuity of $H_0$  around $x$ for all $w$ yields the same for the second term.

For the second claim, we first note that the contribution of $\bar{\Sigma}_0$ to $\tfrac{1}{\alpha} [P_0 (g_{x, \alpha u} g_{x, \alpha v}) - P_0 g_{x, \alpha u}P_0 g_{x, \alpha v}]=\tfrac{1}{\alpha} \left[\Sigma_0(x+ \alpha u , x+ \alpha v) - \Sigma_0(x + \alpha u, x) - \Sigma_0(x, x +\alpha v) + \Sigma_0(x,x) \right]$ is 
\begin{align*}
&\tfrac{1}{\alpha}\left[\bar{\Sigma}_0(x + \alpha u , x + \alpha v) - \bar{\Sigma}_0(x,x)\right] -\tfrac{1}{\alpha}\left[ \bar{\Sigma}_0(x + \alpha u , x) - \bar{\Sigma}_0(x,x)\right] - \tfrac{1}{\alpha}\left[ \bar{\Sigma}_0(x, x + \alpha v) - \bar{\Sigma}_0(x,x)\right],
\end{align*}
which, due to the differentiability of $\bar{\Sigma}_0$, tends to $(u + v) \bar{\Sigma}_0'(x, x)-u \bar{\Sigma}_0'(x, x) - v\bar{\Sigma}_0'(x, x)=0$ as $\alpha \to 0$. Similarly, $\Sigma_0^*$ does not contribute to the limit. The contribution of $\tilde{\Sigma}_0$ therefore determines the limit entirely.  For any fixed $r$ and $w$, we note that
\begin{align*}
\frac{1}{\alpha} \int_x^{x + \alpha r} A_0(x, x, u, w) H_0(du,w)  \longrightarrow r A_0(x, x, x, w) H'_0(x,w) 
\end{align*}
as $\alpha\rightarrow 0$ by the continuous differentiability of $u \mapsto H_0(u, w)$ at $u=x$ and the continuity of $u \mapsto A_0(x,x,u,w)$. Since the continuity of $x\mapsto A_0(x,x,x,w) H'_0(x,w)$ is uniform in $w$ and these functions are $Q_0$-integrable, by the Dominated Convergence Theorem,  for any fixed $r$, we have that
\[ \frac{1}{\alpha} \iint_x^{x + \alpha r} A_0(x, x, u, w) H_0(du,w) Q_0(dw)  \longrightarrow r \int  A_0(x, x, x, w) H_0'(x,w)  Q_0(dw)\ .\] We then find that $\tfrac{1}{\alpha}[\tilde\Sigma_0(x+ \alpha u , x+ \alpha v) - \tilde\Sigma_0(x + \alpha u, x) - \tilde\Sigma_0(x, x +\alpha v) + \tilde\Sigma_0(x,x)]$ can be written, up to a remainder term tending to zero as $\alpha\to 0$, as
\begin{align*}
\frac{1}{\alpha}\iint \left[I_{(x,x+\alpha(u\wedge v))}(y)-I_{(-\infty,0)}(u)I_{(x,x+\alpha u)}(y)-I_{(-\infty,0)}(v)I_{(x,x+\alpha v)}(y)\right]A_0(x,x,y,w)H_0(dy,w)Q_0(dw)
\end{align*} limiting to $\left[  u \wedge v  - I_{(-\infty,0)}(u)u - I_{(-\infty,0)}(v) v\right] \int  A_0(x, x, x, w) H'_0(x,w) Q_0(dw)$,
the claimed covariance. The remainder term we left out can be expressed as \begin{align*}
&\frac{1}{\alpha} \iint_{x}^{x + \alpha (u \wedge v)}\left[ A_0(x+\alpha u, x + \alpha v, y, w)  - A_0(x, x, y, w) \right]H_0(dy,w) Q_0(dw) \\
&\hspace{0.5in}- I_{(-\infty,0)}(u)\frac{1}{\alpha} \iint_{x}^{x + \alpha u}\left[ A_0(x+\alpha u, x , y, w) - A_0(x, x, y, w) \right]H_0(dy,w) Q_0(dw)\\
&\hspace{0.5in}-I_{(-\infty,0)}(v)\frac{1}{\alpha}\iint_{x}^{x + \alpha v}\left[ A_0(x, x + \alpha v, y, w) -  A_0(x, x, y, w) \right] H_0(dy,w)Q_0(dw)\ .
\end{align*}
For $\alpha$ small enough the absolute value of each inner difference is bounded by $\alpha (|u |\vee |v| )| A_0'(x,x,y, w)|$. Since $y\mapsto A_0'(x,x,y, w)$ is continuous and $y\mapsto H_0(y,w)$ is differentiable in a neighborhood of $x$ uniformly in $w$, for $\alpha$ small enough,  the absolute value of the remainder is bounded up to a constant by
\begin{align*}
\iint \left[I_{(x,x+\alpha(u\wedge v))}(y)+I_{(-\infty,0)}(u)I_{(x,x+\alpha u)}(y)+I_{(-\infty,0)}(v)I_{(x,x+\alpha v)}(y)\right]H_0(dy,w)Q_0(dw)\ .
\end{align*} 
Since $y\mapsto H'_0(y,w)$ is bounded near $x$ uniformly in $w$, this bound tends to zero as $\alpha \to 0$. This, in addition to condition (B4), proves (A1). Since $\theta_0'(x)$ and $\Phi_0'(x)$ are assumed positive, (A2) is also satisfied. For (A3), we note that 
\[ E_0\left[\sup_{|u| \leq \delta n^{1/3} } |\d{G}_n f_{n,u}|\right] = n^{1/6} E_0\left[\sup_{|u| \leq  \delta n^{1/3}} |\d{G}_n  g_{x, un^{-1/3}}|\right] = \boundeddet\left(\delta^{1/2}n^{1/6}\right)\]
 for all $n$ large enough is also implied  by assumption (B1) and Theorems 2.14.1 and 2.14.2 of VW. The remainder term satisfies (A3) by condition (B5).\qedhere

%

\end{proof}


\begin{proof}[\bfseries Regularity conditions and proof of Theorem \ref{thm:transform}]
Regularity conditions for Theorem \ref{thm:transform} include that $(s,t, u, z) \mapsto M_{s,0}^{(1)}(u,z) M_{t,0}^{(1)}(u,z)$ and $(s,t, u, z) \mapsto L_{s,0}^{(1)}(u, z) L_{t,0}^{(1)}(u, z)$  satisfy (B3b) and (B3c), and that the following maps are continuously differentiable in $(s, t)$ in a neighborhood of $(x,x)$:
\begin{align*}
&(s,t) \mapsto E_0\left[I_{[0,s]}(U) M_{s,0}^{(1)}(O)M_{t, 0}^{(2)}(O)\right],\ (s,t) \mapsto E_0\left[ I_{[0,s]}(U) M_{s,0}^{(1)}(O)D_{t, 0}^{(2)}(O)\Phi_0'(U)\right],\\
&(s,t) \mapsto E_0\left[ I_{[0,s]}(U) M_{s,0}^{(1)}(O) L_{t,0}^{(2)}(O) \Phi_0'(U)\right],\ (s,t) \mapsto E_0\left[ I_{[0,s]}(U)  L_{s,0}^{(1)}(O)D_{t, 0}^{(2)}(O) \right],\\
&(s,t) \mapsto E_0\left[ I_{[0,s]}(U)  L_{s,0}^{(1)}(O)L_{t,0}^{(2)}(O) \right],\ (s,t) \mapsto E_0\left[ M_{s, 0}^{(2)}(O)M_{t, 0}^{(2)}(O) \right],\ (s,t) \mapsto E_0\left[ D_{s, 0}^{(2)}(O)D_{t, 0}^{(2)}(O) \right],\\
&(s,t) \mapsto E_0\left[ D_{s, 0}^{(2)}(O)L_{t, 0}^{(2)}(O)\right],\ (s,t) \mapsto E_0\left[ L_{s, 0}^{(2)}(O)L_{t, 0}^{(2)}(O)\right].
\end{align*}

We first examine the covariance arising from the use of $\Theta_n$ and the identity transformation. Writing $H_0:(u, z)\mapsto P_0(U \leq u \mid Z=z)$, we have that $\Sigma_0(s,t) = P_0(M_{s,0}^*M_{t,0}^*)$ is equal to
\begin{align*}
& \int \left[I_{[0,s]}(u) M_{s,0}^{(1)}(u,z) + M_{s, 0}^{(2)}(u,z) \right] \left[I_{[0,t]}(u) M_{t,0}^{(1)}(u,z) + M_{t, 0}^{(2)}(u,z) \right]P_0(du, dz)\\
&= \iint_0^{s \wedge t} M_{s,0}^{(1)}(u,z) M_{t,0}^{(1)}(u,z) H_0(du, z)Q_0(dz)\\
&\hspace{.15in}+ \int \left[I_{[0,s]}(u) M_{s,0}^{(1)}(u,z)M_{t, 0}^{(2)}(u,z) + I_{[0,t]}(u) M_{t,0}^{(1)}(u,z)M_{s, 0}^{(2)}(u,z) + M_{t, 0}^{(2)}(u,z)M_{s, 0}^{(2)}(u,z) \right]P_0(du, dz)\ .
\end{align*}
By assumption, the second summand plays the role of $\Sigma_0^*(s,t)$ and satisfies (B3a). The first summand satisfies (B3b) and (B3c) with $A_0(s,t, u, z) = M_{s,0}^{(1)}(u,z) M_{t,0}^{(1)}(u,z)$ by assumption, and $H_0(u, z)$ satisfies (B3d) with $H_0'(u, z) = h_0(u | z)$ equal to the conditional density of $U$ given $Z=z$. Therefore, the scale factor for the Chernoff distribution in Theorem 4 is equal to $[ 4\theta'_0(x) \kappa_0(x)]^{1/3}$, where \[\kappa_0(x) = \int \left[M_{x,0}^{(1)}(x,z)\right]^2 h_0(x \mid z) Q_{Z,0}(dz).\]

We then examine the covariance arising from the use of $\Gamma_n$ and transformation $\Phi_n$. Using integration by parts, we find that $D_{s,0}^*(o) - \theta_0(x) L_{s,0}^*(o)$ is equal to $I_{[0,s]}(u)\Upsilon_{1,s,x}(u,z)+\Upsilon_{2,s,x}(u,z)$, where \begin{align*}
\Upsilon_{1,s,x}:\ &(u,z)\mapsto M_{s,0}^{(1)}(u, z)\Phi_0'(u) -\int_{u}^s  L_{v,0}^{(1)}(u, z) \theta_0(dv) + \left[\theta_0(s) -\theta_0(x)\right]L_{s,0}^{(1)}(u,z)\ ,\\
\Upsilon_{2,s,x}:\ &(u,z)\mapsto D_{s, 0}^{(2)}(u, z) -\int_0^s  L_{v,0}^{(2)}(u, z) \theta_0(dv) + \left[\theta_0(s) -\theta_0(x)\right]L_{s,0}^{(2)}(u, z)\ .
\end{align*}  The covariance $\Sigma_0(s,t)=P_0 [D_{s,0}^*- \theta_0(x) L_{s,0}^* ] [D_{t,0}^*- \theta_0(x) L_{t,0}^* ] $ can then be written as the sum $\Sigma_{0,1}(s,t)+\Sigma_{0,2}(s,t)+\Sigma_{0,3}(s,t)+\Sigma_{0,4}(s,t)$ of all cross-product terms. The sum $\Sigma_{0,2}+\Sigma_{0,3}+\Sigma_{0,4}$ constitutes $\Sigma_0^*$, where the summands are defined pointwise as $\Sigma_{0,2}(s,t)=\iint I_{[0,s]}(u)\Upsilon_{1,s,x}(u,z)\Upsilon_{2,t,x}(u,z)P_0(du,dz)$, $\Sigma_{0,3}(s,t)=\Sigma_{0,2}(t,s)$ and $\Sigma_{0,4}(s,t)= \iint \Upsilon_{2,s,x}(u,z)\Upsilon_{2,t,x}(u,z)P_0(du,dz)$.
By assumption, each of these expressions is continuously differentiable in $(s,t)$ in a neighborhood of $(x,x)$. Finally, we have $\Sigma_{0,1}(s,t)=\iint I_{[0,s\wedge t]}(u) \Upsilon_{1,s,x}(u,z)\Upsilon_{1,t,x}(u,z)H_0(du,z)Q_{Z,0}(dz)$. The product $\Upsilon_{1,s,x}(u,z)\Upsilon_{1,t,x}(u,z)$ forms $A_0(s,t,u,z)$, which satisfies (B3b) and (B3c) by assumption. Hence, in this case, the scale parameter is $[4 \theta_0'(x) \kappa_0^*(x)/\Phi'_0(x)^2]^{1/3}$ in view of Theorem \ref{dist_asy_lin}, where
\[\kappa_0^*(x)= \int \left[M_{x,0}^{(1)}(x,z) \Phi_0'(x)\right]^2h_0(x \mid z) Q_{z,0}(dz)=\Phi_0'(x)^2\kappa_0(x)\ .\] Thus, the scale factor obtained coincides with that obtained with $\Theta_n$ and identity transformation.
\end{proof}

%

%

\begin{proof}[\bfseries Proof of Theorem~\ref{local_empirical}]
Let $\s{F}_{x,n,\delta} := \{ n^{1/6} g_{x,un^{-1/3}}(\pi)  : |u| \leq \delta, \pi \in \s{P}\} = n^{1/6} \s{G}_{x,\s{P},\delta n^{-1/3}}$, which has envelope $F_{x,n,\delta} = n^{1/6} G_{x,\s{P},\delta n^{-1/3}}$.  We first show that the process $\{ \d{G}_n n^{1/6} g_{x,u/n^{1/3}}(\pi) : |u| \leq \delta, \pi \in \s{P}\}$ is asymptotically $\bar\rho$-equicontinuous using Theorems 2.11.1 and 2.11.9 of VW, where $\bar\rho$ is the product semimetric.
We begin by assessing display (2.11.21) of VW. For the first line, we note that $P_0 F_{x,n,\delta}^2 = n^{1/3} P_0 G_{x,\s{P},\delta n^{-1/3}}^2 \leq c\delta$ for all $n$ large enough, so $P_0 F_{x,n,\delta}^2 = \boundeddet(1)$ as $n \to \infty$ for all fixed $\delta$. For the second line, we have, for any $\eta ,\epsilon > 0$, 
\begin{align*}
P_0 F_{x,n,\delta}^2 \{ F_{x,n,\delta} > \eta n^{1/2}\} &= n^{1/3} P_0 G_{x,\s{P},\delta n^{-1/3}}^2 \{ G_{x,\s{P},\delta n^{-1/3}} > \eta n^{1/3}\}\\
& = \delta (\delta n^{-1/3})^{-1}  P_0 G_{x,\s{P},\delta n^{-1/3}}^2 \{ G_{x,\s{P},\delta n^{-1/3}} > (\delta\eta) (\delta n^{-1/3})^{-1}\}\ ,
\end{align*}
which gives $P_0 F_{x,n,\delta}^2 \{ F_{x,n,\delta} > \eta n^{1/2}\} \leq \delta \epsilon'$ with $\epsilon' := \delta \eta$ for $n$ large enough. Next, we must show that 
\[ \sup\left\{ n^{1/3} P_0\left[g_{x,u n^{-1/3}}(\pi_1) - g_{x,vn^{-1/3}}(\pi_2)\right]^2 : |u - v| < \delta_n, \rho(\pi_1, \pi_2) < \delta_n\right\} \longrightarrow 0 \]
as $n \to \infty$ for all $\delta_n \downarrow 0$. We can bound the square root of $P_0\left[g_{x,un^{-1/3}}(\pi_1) - g_{x,vn^{-1/3}}(\pi_2)\right]^2$ by 
\[
\left\{ P_0\left[g_{x,un^{-1/3}}(\pi_1) - g_{x,vn^{-1/3}}(\pi_1)\right]^2 \right\}^{1/2}+\left\{ P_0\left[g_{x,vn^{-1/3}}(\pi_1) - g_{x,vn^{-1/3}}(\pi_2)\right]^2 \right\}^{1/2}.
\]
By assumption, for all $n$ large enough and up to a multiplicative constant, the first summand is bounded up by $(|u-v|n^{-1/3})^{1/2}$, and the second summand, by $\rho(\pi_1, \pi_2) (|v| n^{-1/3})^{1/2}$. Thus, we find that
\[n^{1/3} P_0\left[g_{x,un^{-1/3}}(\pi_1) - g_{x,vn^{-1/3}}(\pi_2)\right]^2 =\boundeddet\left(\left[|u-v| + |v| \rho(\pi_1, \pi_2)\right]^2\right),\]
uniformly over $u$, $v$, $\pi_1$ and $\pi_2$, which satisfies the requirement. Under (C1a), for any $\delta > 0$ and $n$ large enough, we have that
\begin{align*}
&\int_0^{t}  \left[\sup_Q \log N(\varepsilon \|F_{x,n,\delta}\|_{P_0,2}, \s{F}_{x,n,\delta}, L_2(Q))\right]^{1/2}d\varepsilon\\
&= \int_0^{t}  \left[\sup_Q \log N(\varepsilon n^{1/6} \|G_{x,\s{P}, \delta/n^{1/3}}\|_{P_0,2}, n^{1/6} \s{G}_{x,\s{P},\delta/n^{1/3}}, L_2(Q))\right]^{1/2}d\varepsilon \\
&= \int_0^{t} \sup_Q \left[ \log N(\varepsilon  \|G_{x,\s{P},\delta n^{-1/3}}\|_{P_0,2}, \s{G}_{x,\s{P},\delta n^{-1/3}}, L_2(Q))\right]^{1/2} d\varepsilon \ =\ \boundeddet\left(\int_{0}^{t} \varepsilon^V d\varepsilon\right)\ =\ \frac{t^{V+1}}{V+1}\ \rightarrow\ 0
\end{align*} as $t\rightarrow0$ since $V > -1$. An identical analysis holds under (C1b). We have thus verified the conditions of Theorems 2.11.1 or 2.11.9 of VW, and hence, $\{ \d{G}_n n^{1/6} g_{x,un^{-1/3}}(\pi) : |u| \leq \delta, \pi \in \s{P}\}$ is asymptotically $\bar\rho$-equicontinuous. Using (C4) and Lemma 4 (stated and proved in the Supplementary Material), we obtain the first statement of the theorem. For the second statement, we use  Theorem 2.14.1 and 2.14.2 of VW to obtain that
\[ E_0\left\{\sup_{|u| \leq \delta ,\pi^* \in \s{P}} \left|\d{G}_n n^{1/6}\left[g_{x,un^{-1/3}}(\pi^*) - g_{x,un^{-1/3}}(\pi)\right] \right|\right\} =\boundeddet\left(  \|F_{x,n,\delta} \|_{P_0,2}\right) =\boundeddet\left(\delta^{1/2}\right).\qedhere\]
\end{proof}

\begin{proof}[\bfseries Proof of Theorem~\ref{thm:hazard}]
We need to verify conditions (A1)--(A5) for the pair $(\Gamma_n,\Phi_n)$. Let $W_{n,x}^*$ be local process for this pair, which we can write pointwise as
\begin{align*}
W_{n,x}^*(u)\ &=\ r_n^2 \left\{ \left[S_n(x + u r^{-1}_n) - S_0(x + ur^{-1}_n)\right] - \left[S_n(x) - S_0(x)\right]\right\} \\
&\hspace{0.4in}-\theta_0(x) r_n^2 \left\{\left[\Phi_n(x + ur_n^{-1}) - \Phi_0(x + ur^{-1}_n)\right] - \left[\Phi_n(x) - \Phi_0(x)\right]\right\}  \\
&= \ W_{n,x}(u)-  \theta_0(x) r_n^2 \int_x^{x + ur_n^{-1}} \left[S_n(v) - S_0(v)\right]dv\ ,
\end{align*}
where $W_{n,x}$ is the local process for the pair $(\Gamma_n, \n{Id})$. We can rewrite the second term as
\begin{align*}
 \theta_0(x) r_n^{-1}\int_0^{u} W_{n,x}(v)dv - ur_n \theta_0(x)\left[S_n(x) - S_0(x)\right].
\end{align*}
Because for each $M > 0$ we have that $\{W_{n,x}(u) : |u| \leq M\}$ converges weakly in $\ell^{\infty}[-M, M]$ by (A1), so does $\left\{ \int_0^{u} W_{n,x}(v) dv : |u| \leq M\right\}$ by the continuous mapping theorem. The latter process is thus uniformly asymptotically negligible when multiplied by $r_n^{-1}$. The second term is also negligible since $S_n(x) - S_0(x) =\fasterthan(r_n^{-1})$. It follows then that $W_{n,x}^*$ and $W_{n,x}$ converge weakly to the same limit in $\ell^{\infty}[-M, M]$ and so, conditions (A1) and (A2) are automatically satisfied for $W_{n,x}^*$. The above expansion gives that $\sup_{|u| \leq \delta r_n } | W_{n,x}^*(u)|$ has mean bounded above by
\[
 E_0\left[\sup_{|u| \leq \delta r_n} | W_{n,x}(u)|\right] +\theta_0(x) r_n^{-1} E_0\left[\sup_{|u| \leq \delta r_n} \left| \int_0^u W_{n,x}(v) dv \right|\right] 
 + \delta r_n \theta_0(x)   E_0\left[r_n|S_n(x) - S_0(x)|\right],
 \] itself bounded by $f_n(r_n\delta) +\theta_0(x) \delta f_n(r_n\delta) + \theta_0(x) r_n\delta$
 since $\left| \int_0^u W_{n,x}(v) \, dv \right| \leq |u| \sup_{|v| \leq |u|} |W_{n,x}(v)|$. This expression satisfies (A3) since $\delta f_n(r_n\delta) \leq f_n(r_n\delta)$ for each $\delta \leq 1$. Condition (A4) is satisfied since \[E_0\left\{\sup_{|v|\leq \delta} \left| \int_0^{x+ v} \left[S_n(u) - S_0(u)\right] du \right|\right\} = \boundeddet\left(E_0\left[\sup_{u \leq x + \delta} |S_n(u) -S_0(u)|\right]\right) .\]This is similarly true for (A5).\qedhere

\end{proof}

\clearpage
\section*{Supplementary Material}

\newtheorem{innercustomlemma}{Lemma}
\newenvironment{customlemma}[1]
  {\renewcommand\theinnercustomlemma{#1}\innercustomlemma}
  {\endinnercustomlemma}

\vspace{.1in}

\subsection*{Proof of Lemmas}

\begin{proof}[\bfseries Proof of Lemma 1]
Since $\theta(x) = \psi(\Phi(x))$ and $\psi = -\partial_-\n{LCM}_{[a,b]}(-\Gamma\circ\Phi^-)$, where LCM is the least concave majorant operator, and $-\Gamma\circ\Phi^-$ is by assumption upper semi-continuous, the standard switch relation (e.g., Lemma 4.1 of \citealp{van2006estimating}, Lemma 3.2 of \citealp{groene2014shape}) implies that $\theta(x) > c$ if and only if $\sup \argmax_{u \in [a,b]} \left\{cu -\Gamma(\Phi^{-}(u)) \right\} <\Phi(x)$. We note that the set of maximizers is closed because $cu -\Gamma(\Phi^{-}(u))$ is upper semi-continuous.

If $c \neq 0$, the argmax can only contain elements in the range of $\Phi$, since on intervals where $\Phi^-$ is constant, the function can be made larger by taking $u$ to one end of the interval -- which end of the interval depends on the sign of $c$. We have used here the fact that $a$ and $b$ are by assumption in the range of $\Phi$. If $c = 0$, taking $\sup$ of the argmax ensures that the result will be at the right end of an interval. This shows that
\[  \sup \argmax_{u \in [a,b]} \left\{cu -\Gamma(\Phi^{-}(u)) \right\} < \Phi(x)\ \ \mathrm{iff}\ \ \sup \argmax_{u \in J^*} \left\{cu -\Gamma(\Phi^{-}(u)) \right\}< \Phi(x)\ , \]
where $J^* := [a,b] \cap \text{range}(\Phi)$. Let $\hat{u} =  \sup \argmax_{u \in J^*} \left\{cu -\Gamma(\Phi^{-}(u)) \right\}$. Because $\Phi^-$ is strictly increasing  on $\text{range}(\Phi)$ and hence $\Phi^- = \Phi^{-1}$  on $\text{range}(\Phi)$, and furthermore, $u \in \text{range}(\Phi)$ if and only if $\Phi(\Phi^-(u)) = u$, for every $u \in J^*$ there is a unique $v \in I^*$ such that $v = \Phi^-(u)$ and $\Phi(v) = u$. Let $\hat{v}\in I^*$ be such an element corresponding to $\hat{u}$. Then, we have that $\hat{u} < \Phi(x)$ if and only if $\Phi(\hat{v}) < \Phi(x)$, $c \Phi(\hat{v}) - \Gamma(\hat{v}) \geq c\Phi(v) - \Gamma(v)$ for all $v \in I^*$, and for any $v \in I^*$ such that equality holds, $v < \hat{v}$. Finally, $\Phi(\hat{v}) < \Phi(x)$ if and only if $\hat{v} < \Phi^-(\Phi(x))$ since $\hat{v} \in I^*$ and $\Phi$ is right-continuous and non-decreasing. It follows that $\theta(x)>c$ if and only if 
\[\sup \argmax_{v \in I^*} \left\{c\Phi(v) -\Gamma(v) \right\} < \Phi^-(\Phi(x))\ .\qedhere\]
\end{proof}

\vspace{.1in}
\begin{proof}[\bfseries Proof of Lemma 2]
First, note that $\omega^-(\eta) > 0$ for any $\eta > 0$ by right-continuity of $\omega$ at $\eta=0$. Thus, $\omega(v) < \eta$ for all $v < \omega^-(\eta)$, so that $\omega(\gamma(u)) < \eta$ for all $u < \gamma^{-1}(\omega^-(\eta))$. It is straightforward to see that $c(\delta / 2, \eta/2) /2 < \gamma^{-1}(\omega^-(\eta))$, which implies that $r(\delta , \eta) > 0$ for all $\delta , \eta >0$.

Let $\rho_\eta(d) := \int_0^{d} \left[ \eta -\omega(\gamma(u)) \right]du$. Since $\Phi_0$ is continuous and strictly increasing at $x$, we have that $\Phi_0^-(\Phi_0(x)) = x$. Recall that $\psi_0:= \theta_0 \circ \Phi_0^-$. Setting $\tilde\omega = \omega \circ \gamma$, the moduli of continuity of $\theta_0$ and of $\Phi_0^-$ at $x$ imply that $|\psi_0(u) - \psi_0(t)| \leq \tilde\omega(|u - t|)$ for all $u$ and $t = \Phi_0(x)$. Note that $\tilde\omega^- = \gamma^{-1} \circ \omega^-$.

First, suppose $x \in I_n$ so that $\Phi_n^- (\Phi_n(x)) = x$. Define the functions 
$R_{n,\eta,x}(u):=\Gamma_n(u) - \Gamma_0(u) - [\eta + \theta_0(x)][\Phi_n(u) - \Phi_0(u)]$ and
\[
h_{\eta, t}(u):=[\eta+\psi_0(t)]\Phi_0(u)- \Gamma_0(u)=\int_0^{\Phi_0(u)} [\eta + \psi_0(t) - \psi_0(v)]dv\ .
\]
By Lemma 1, we have that $\theta_n(x) - \theta_0(x) > \eta$ holds if and only if
\begin{gather*}
\sup \argmax_{u \in I_n} \left\{[\eta + \theta_0(x)]\Phi_n(u) -\Gamma_n(u) \right\} < x\ \ \textrm{iff}\ \  \sup \argmax_{u \in  I_n} \left\{h_{\eta, t}(u)-R_{n,\eta,x}(u) \right\} <  x \\
\textrm{iff}\ \ \sup_{u \in I_n: u \leq  x - \epsilon}\left\{ h_{\eta, t}(u) -R_{n,\eta,x}(u)  \right\}> \sup_{u \in I_n : x -\epsilon \leq u} \left\{h_{\eta, t}(u)-R_{n,\eta,x}(u) \right\}
\end{gather*}
for some $\epsilon > 0$.  Note that 
$\sup_{u \in I_n: u \leq x - \epsilon}\left\{ h_{\eta, t}(u) -R_{n,\eta,x}(u)  \right\} \leq h_{\eta, t}(x) + \sup_{u \in I_n, u < x}\left\{ -R_{n,\eta,x}(u)  \right\}$ since $h_{\eta, t}(u)$ is non-decreasing for $u \leq x$. Let $v_{n,\eta,t}^+ := \sup\{ v \in I_n : v \geq x, \tilde\omega(\Phi_0(v) - t) \leq \eta\}$. Then, we can write that $\sup_{u \in I_n : x -\epsilon \leq u } \left\{h_{\eta, t}(u)-R_{n,\eta,x}(u) \right\}\geq h_{\eta, t}(v_{n,\eta,t}^+ ) + \inf_{x \leq u}\{ -R_{n,\eta,x}(u)\}$.
 Hence, we have that $\theta_n(x) - \theta_0(x) > \eta$ implies that 
\begin{align*}
&h_{\eta, t}(x) + \sup_{u \leq x }\left\{ -R_{n,\eta,x}(u)  \right\}> h_{\eta, t}(v_{n,\eta,t}^+) + \inf_{x \leq u}\{ -R_{n,\eta,x}(u)\}\ \ \textrm{iff}\ \ X_{n,\eta}>  \int_{t}^{\Phi_0(v_{n,\eta,t}^+ )}[\eta+\psi_0(t) - \psi_0(u)] du
\end{align*} with the latter statement implying that $X_{n,\eta}> \int_{0}^{\Phi_0(v_{n, \eta, t}^+) - t}[\eta-\tilde\omega(u)]du=\rho_\eta(\Phi_0(v_{n, \eta, t}^+) - t)$, 
where we set $X_{n,\eta} := \sup_{u \in I_n, u < x} -R_{n,\eta,x}(u) + \sup_{u \in I_n, x \leq u}R_{n,\eta,x}(u)$.  An analogous argument for the opposite tail with $v_{n,\eta,t}^- := \inf\{ v \in I_n : v \leq x, \tilde\omega(t - \Phi_0(v)) \leq \eta\}$ shows that $\theta_n(x) - \theta_0(x) < -\eta$ implies that \[Y_{n,\eta} \geq \int_{\Phi_0(v_{n,\eta,t}^-)}^t [\eta + \psi_0(u) - \psi_0(t)]du\ ,\]which implies that $Y_{n,\eta} \geq \int_0^{t - \Phi_0(v_{n,\eta,t}^-) } [\eta -\tilde\omega(u)]du=\rho_\eta(t-\Phi_0(v^-_{n,\eta,t}))$, where we have set $Y_{n,\eta} := \sup_{u \in I_n: u \geq x}  -R_{n,-\eta,x}(u) +  \sup_{u \in I_n : u \leq x} R_{n,-\eta,x}(u)$. 

Now, we have that $\max( X_{n,\eta}, Y_{n,\eta})\leq 2 \| \Gamma_n - \Gamma_0 \|_{\infty, I_n} + 2(\eta +| \theta_0(x)|) \|\Phi_n- \Phi_0 \|_{\infty, I_n} =: Z_{n,\eta}$.
Let $d_{n,\eta}(t):=[\Phi_0(v_{n, \eta, t}^+) - t] \wedge [t - \Phi_0(v_{n,\eta,t}^-)]$.  Then, since $\eta \geq \tilde\omega(u)$ for $u \leq d_{n,\eta}(t)$, $d \mapsto \rho_\eta(d)$ is nondecreasing for $d \leq d_{n,\eta}(t)$, and hence, $
 \{ |\theta_n(x) - \theta_0(x) | > \eta \} \subseteq \{Z_{n,\eta} \geq \rho_\eta\left(d_{n,\eta}(t) \right)\}$.
Intuitively, since $\Phi_0$ is strictly increasing and continuous, if $\Phi_n$ is uniformly close to $\Phi_0$, then $\Phi_0(v_{n, \eta, t}^+) - t$ and $t - \Phi_0(v_{n,\eta, t}^-)$ should be close to $\tilde\omega^-(\eta)$ with high probability. Therefore, we use the law of total probability with the event $\{d_{n,\eta}(t) < c(\delta, \eta)/2\}$ to see that
 \begin{align*}
 \{Z_{n,\eta} \geq \rho_\eta\left(d_{n,\eta}(t) \right)\}\subseteq \left\{ Z_{n,\eta} \geq  r(\delta, \eta)  \right\} \cup \left\{ d_{n,\eta}(t) < c(\delta, \eta)/2\right\}.
\end{align*}

Now, $d_{n,\eta}(t) < c(\delta, \eta)/2$ implies that either $\Phi_0(v_{n,\eta, t}^+) - t < c(\delta, \eta)/2$ or $t - \Phi_0(v_{n,\eta, t}^-)<c(\delta, \eta)/2$. Suppose the former. Then, for all $v \in I_n$ such that $\Phi_0(v) \geq t +c(\delta, \eta)/2$, it must be true that $\tilde\omega(\Phi_0(v) - t) > \eta$ and hence $\Phi_0(v) - t \geq \tilde\omega^-(\eta)$. Thus, there is no $v \in I_n$ such that $\Phi_0(v) \in t+ [c(\delta, \eta)/2, \tilde\omega^-(\eta))$, which includes the interval $t+[c(\delta, \eta) / 2, c(\delta,\eta))$.  Note that $\Phi_0^-(t + \gamma^{-1}(\delta)) \leq x+ \delta$, which implies that $\Phi_0(x + \delta) \geq t + \gamma^{-1}(\delta) \geq t+ c(\delta,\eta)$, and thus $\Phi_0$ is strictly increasing and continuous on $[\Phi_0^{-1}(t + c(\delta, \eta) / 2), \Phi_0^{-1} ( t + c(\delta,\eta)) ]$. Hence, there is no $v \in I_n$ also contained in $[\Phi_0^{-1}(t + c(\delta, \eta) / 2), \Phi_0^{-1} ( t +  c(\delta, \eta)) ]$.
Suppose instead  that $t - \Phi_0(v_{n,\eta, t}^-) < c(\delta, \eta)/2$. Then, by similar reasoning, there is no $v \in I_n$ also contained in $[\Phi_0^{-1} ( t -  c(\delta, \eta)) , \Phi_0^{-1}(t -c(\delta, \eta) / 2)]$.
Since $\Phi_0(x-\delta) \geq 0$ and $\Phi_0(x + \delta) \leq u_n$ by assumption, this implies that $\Phi_n$ is constant on at least one of these intervals. Since $\Phi_0$ is strictly increasing and continuous on the intervals, we then have that the supremum distance between $\Phi_n$ and $\Phi_0$ on one of these intervals is at least $c(\delta, \eta)/4$. We have now shown that if $x \in I_n$, then
\[ \left\{ d_{n,\eta}(t) < c(\delta, \eta)/2\right\} \subseteq \left\{\|\Phi_n - \Phi_0\|_{\infty, [x-\delta, x+\delta]} \geq c(\delta, \eta)/4\right\}.\]

Now, if $x \notin I_n$, then since $\psi_n$ is  the left-derivative of $\bar{\Psi}_n$ and $\Phi_n$ is right-continuous, we have $\theta_n(x) = \theta_n( x_n)$ for $x_n := \Phi_n^-(\Phi_n(x))) < x$. Hence, we have that
\begin{align*}
&\{|\theta_n(x) - \theta_0(x) | > \eta\} \subseteq \{ |\theta_n( x_n) - \theta_0( x_n)| > \eta / 2\} \cup \{ \theta_0(x) - \theta_0(x_n) > \eta / 2\} \\
 &\subseteq  \{ |\theta_n( x_n) - \theta_0( x_n)| > \eta / 2, x - x_n < \delta / 2\} \cup \{| \theta_0(x_n) - \theta_0(x)| > \eta / 2, x - x_n < \delta/2\}\cup \{ x - x_n \geq \delta /2\} .
  \end{align*}
Because by assumption $\Phi_n(x) \in (0, u_n)$, and so, $x_n \in I_n$, we can use the above inclusion on the first event with $\delta$ replaced by $\delta / 2$. For the second term, we note that
\[ \{| \theta_0(x_n) - \theta_0(x)| > \eta / 2, x - x_n < \delta/2\} \subseteq \{ \omega^-(\eta/ 2) \leq x - x_n  < \delta/2\}.\]
 Hence, we have that 
\begin{align*}
 \{|\theta_n(x) - \theta_0(x) | > \eta\}& \subseteq  \{ |\theta_n( x_n) - \theta_0( x_n)| > \eta / 2, x -  x_n < \delta / 2 \} \cup \{x -  x_n \geq \omega^-(\eta/ 2)  \wedge \delta / 2\}.
  \end{align*}
The second event implies that $\Phi_n$ is constant on $[x - \omega^-( \eta/2) \wedge \delta / 2, x]$, and since $\Phi_0$ is strictly increasing and continuous there, it implies that
\[\|\Phi_n - \Phi_0\|_{\infty, [x-\delta, x+\delta]} \geq \left(t - \Phi_0(x - \omega^-( \eta/2) \wedge \delta / 2) \right) / 2 \geq \gamma^{-1}(\omega^-( \eta/2) \wedge \delta / 2) / 2 = c(\delta / 2, \eta / 2) / 2\ .\]
Therefore, we find that $\left\{ |\theta_n(x) - \theta_0(x) | > \eta \right\}$ is contained in 
\begin{align*}
&\left\{ Z_{n,\eta/2} \geq r(\delta / 2, \eta / 2)  \right\} \cup \left\{\|\Phi_n - \Phi_0\|_{\infty, [x-\delta, x+\delta]} \geq  c(\delta / 2, \eta / 2)/4\right\}\\
&\hspace{1.52in}\cup \left\{\|\Phi_n - \Phi_0\|_{\infty,  [x-\delta, x+\delta]} \geq c(\delta / 2, \eta / 2) / 2 \right\} \\
&=\ \left\{ Z_{n,\eta/2} \geq  r(\delta / 2, \eta / 2)  \right\} \cup \left\{\|\Phi_n - \Phi_0\|_{\infty,  [x-\delta, x+\delta]} \geq  c(\delta / 2, \eta / 2)/4\right\}.
\end{align*}
The pointwise inequality follows.

The uniform inequality follows from the pointwise inclusions. We note that $\sup_{x \in I_{n, \beta}} |\theta_n(x) - \theta_0(x) |> \eta$ implies there is an $x \in I_{n,\beta}$ such that $|\theta_n(x) - \theta_0(x) | > \eta$.  Thus, we have that $\{ \sup_{x \in I_{n, \beta}} |\theta_n(x) - \theta_0(x) |> \eta\}$ is contained in $\{ \exists x \in  I_{n, \beta}: |\theta_n(x) - \theta_0(x) | > \eta\}$, which can be decomposed as
\[\left\{\exists x \in  I_{n, \beta}: |\theta_n(x) - \theta_0(x) | > \eta, \Phi_n(x) \in (0, u_n)\right\} \cup \left\{\exists x \in  I_{n, \beta}: |\theta_n(x) - \theta_0(x) | > \eta, \Phi_n(x) \notin (0, u_n)\right\}.
\]
Since the moduli of continuity are assumed to hold for all $x$, and by construction, for every $x \in I_{n,\beta}$, $\Phi_0(x -\beta)$ and $\Phi_0(x +\beta)$ are in $J_n$, the pointwise inclusion can be applied to the first event with $\delta = \beta$. For the second event, note that $\Phi_0(x -\beta),\Phi_0(x +\beta)\in J_n$ and $\Phi_n(x) \notin (0, u_n)$ imply that $|\Phi_n(x) - \Phi_0(x)| \geq \gamma^{-1}(\beta)$ and 
\begin{align*}
&\left\{\|\Phi_n - \Phi_0\|_{\infty, I} \geq c(\beta/2, \eta/2)/4)\right\} \cup \left\{ \|\Phi_n - \Phi_0\|_{\infty,  [x-\delta, x+\delta]} > \gamma^{-1}(\beta)\right\} \\
&=\ \left\{\|\Phi_n - \Phi_0\|_{\infty,  [x-\delta, x+\delta]} \geq  c(\beta/2, \eta/2)/4\right\}. \qedhere
\end{align*}\end{proof}

\vspace{.1in}
\begin{customlemma}{3}
Under the conditions of Theorem 3, $\sup \argmax_{v \in c_n(I_n - x)} \{  H_{n, x, \eta}(v) + R_{n}(v)\} = \bounded(1).$
\end{customlemma}
\begin{proof}[\bfseries Proof of Lemma 3]
To establish that $\hat{v}_{n} = \bounded(1)$, we use Theorem 3.2.5 of VW. Write $\displaystyle c_n^{-1} \hat{v}_n(x,\eta) = \sup\argmax_{v \in I_n - x} M_{n,x}(v)$ for 
\[M_{n,x}(v) := -[\Gamma_n(x + v) - \Gamma_n(x)] + \theta_0(x) [\Phi_n(x +v) -\Phi_n(x)]+ \eta c_n^{-\alpha} \Phi_n(x + v)\ .\]
Defining $M_{0,x}(v) := [\Gamma_0(x + v) - \theta_0(x) \Phi_0(x + v)] - [\Gamma_0(x) - \theta_0(x) \Phi_0(x)]$, we have that $M_{0,x}'(v) = \Phi_0'(x+v)[\theta_0(x + v) - \theta_0(x)] > c |v|^{\alpha}$ for $v$ in a neighborhood of $0$ and some $c > 0$, which implies that $-M_{0,x}(v) \leq -c'|v|^{\alpha + 1}$ for $v$ in a neighborhood of $0$ and some $c' > 0$. In the notation of Theorem 3.2.5 of VW, we then have $d(v, 0) := |v|^{\frac{\alpha + 1}{2}}$. The next requirement concerns the modulus of continuity of $M_{n,x}(v) - M_{0,x}(v)$, which we can write as
\begin{align*}
&E_0\left[ \sup_{|v| < \delta^{2/(\alpha + 1)}}\left| (M_{n,x}- M_{0,x})(v) - (M_{n,x} - M_{0,x})(0) \right |\right]\\
&\hspace{.5in}=\ E_0\left[ \sup_{|v| < \delta^{2/(\alpha + 1)}} \left| \eta c_n^{-\alpha}  [\Phi_n(x + v) - \Phi_n(x)]-c_n^{-(\alpha + 1)}W_{n,x}(c_n v) \right |\right] \\
&\hspace{.5in}\leq\  c_n^{-(\alpha + 1)} E_0\left[ \sup_{|u| < c_n\delta^{2/(\alpha + 1)} } \left| W_{n,x}(u)\right|\right] +  |\eta| c_n^{-\alpha} E_0\left[ \sup_{|v| \leq \delta^{2/(\alpha + 1)}}  \left| \Phi_n(x + v) - \Phi_n(x) \right |\right].
\end{align*}
By assumption, the first term is bounded by $c_n^{-(\alpha + 1)} f_n\left(c_n\delta^{2/(\alpha + 1)}\right)$. Taking differences with $\Phi_0$, and since $\Phi_0$ is continuously differentiable at $x$, we can find $\delta$ small enough such that the second term is bounded up to a constant by $c_n^{-\alpha}| \eta| \left(c_n^{-1}+ \delta^{2/(\alpha + 1)}\right)$ for all $n$ large enough. We thus have that, for all $n$ large enough, the above expression is bounded up to a constant by
\[ \tilde{f}_n(\delta) := c_n^{-(\alpha + 1)} f_n\left(c_n\delta^{2/(\alpha + 1)}\right)+ c_n^{-\alpha}| \eta| \left(c_n^{-1}+ \delta^{2/(\alpha + 1)}\right)\ .\]
By assumption, $\delta\mapsto \delta^{-\beta}f_n(c_n\delta)$ is decreasing for some $\beta \in (1,1+\alpha)$, which implies that $\delta \mapsto \delta^{-2\beta / (\alpha + 1)} \tilde{f}_n(\delta)$ is decreasing, where $2\beta / (\alpha + 1) \in (0, 2)$ as required by VW. Additionally, 
\[c_n^{\alpha+1} \tilde{f}_n\left(c_n^{-(\alpha + 1)/2}\right) = f_n(1) +2 |\eta|  = \boundeddet(1)\ . \] If we can establish that $c_n^{-1}\hat{v}_{n}(x,\eta) = \fasterthan(1)$, we will have checked all the conditions of Theorem 3.2.5 of VW, yielding 
\[ | \hat{v}_n(x,\eta)|^{\frac{\alpha + 1}{2}} =  c_n^{(\alpha + 1)/2} d\left(c_n^{-1} \hat{v}_n(x,\eta), 0\right) = \bounded(1)\ ,\] 
and hence $\hat{v}_n(x, \eta) = \bounded(1)$.

 Simplifying further, we have that $c_n^{-1} \hat{v}_n(x,\eta) = -x + \sup\argmax_{v \in I_n} \tilde{M}_{n,x}(v)$, where $\tilde{M}_{n,x}(v):= -\Gamma_n(v) + \Phi_n(v) [ \theta_0(x) + c_n^{-\alpha}\eta] $. Setting 
 \begin{align*}
 h_x(v)\ &:=\ \psi_0(t)\Phi_0(v) - \Gamma_0(v)\ =\ \int_0^{\Phi_0(v)}[\psi_0(t) - \psi_0(u)]du \\
 R_{n,x}(v)\ &:=\  \eta c_n^{-\alpha}  \Phi_n(v) + \theta_0(x) [\Phi_n(v) - \Phi_0(v)] - [\Gamma_n(v) - \Gamma_0(v)]\ ,
 \end{align*} write $\tilde{M}_{n,x}(v) = h_x(v) + R_{n,x}(v)$. Note that $\sup_{v \in I_n} |R_{n,x}(v)| = \fasterthan(1)$, and that $h_x$ is unimodal and maximized at $v = x$, but $x$ may not be in $I_n$ for any $n$. Define $x_n^+ := \inf\{ v \in I_n : v \geq x\}$ and let $\epsilon > 0$. Then, $ \sup\argmax_{v \in I_n} \tilde{M}_{n,x}(v) \leq x - \epsilon$ implies that 
$h_x(x -\epsilon) + \sup_{v \in I_n : v < x} R_{n,x}(v) > h_x(x_n^+) + \inf_{v \in I_n: v \geq x} R_{n,x}(v)$, which in turn implies that
\[2\sup_{v \in I_n} |R_{n,x}(v)|+\int_{t}^{\Phi_0(x_n^+)} [\psi_0(v) - \psi_0(t)]dv\ >\ \int_{\Phi_0(x - \epsilon)}^{t} [\psi_0(t) - \psi_0(v)]dv\ .
\]
Since $\Phi_0$ and $\psi_0$ are differentiable with positive derivative at $x$ and $t$, respectively, for all $\epsilon > 0$, $\int_{\Phi_0(x - \epsilon)}^{t} [\psi_0(t) - \psi_0(v)]dv =: \delta_{\epsilon} > 0$. Additionally, by the boundedness of $\psi_0$, \[\int_{t}^{\Phi_0(x_n^+)} [\psi_0(v) - \psi_0(t)]dv\ \leq\  c\left[\Phi_0(x_n^+) - \Phi_0(x)\right]\] for some $c < \infty$. We claim that $x_n^+ \inprob x$. To see this, first note that $x_n^+ > x + \epsilon'$ implies that $\Phi_n(x) = \Phi_n(x + \epsilon')$. Hence, for all $0 \leq u \leq \delta \wedge \epsilon'$, we have that 
\[ [\Phi_n(x) - \Phi_0(x)] - [\Phi_n(x+u) -\Phi_0(x + u)] = \Phi_0(x + u) - \Phi_0(x) \geq c' u\]
for some $c' > 0$, again using that $\Phi_0$ is differentiable with positive derivative at $x$. This implies that 
$0 < (\delta \wedge \epsilon)c' \leq  2\sup_{|u| < \delta} |\Phi_n(x+u) - \Phi_0(x+u)|$, 
the probability of which goes to zero for any $\epsilon > 0$. Hence, $x_n^+ \inprob x$, and so, $\Phi_0(x_n^+) \inprob \Phi_0(x)$ by the Continuous Mapping Theorem. We have shown that
\[ P_0\left( \sup\argmax_{v \in I_n} \tilde{M}_{n,x}(v) -x  \leq - \epsilon \right) \leq P_0(\fasterthan(1) > \delta_\epsilon)\ ,\]
which goes to $0$ for each $\epsilon > 0$. The argument for the opposite tail probability is completely analogous, and hence $ \sup\argmax_{v \in I_n} \tilde{M}_{n,x}(v) \inprob x$.\qedhere

\end{proof}

\vspace{.1in}
\begin{customlemma}{4} Let $\{ V_n(u, f) : u \in \s{U}, f\in\s{F}\}$ be a sequence of stochastic processes indexed by $\s{U} \times \s{F}$, where $(\s{U}, d_1)$ and $(\s{F}, d_2)$ are semi-metric spaces. Let $\rho$ be the corresponding product semi-metric on $\s{U} \times \s{F}$. Suppose $V_n$ are asymptotically uniformly $\rho$-equicontinuous in the sense of VW and $d_2(f_n, f_0)$ tends to zero in probability. Then, $\sup_{u \in \s{U}} |V_n(u, f_n) - V_n(u, f_0)|$ tends to zero in probability.\end{customlemma}
\begin{proof}[\bfseries Proof of Lemma 4]
The result follows immediately upon noting that
\[
\left\{ \sup_{u \in \s{U}} | V_n(u, f_n) - V_n(u, f_0)| > \epsilon \right\} 
\subseteq \left\{ \sup_{\rho((u,f), (v, g)) < \delta} | V_n(u, f) - V_n(v, g)| > \epsilon \right\} \cup \left\{d_2(f_n, f_0) \geq \delta \right\}. \qedhere
\]
\end{proof}
 
 \vspace{.1in}
 
%

 \subsection*{Heuristic justification for the form of $D_{x,0}^*$ assumed in Theorem 5}
 
Denote by $L_2^0(P_0)$ the set of all functions from $\mathscr{O}$ to $\mathbb{R}$ with $P_0$-mean zero and finite $P_0$-variance. In a nonparametric model, the efficient influence function $M_{x,0}^*$ of $\Theta_0(x)$ is the unique element of $L_2^2(P_0)$  such that, for each $s_0\in L_2^0(P_0)$,
\[\left.\frac{\partial}{\partial \varepsilon} \Theta_{P_{\varepsilon}}(x) \right|_{\varepsilon=0} = \int_0^{x} \left.\frac{\partial}{\partial \varepsilon}\theta_{P_{\varepsilon}}(u) \right|_{\varepsilon=0} du = P_0(M_{x,0}^* s_0)\]
for each regular, one-dimensional parametric path $\{P_{\varepsilon}\}\subset \mathscr{M}$ through $P_0$ at $\varepsilon=0$ and with score $s_0$ at $\varepsilon=0$. In the presence of a transformation depending on $P_0$, the efficient influence function $L_{x,0}^*$ of $\Phi_0(x)$ similarly satisfies $\left.\frac{\partial}{\partial \varepsilon} \Phi_{P_{\varepsilon}}(x)\right|_{\varepsilon=0}=P_0(L_{x,0}^* s_0)$. The nonparametric efficient influence function $D_{x,0}^*$ of $\Gamma_0(x)$ thus satisfies
\[P_0(D_{x,0}^* s_0) = \left.\frac{\partial}{\partial \varepsilon} \Gamma_{P_{\varepsilon}}(x) \right|_{\varepsilon=0} = \int_0^x \left.\frac{\partial}{\partial \varepsilon}\theta_{P_{\varepsilon}}(v) \right|_{\varepsilon=0} \Phi_0(dv) + \int_0^x \theta_0(v) \left.\frac{\partial}{\partial \varepsilon}\Phi_{P_{\varepsilon}}(dv)\right|_{\varepsilon=0}  .\]
The first term typically contributes $o\mapsto I(u \leq x) M_{x,0}^{(1)}(o)\Phi_0'(u) + D_{x, 0}^{(2)}(o)$ to the form of $D^*_{x,0}$. $M_{x,0}^{(1)}$ in this term is deliberately the same as in the influence function of $\Theta_0(x)$.  The second term is equal to \[\int_0^x \theta_0(v)P_0\left(L^*_{dv,0}s_0\right)=P_0 \left[ s_0\int_0^x \theta_0(v) L_{dv,0}^*\right],\] so that this term contributes  $o\mapsto \int_0^x \theta_0(v) L_{dv,0}^*(o)$ to the form of $D^*_{x,0}$.
Hence, we get the general form
\[D_{x,0}^*(o) = I(u \leq x) D_{s,0}^{(1)}(o)\Phi_0'(u) + D_{x, 0}^{(2)}(o) + \int_0^x \theta_0(v) L_{dv,0}^*(o)\ .\qedhere\]

\subsection*{Applications of the general theory: additional details}


{\bfseries \textit{Monotone density function with independent censoring.}} We start by analyzing the estimator of a monotone density function with independent censoring. The local difference $g_{x,u}$ can be written as 
\begin{align*}(y,\delta)\mapsto 
 \frac{-[S_0(x + u) - S_0(x)] \delta I_{[0,x+u]}(y)}{S_0(y) G_0(y)}-\frac{S_0(x)\delta I_{(x,x+u]}(y)}{S_0(y) G_0(y)}+ \int_{v<y} \frac{I_{(x,x+u]}(v)}{S_0(v)G_0(v)} \Lambda_0(dv)\ .
 \end{align*}
The class of functions $\s{G}_{x,R}$ is a Lipschitz transformation of the classes $\{u\mapsto S_0(x + u) - S_0(x) : |u| \leq R\}$ and  $\{ t\mapsto I_{(x,x+u]}(t): |u| \leq R\}$, and hence satisfies (B1). An envelope function for $\s{G}_{x,R}$ is given by
\[G_{x,R}: (y,\delta)\mapsto \frac{\delta I_{[0,R]}(|y - x|)}{S_0(y) G_0(y)}(1 + KR) + \int_{x - R}^{x + R} \frac{I_{[0,y)}(v)}{S_0(v)G_0(v)}\Lambda_0(dv)\ .\]
It is easy to see that (B2) is satisfied if $S_0$ and $G_0$ are positive in a neighborhood of $x$. The covariance function is given by $\Sigma_0:(s,t)\mapsto\int_0^{s\wedge t} \frac{S_0(s)S_0(t)}{S_0(u)G_0(u) }\Lambda_0(du)$. Display (3) of the main text is thus satisfied with $A_0(s, t, v,w) = [S_0(t) S_0(s)] / [G_0(v) S_0(v)]$ and $H_0(v,w) = \Lambda_0(v)$. Condition (B3) is satisfied if $\theta_0$ is positive and continuous in a neighborhood of $x$. We then get $\kappa_0(x) = [S_0(x) / G_0(x)] \lambda_0(x) = f_0(x) / G_0(x)$, so that the scale parameter is $\tau_0(x) = [4 f_0'(x) f_0(x) / G_0(x)]^{1/3}$. This agrees with the results of \cite{huang1995right}.

It remains to scrutinize the conditions arising from the remainder term $H_{x,n}$. If $G_n$ is the Kaplan-Meier estimator of $G_0$, it is always true that $\d{P}_n D_{n,x}^* = 0$, where $D_{n,x}^*$ is the estimator of $D_{0,x}^*$ obtained by replacing $S_0$ and $G_0$ by $S_n$ and $G_n$, respectively. It is easy to verify that $H_{x,n}$ can be decomposed as $H^{(1)}_{x,n}+H^{(2)}_{x,n}$, where $H^{(1)}_{x,n}:=(\d{P}_n-P_0)(D^*_{n,x}-D^*_{0,x})$ is the usual empirical process term and \[H^{(2)}_{x,n}:=S_n(x) \int_0^x \frac{S_0(u)}{S_n(u)} \left[ \frac{ G_0(u)}{G_n (u)} - 1 \right] (\Lambda_n - \Lambda_0)(du)\] is the second-order remainder term. The local remainder emanating from $H^{(1)}_{x,n}$ can be studied using results from Section~4.3 of the main text. We instead focus on the local remainder $K^{(2)}_{n}(\delta) := n^{2/3} \sup_{|u| \leq \delta n^{-1/3}} | H_{n, x+u}^{(2)} - H_{n, x}^{(2)}|$, which can be bounded as
\begin{align*}
K^{(2)}_{n}(\delta)\ \leq&\ \ n^{2/3} \int_{0}^{x + \delta n^{-1/3}}\frac{S_0(u)}{S_n(u)} \left| \frac{ G_0(u)}{G_n (u)} - 1 \right| |(\Lambda_n - \Lambda_0)(du)|\sup_{|u| \leq \delta n^{-1/3}} | S_n(x +u) - S_n(x)|  \\
 &\ + n^{2/3}\int_{x-\delta n^{-1/3}}^{x + \delta n^{-1/3}} \frac{S_0(u)}{S_n(u)} \left| \frac{ G_0(u)}{G_n (u)} - 1 \right| |(\Lambda_n - \Lambda_0)(du)|\ .
 \end{align*}
Writing $| S_n(x +u) - S_n(x)| = n^{-1/2} \{n^{1/2}[S_n(x + u) - S_0(x+u)] -n^{1/2}[S_n(x ) - S_0(x)]\} + [S_0(x + u) - S_0(x)]$, we note that $ \sup_{|u| \leq \delta n^{-1/3}} | S_n(x +u) - S_n(x)| = \fasterthan(n^{-1/2}) + \bounded(n^{-1/3})$ in view of the weak convergence of $n^{1/2}(S_n - S_0)$ in a neighborhood of $x$. Using that $\int_a^b |f(u)| |g(du)| \leq \sup_{u \in [a,b]} |f(u)| \| g\|_{TV, [a,b]}$ with $\| \cdot\|_{TV, [a,b]}$ denoting the total variation norm over $[a,b]$, and that $\|\Lambda_n - \Lambda_0\|_{TV, [a,b]} \leq \|\Lambda_n\|_{TV, [a,b]} + \|\Lambda_0\|_{TV,[a,b]}  = [\Lambda_n(b ) - \Lambda_n(a)] + [\Lambda_0(b) - \Lambda_0(a)]$ in view of the monotonicity of $\Lambda_n$ and $\Lambda_0$, we find that
\begin{align*}
K^{(2)}_{n}(\delta)\ =\ \left[\fasterthan(n^{-1/6}) + \bounded(1)\right] \bounded(n^{-1/6})\ =\ \bounded(n^{-1/6})\ ,
\end{align*}
which is sufficient to establish conditions (B4) and (B5).

\vspace*{1em}

\noindent {\bfseries \textit{Monotone density function with conditionally independent censoring.}} We now turn to analysis of the proposed estimator of a monotone density function with conditionally independent censoring. Conditions (B1) and (B2) are satisfied under Lipschitz conditions on $S_0$ and $G_0$ uniformly over $w$. The asymptotic covariance function is given by
\[\Sigma_0(s,t) = \int S_0(t \mid w) S_0(s \mid w)Q_0(dw) - S_0(t) S_0(s) + \iint_0^{s\wedge t} \frac{S_0(t \mid w) S_0(s \mid w)}{G_0(y\mid w) S_0(y \mid w)} \Lambda_0(dy \mid w) Q_0(dw)\ ,\] and so, we find that display (3) holds with $\Sigma_0^*(s,t)=\int S_0(t\mid w)S_0(s\mid w)Q_0(dw)-S_0(t)S_0(s)$, $A_0(s,t, v, w) = [S_0(t \mid w) S_0(s \mid w)]/[G_0(v\mid w) S_0(v \mid w)]$ and $H_0(v, w) = \Lambda_0(v \mid w)$. Thus, condition (B3) holds, and we get $\kappa_0(x) = \int [f_0(x \mid w)/G_0(x \mid w)]Q_0(dw)$, where $f_0(x \mid w)$ is the conditional density of $T$ at $x$ given $W = w$.

The remainder term $H_{x,n}$ again has the form $H^{(1)}_{x,n}+H^{(2)}_{x,n}$ with $H^{(1)}_{x,n}:=(\d{P}_n-P_0)(D^*_{n,x}-D^*_{0,x})$ and \[H^{(2)}_{x,n} := \int S_n(x \mid w) \int_0^x \frac{S_0(y \mid w)}{S_n(y \mid w)} \left[ \frac{G_0(y \mid w)}{G_n(y\mid w)} - 1\right] (\Lambda_n - \Lambda_0)(dy \mid w) Q_0(dw)\ .
\] Once more, we focus on $H^{(2)}_{x,n}$. 
 Writing $S_n^{(0)} := S_n - S_0$, if $S_n$ and $G_n$ are bounded away from zero in a neighborhood of $x$ with probability tending to one, and if $S_0$ is Lipschitz in $x$ uniformly in $w$, then the term $K_{n}^{(2)}(\delta) := n^{2/3}\sup_{|u| \leq \delta n^{-1/3}}| H^{(2)}_{x + u,n} - H^{(2)}_{x,n}|$ is bounded by a constant multiple of 
 \[n^{2/3}\left\{ \left[ E_{0}\sup_{|u-x| \leq \epsilon}|S_n^{(0)}(u \mid W) -S_n^{(0)}(x \mid W) |^2 \right]^{1/2} + \delta n^{-1/3} \right\}\left[E_{0}\sup_{u \leq x + \epsilon}| G_n(u \mid W) - G_0(u \mid W)|^2 \right]^{1/2}\] with probability tending to one.
Control of $K_{n}^{(2)}(\delta) $ is highly dependent on the behavior of  $S_n$ and $G_n$. If, for instance,  $S_n - S_0$ and $G_n - G_0$ uniformly tend to zero in probability at rates faster than $n^{-1/3}$, then conditions (B4) and (B5) are satisfied. For estimators of the form $S_n(x \mid w) =  \exp\left[-\int_0^x  \lambda_n(u \mid w) du\right]$ with $\lambda_n$ an estimator of the conditional hazard $\lambda_0$, we find that $K_{n}^{(2)}(\delta) $ is bounded by a constant multiple of 
 \[ \delta n^{1/3} \left[E_{0} \sup_{u \leq x+ \epsilon} |\lambda_n(u \mid W) - \lambda_0(u \mid W)|^2E_{0} \sup_{u \leq x+ \epsilon} |G_n(u \mid W) - G_0(u \mid W)|^2\right]^{1/2}\] with probability tending to one, and so, we require that the product of the convergence rates of $\lambda_n-\lambda_0$ and $G_n-G_0$ to be faster than $n^{-1/3}$.
 
 \vspace*{1em}
 
\noindent {\bfseries \textit{Monotone regression function with no confounding.}} We now analyze the asymptotic distribution of the isotonic regression estimator. We find the localized difference function to be $g_{x,u}: (a,y)\mapsto [y - \theta_0(x)] I_{(x,x+u]}(a) - [\Gamma_0(x+u) - \Gamma_0(x)] + \theta_0(x) [\Phi_0(x+u) - \Phi_0(x)]$. The second and third summands are constant as functions of $(a,y)$ and Lipschitz in $u$ with a constant envelope function. Hence, they easily satisfy conditions (B1) and (B2). The first summand is the fixed function $(a,y)\mapsto [y - \theta_0(x)]$ multiplied by an element of the class $\{v\mapsto I_{(x,x+u]}(v):u>0\}$. This class has been studied for the Grenander estimator of a monotone density function; in particular, it is known to possess polynomial covering numbers. The natural envelope for the class generated by the first summand is thus $(a,y)\mapsto |y - \theta_0(x)| I_{[0,R]}(|a - x|)$, which satisfies (B1) and (B2) if, in a neighborhood of $x$, the conditional variance function, defined pointwise as $\sigma_0^2(t) := \n{Var}_0(Y \mid A = t)$, is bounded  and $\Phi_0$ possesses a positive, continuous density. In such cases, Theorem 4 holds. Through straightforward calculations, we find that 
\[ \Sigma_0(s,t) = -\left[\Gamma_0(s) - \theta_0(x) \Phi_0(s)\right]\left[\Gamma_0(t) - \theta_0(x) \Phi_0(t)\right] + E_{0} \left\{1_{(-\infty,s\wedge t]}(A) \left[Y -\theta_0(x)\right]^2\right\}\ . \]
The first summand is continuously differentiable at $(x,x)$ since each of $\theta_0$ and $\Phi_0$ are continuously differentiable at $x$. The second summand can be expressed as $\int_{-\infty}^{s\wedge t} \{\sigma_0^2(u) + \left[\theta_0(u) - \theta_0(x)\right]^2\} \Phi_0(du)$. We thus confirm that display (3) holds with $A_0(s,t,v,w) = \sigma_0^2(v) + [\theta_0(v) - \theta_0(x)]^2$ and $H_0 = \Phi_0$.  Provided $\sigma_0^2$ is continuous at $x$ and $\Phi_0$ is continuously differentiable at $x$, condition (B3) holds. As such, we obtain that $n^{1/3}\left[\theta_n(x) - \theta_0(x)\right]$ has a scaled Chernoff distribution with scale parameter
\[ \tau_0(x)=\left[\frac{4 \mu_0'(x) \sigma_0^2(x)}{f_0(x)} \right]^{1/3}\] coinciding with the classical results of \cite{brunk1970regression}.
 
 \vspace*{1em}
 
\noindent {\bfseries \textit{Monotone regression function with confounding by recorded covariates.}} Finally, we turn to an analysis of the proposed estimator of a monotone covariate-adjusted dose-response function.  For condition (B4), we focus as before on the second-order remainder term $H_{x,n}^{(2)}$ given by 
\[ \iint_{-\infty}^x \left[\mu_n(u, w) - \mu_0(u, w)\right] \left[ \frac{g_n(u, w)}{g_0(u,w)} - 1\right] \Phi_0(du)Q_0(dw) - \iint_{-\infty}^x \mu_n(u, w) \, (\Phi_n - \Phi_0)(du) (Q_n - Q_0)(dw)\ .\]
 The contribution to $K_n(\delta)$ of the first summand above is bounded above by
\begin{align*}
2\delta n^{1/3} \sup_{|x - u| \leq \delta n^{-1/3}} \left[f_0(u) \left\{ E_{0}\left[ \mu_n(u, W) - \mu_0(u, W)\right]^2  E_{0}\left[ \frac{g_0(u, W)}{g_n(u,W)} - 1\right]^2 \right\}^{1/2}\right].
\end{align*}
which implies that condition (B4) is satisfied if, for some $\epsilon > 0$,
\[ \sup_{|x - u| \leq \epsilon}E_{0}\left[ \mu_n(u, W) - \mu_0(u, W)\right]^2  \sup_{|x - u| \leq \epsilon} E_{0}\left[ \frac{g_0(u, W)}{g_n(u,W)} - 1\right]^2 = \fasterthan\left(n^{-1/3}\right).\] The contribution of the second summand to $K_n(\delta)$ can easily be controlled using empirical process theory. To scrutinize condition (B5), the relevant portion of the covariance function $\Sigma_0(s,t)$ is given by  
\[ \iint_{0}^{s \wedge t} \left\{ \frac{\sigma_0^2(a,w)}{g_0(a,w)} + [\theta_0(a) - \theta_0(x)]^2 \right\} \Phi_0(da) Q_0(dw)\ ,\]
where $\sigma_0^2:(a,w)\mapsto \n{Var}_0(Y \mid A = a, W =w)$ denotes the conditional variance function of $Y$ given $A$ and $W$. Under certain smoothness conditions, we have that $\kappa_0(x) = f_0(x)^2\int \left[ \sigma_0^2(x,w)/f_0(x \mid w)\right]Q_0(dw)$, from which we find that the scale parameter of the limit Chernoff distribution to be 
\[ \tau_0(x) = \left\{ 4 \nu_0'(x)\int\left[ \frac{\sigma_0^2(x, W)}{f_0(x \mid W)}\right]Q_0(dw) \right\}^{1/3}.\]

\end{document}